\documentclass[leqno,11pt]{amsart}
\usepackage{amssymb, amsmath,amsmath,latexsym,amssymb,amsfonts,amsbsy, amsthm,mathtools,graphicx,CJKutf8,CJKnumb,CJKulem,color}
\usepackage[english]{babel}
\usepackage{breqn}
\usepackage{tikz}
\usepackage{bigints}
\usepackage{yfonts}
\usepackage{comment}
\usetikzlibrary{calc,hobby}
\usepackage{bbm}
\usepackage{mathtools}
\mathtoolsset{showonlyrefs}
\usepackage[shortlabels]{enumitem}
\usepackage{graphicx}
\usepackage{theoremref}
\newtheorem{theorem}{Theorem}[section]
\numberwithin{equation}{section}

\newtheorem{remark}[theorem]{Remark}

\newtheorem{corollary}[theorem]{Corollary}
\newtheorem{lemma}[theorem]{Lemma}

\usepackage{amssymb}
\usepackage{blindtext}
\usepackage{xcolor}
\usepackage{mathrsfs}
\usepackage{ulem}
\usepackage{stmaryrd}
\usepackage{hyperref}
\hypersetup{
    colorlinks=true,
    linkcolor=blue,
    filecolor=magenta,      
    urlcolor=cyan,
}
\newcommand\sometext

\newcommand{\tb}{\textbf}

\newcommand{\loc}{\mathrm{loc}}

\newcommand{\subscript}[2]{$#1 _ #2$}

\newcommand{\norm}[1]{\left\lVert#1\right\rVert}

    \title[transition to instability] {The transition to instability for stable shear flows in inviscid fluids}
\author{Daniel Sinambela}
\address[D. Sinambela]{Department of Mathematics, New York University Abu Dhabi, Saadiyat Island, P.O. Box 129188, Abu Dhabi, United Arab Emirates.}
\email{dos2346@nyu.edu}
\author{Weiren Zhao}
\address[W. Zhao]{Department of Mathematics, New York University Abu Dhabi, Saadiyat Island, P.O. Box 129188, Abu Dhabi, United Arab Emirates.}
\email{zjzjzwr@126.com, wz19@nyu.edu}

\begin{document}
\begin{abstract}
    In this paper, we study the generation of eigenvalues of a stable monotonic shear flow under perturbations in $C^s$ with $s<2$. More precisely, we study the Rayleigh operator $\mathcal{L}_{U_{m,\gamma}}= U_{m,\gamma}\partial_x-U''_{m,\gamma}\partial_x\Delta^{-1}$ associated with perturbed shear flow $(U_{m,\gamma}(y),0)$ in a finite channel $\mathbb{T}_{2\pi}\times [-1,1]$ where $U_{m,\gamma}(y)=U(y)+m\gamma^2\widetilde\Gamma(y/\gamma)$ with $U(y)$ being a stable monotonic shear flow and $\big\{m\gamma^2\widetilde\Gamma(y/\gamma)\big\}_{m\geq 0}$ being a family of perturbations parameterized by $m$. We prove that there exists $m_*$ such that for $0\leq m<m_*$, the Rayleigh operator has no eigenvalue or embedded eigenvalue, therefore linear inviscid damping holds. Otherwise, instability occurs when $m\geq m_*$. Moreover, at the nonlinear level, we show that asymptotic instability holds for $m$ near $m_*$ and growing modes exist for $m>m_*$ which equivalently leads to instability.
\end{abstract}
\maketitle

\section{Introduction}\label{Introduction}
\allowdisplaybreaks
We consider the two dimensional incompressible Euler equations in a finite $x$-periodic channel $\{(x,y): x\in \mathbb{T}_{2\pi},  y \in [-1,1]\},$
\begin{equation}\label{original euler}
 \begin{cases}
    \partial_t\tb{u} +\tb{u} \cdot \nabla \tb{u}+\nabla P=0 ,\\
   \nabla \cdot \tb{u}=0 ,\\
    v(t,x,-1)=0=v(t,x,1) ,
\end{cases}  
\end{equation}
 where $\tb{u}:=(u(t,x,y),v(t,x,y))$. Using the momentum (both horizontal and vertical) equations and incompressibility condition, one can recast \eqref{original euler} in terms of vorticity $\mathfrak{w}= \partial_xv(t,x,y)-\partial_y u(t,x,y)$ as follows
 \begin{equation}\label{Euler Equations}
     \partial_t \mathfrak{w} + \tb{u} \cdot \nabla \mathfrak{w}=0.
 \end{equation}
 A natural first step in our overall analysis is the linearization of \eqref{Euler Equations} around a shear flow $(U(y),0)$. This process yields the linearized problem
\begin{equation}\label{Linearized Euler Equations}
    \partial_t \omega + U(y)  \partial_x \omega +U''(y)\partial_x(-\Delta)^{-1}\omega=0,
\end{equation}
where we use $\omega$ to denote vorticity of the linearized equation. In terms of the stream function $\psi$, the linearized equation reads
\begin{equation} \label{linearized equation in terms of streamfunction}
    \partial_t \Delta \psi + U(y) \partial_x \Delta \psi-U''(y)\partial_x \psi=0.
\end{equation}

For the remaining portion of the paper, we will use $\mathcal{L}_{U}$ to denote the linearized operator displayed in \eqref{Linearized Euler Equations}, namely
\begin{equation}
\mathcal{L}_{U}=U(y)\partial_x - U''(y)\partial_x(\Delta^{-1}).
\end{equation} In the literature, the operator $\mathcal{L}_{U}$ goes by the name Rayleigh operator. Throughout the paper, we cling to the following assumptions on the background horizontal velocity $U$, 
\begin{enumerate}
[label=\textbf{(\subscript{A}{{\arabic*}})}]
    \item \label{Spectral} $\mathcal{L}_{U}$ has no eigenvalue or embedded eigenvalue.
    \item \label{Monotone} $U$ is monotonic, namely, there exits $c_0$ such that $U'>c_0>0$.
\end{enumerate}

It is well-understood that if assumption~\ref{Spectral} holds, then the background shear flow $(U(y),0)$ is linearly stable. Moreover, it was predicted earlier by Orr \cite{orr1907} in 1907 that for the Couette flow, the velocity field converges to a shear flow as time goes to infinity. In the modern language, this phenomenon goes by the name inviscid damping. A couple decades after that, Case \cite{Case1960} predicted that the linear inviscid damping still holds true even for monotonic shear flows. Since then, there has been an immensely growing number of active research in the area.  We point readers to some of the recent papers on the subject: \cite{zillinger2017linear, WeiZhangZhao2015, jia2020linear,jia2020linearGevrey} for the results on the linear inviscid damping pertaining to monotonic shear flows. For non-monotonic shears, however, there is another interesting phenomena that has been observed at linear level known as vorticity depletion.  We also refer to some recent papers in this direction and references therein: \cite{bedrossiancotizelatiVicol2019vortex, WeiZhangZhao2019, WeiZhangZhao2020, IonescuJia2020,ionescuJia2022spectraldensity, ionescuSameerJia2022linear}.

Going back to inviscid damping, in the past few years, a number of breakthrough papers have pointed out that such damping mechanism can indeed be carried over from linear to nonlinear level. We mention some results by \cite{BedrossianMasoudi2015,  MasmoudiZhao2020, ionescuJia2020non, IonescuJia2022}. It is worth mentioning that the aforementioned papers require Gevrey-smooth perturbations in order to allow for the inviscid damping to occur. This triggers one to ask a question what if the perturbations is taken to be less smooth? It has been observed that the nonlinear inviscid damping has everything to do with the class of perturbation and the topology of the function space where the perturbation is taken. In \cite{LinZeng2011}, Lin and Zeng proved that nonlinear inviscid damping for Couette flow is not true when the velocity perturbation is chosen in $H^s$ for $s< 5/2$. Their idea is to perturb the Couette flow, such that the spectral assumption \ref{Spectral} fails to hold. More precisely, they proved that by choosing suitable $(m,\gamma)$, the Rayleigh operator $\mathcal{L}_{U_{m,\gamma}}$ has an embedded eigenvalue. The perturbed shear flow in their paper takes the form
\begin{equation*}
    (U_{m,\gamma},0)=\Big(y+m\gamma^2E\big(\frac{y}{\gamma}\big),0\Big),
\end{equation*} 
where $E$ is the error function. 

Studying the spectra of an operator under certain class of perturbation has been a classical subject of investigation in mathematics. For self-adjoint operators, in particular the Schr\"{o}dinger operator, the spectra theory is well-developed, see \cite{kato1995, Kiselev1998, ChristKiselev1998, Remling1998, simon2000schrodinger, killip2002perturbations} and the references therein. The present work sees the importance of studying the spectra of the linearized operator under a class of perturbation. In particular, we study the spectra changing of the Rayleigh operator 
\begin{equation}\label{linearized operator}
\mathcal{L}_{U_{m,\gamma}}= U_{m,\gamma}\partial_x-U''_{m,\gamma}\partial_x\Delta^{-1},
\end{equation}
associated with shear flow $(U_{m,\gamma}(y),0)$ in a finite channel $\mathbb{T}_{2\pi}\times [-1,1]$ where the perturbed horizontal velocity takes the form 
$$
U_{m,\gamma}(y)=U(y)+m\gamma^2\widetilde\Gamma(y/\gamma)
$$ with $\big\{m\gamma^2\widetilde\Gamma(y/\gamma)\big\}_{m\geq 0}$ being a family of the perturbations with an amplitude parameter $m$. 

\subsection{Main results}
Let us now state the main contributions of the paper. In the present work, we classify three different regimes of the existence and nonexistence of eigenvalue or embedded eigenvalue of the Rayleigh operator \eqref{linearized operator} associated with the shear flow $(U_{m,\gamma}(y),0)$, where
\begin{equation}\label{perturbed shear}
    U_{m,\gamma}(y)=U(y)+ m \gamma\int_{0}^{y} \Gamma(z/\gamma)\; dz=:U(y)+m \gamma^2 \widetilde{\Gamma}(y/\gamma),
\end{equation}
and $\Gamma$ satisfies
\begin{equation}\label{condition on Gamma}
  \Gamma \in L^1(\mathbb{R}) \cap H^4(\mathbb{R}),\quad -C_1<\dfrac{\Gamma'(y)}{y}<0,\quad \int_{\mathbb{R}} \dfrac{\Gamma' (y)}{y} \;dy=-1. 
\end{equation}

This classification depends on the distance of the parameter $m$ relative to $m_*$.
 \begin{theorem}[Linear]\label{linear}
Let $U \in C^4(-1,1)$ satisfy \ref{Monotone} together with the conditions $ U(0)=0, U''(0)=0$ and $-C \leq U''(y)/U(y)<0$. If the
assumption~\ref{Spectral} holds, then there exist $\delta_0$ and $\gamma_0>0$ such that for all $0<\gamma < \gamma_0$, there is $m_*$ so that the Rayleigh operator $\mathcal{L}_{U_{m,\gamma}}$
\begin{enumerate}[label=\normalfont{(\alph*)}]
    \item\label{Stability}\textup{(Stability)} has no eigenvalue or embedded eigenvalue when $0\leq m<m_*$;\\
    
    \item\label{Neutral mode} \textup{(Neutral mode)}  has only an embedded eigenvalue when $m=m_*$;\\
    
    \item\label{Instability} \textup{(Instability)} has at least two eigenvalues (each corresponds to unstable and stable mode respectively) when $m_*<m\leq m_*+\delta_0$.
\end{enumerate}
In the regime where $m_*<m\leq m_*+\delta_0$, if $\kappa=\kappa_r +i\kappa_i$ is an eigenvalue of $\mathcal{L}_{U_{m,\gamma}}$, then $\kappa_r$ and $\kappa_i$ satisfy the following estimates
\begin{equation}\label{estimate for lambda r}
\begin{aligned}
   C^{-1}\gamma |m-m_*|\leq &|\kappa_r|\leq C\gamma |m-m_*|,\quad
   &|\kappa_i|\leq C\gamma |m-m_*|.
  \end{aligned}
\end{equation}
Here $C>1$ and $\delta_0$ are constants independent of $\gamma$.

\end{theorem}
\begin{remark} Let us make a few remarks related to Theorem~\ref{linear}.
\begin{enumerate}[label=\normalfont{(\roman*)}]
    \item A direct application of the main result in \cite{WeiZhangZhao2015} tells us that the linear inviscid damping holds for the shear flows $(U_{m,\gamma},0)$ with $0\leq m<m_*$. 
    \item Since $\delta_0$ is independent of $\gamma$, the growing mode can reach the size $\approx \delta_0\gamma$ by increasing $m$, which also implies the lower bound of the semigroup for $m=m_*+\delta_0$
    \begin{equation*}
       \big\|e^{t\mathcal{L}_{U_{m,\gamma}}}\big\|_{L^2\to L^2}\geq C^{-1}e^{C^{-1}\delta_0\gamma t}.
    \end{equation*}
    \item We would like to highlight the implication of Theorem~\ref{linear} (c): there exists $\delta_1<\delta_0$ so that $m_*+\delta_1<m<m_*+\delta_0$, the shear flow $U_{m,\gamma}$ is also nonlinearly unstable. This follows from the result of Grenier \cite{Grenier2000}. Hence, in this regime, linear instability does imply nonlinear instability. We also refer to \cite{Lin2003, Vishik2003,deng2018long,DengZillinger2021} for recent instability results of shear flows. 
    \item  Although our work is done in the finite channel $\mathbb{T}_{2\pi} \times  [-1,1]$, it is indeed extendable to the infinite channel $\mathbb{T}_{2\pi} \times \mathbb{R}$ setting. This requires some digressions into some of the details in our analysis. We avoid to pursue it here. 
    \item The parameter $m$ in the Theorem~\ref{linear} is  taken to be near $m_*$. As seen in the statement, the emergence or disappearance of eigenvalues or embedded eigenvalue hinges  on the location of $m$ relative to $m_*$.  In the paper \cite{limasmoudizhao2022} the authors prove existence of growing modes for $m>M_0$, where $M_0$ is required to be sufficiently large. The result in the present work is different compared to theirs in which ours focuses more on the regime where $m>M_0$ and $m_*\ll M_0$. Additionally, here we work with a general monotonic background shear flow rather than the Couette flow. We also manage to obtain a precise description on how a growing mode is generated as the parameter $m$ varies. 
    
    It is interesting to follow the movement of the eigenvalue as $m$ increases and connect two regions discussed in \cite{limasmoudizhao2022} and the present work. This may require a global bifurcation argument. 
    \item We also introduce the general transition threshold problem we tackle for a family of linear operators which is closely related to the stability problem: 

    Let $T$ be a linear operator and $\{L_m\}_{m\in [a,b]}$ be a family of linear perturbation operators continuous in $m$ with $L_a=0$. Let $A\subset \mathbb{C}$ be the region that we are interested in. Suppose that at $m=a$,  the operator $T$ has no eigenvalue in $A$, namely, $\sigma_d(T)\cap A=\emptyset$ and at $m=b$, the operator $T+L_b$ has eigenvalue in $A$. One can then ask the following question : Determine the largest $m_1$ such that $\sigma_d(T+L_m)\cap A=\emptyset$ for $a\leq m<m_1$ and the smallest $m_2$ such that $\sigma_d(T+L_m)\cap A\neq \emptyset$ for $m_2 <m\leq b$! In spirit, this is essentially the question we are tackling here.
\end{enumerate}
\end{remark}
Next, we state our second result pertaining instability of $U_{m,\gamma}$ at the nonlinear level. We regard it as a quantitative description of the ``boundary'' between the linearly stable shear flows and linearly unstable shear flows. We record it below.
 \begin{theorem}[Nonlinear]\label{nonlinear}
Fix $N\in \mathbb{Z}^+ \cup \{0\}$. Consider $U_{m,\gamma} \in C^{2N+4}(-1,1)$ given by \eqref{perturbed shear} satisfying the assumptions in Theorem~\ref{linear} and the condition
\[
  U^{(2n)}_{m,\gamma} (0)=0, \text{ for } n=1,2,...,N, N+1.
\]
Let $\gamma_0$ and $m_*$ be the same as in Theorem~\ref{linear}. 
Then for all $0\leq\gamma<\gamma_0$ and $m$ satisfying 
\begin{align*}
    |m-m_*|\leq C \gamma^N,
\end{align*} 
there exists a nontrivial steady solution $(u(x,y),v(x,y))$ close to the background shear flow $(U_{m,\gamma}(y),0)$ in the following sense 
    \begin{equation}\label{norm of the distance background and nontrivial}
    \norm{(u(x,y),v(x,y))-(U_{m,\gamma}(y),0)}_{H^{5/2-\tau+N}(\mathbb{T}_{2\pi} \times (-1,1))}\leq \gamma^\tau,
    \end{equation}
for any $0<\tau<1$.
\end{theorem}
\begin{remark}
    Let us make a few remarks regarding our results in Theorem \ref{nonlinear}:
    \begin{itemize}
        \item The authors \cite{LinZeng2011} consider the Couette flow. Their results coincide with ours precisely when we set $N=0$, $U(y)=y$, and $\Gamma(y)=2e^{-y^2}/\sqrt{\pi}$. They proved that the stable Couette flow becomes unstable when the perturbation is taken in weaker topology, namely in the Sobolev space $H^{5/2-}$. In the present work, we verify that at the nonlinear level, the parameter $m$ also plays a key role in causing such phenomenon. Concretely, it is related to the regularity of the function space where the perturbation that leads to asymptotic instability is taken: when $m$ is chosen even closer to $m_*$, one can take the perturbation in a smoother Sobolev space.   
        \item We refer to \cite{LinWeiZhangZhu2022,wang2022dynamics,castro2023traveling,franzoi2023space} for recent works on the nontrivial stationary solutions near Couette flow and to \cite{coti2023stationary} for the stationary structures near the Kolmogorov and Poiseuille flows. 
        \item We also mention some results on the 2-D water wave problem with constant vorticity \cite{Berti2021,sinambela2022existence,haziot2022traveling} and references therein.
    \end{itemize}
    
\end{remark}

Before stating the idea of our proofs, we would like to introduce notations and conventions that are used frequently throughout the rest of the paper. We record them in the following subsection. 
\subsection{Notations and conventions} 
We use symbols $'$ and $''$ to denote the first and second derivative in $y$, respectively. Additionally, we will use $f^{(n)}(y)$ to mean the $n$-th derivative of the function $f$ with respect to $y$. Due to a number of parameters and variables involved in our analysis, we would like to point out conventions we have followed throughout the paper. For eigenfunctions, solutions of differential equation, or any other functions used in our analysis, in general, they will depend on the variables/parameters: $y,m,\gamma,c,\lambda$. The function and its complete (ordered) argument appear as $\phi(y,m,\gamma,c,\lambda)$. 

However, we warn readers that the argument in the function may change depending on the context. The reason for this is to keep the notation simpler and compact when possible. More precisely, if a priori we know that one (or more) of the variables/parameters in the argument is set to $0$, then we suppress the dependency of that variable(s)/parameter(s) in the argument of the function. For instance, let us say that we fix $\gamma=0$, then the function $\phi$ together with its argument read $\phi(y,m,c,\lambda)$. Additionally, if, on top of $\gamma=0$, we fix another variable: $c=0$, then the function and its argument read $\phi(y,m,\lambda)$. If the variable(s)/parameter(s) is fixed to a specific value other than $0$, then we shall display the value on the fixed variable(s)/parameter(s) in the function. 

In the present work, the constants in our computations will be denoted in various ways. We use $C, c_0, C_0$ and $C_1$ to mean some generic positive constants. In general, we do not display the parameter on which they depend on. However, there are a few parts in our analysis where we explicitly state the variable dependency in the constant.

Throughout our estimates, we will frequently use the notations $``\lesssim", ``\gtrsim", ``\approx"$. For arbitrary functions $f$ and $g$, the inequality $f \lesssim g$ means that there exists a constant C such that $f \leq C g$. Similarly, $f \gtrsim g$  means that there exists $C$ such that $f \geq C g$. Lastly, the notation $f \approx g$ is equivalent to the statement that $C^{-1}g \leq f \leq C g$, for some constant $C$. 

\subsection{Idea of the proofs}
In this subsection, we outline key ideas and mathematical machinery used in our proofs. We fix the two parameters $m,\gamma$ at the outset. Concerning Theorem~\ref{linear}, we will work for the most part with the linearized equation in terms of the vorticity \eqref{Linearized Euler Equations} or stream function \eqref{linearized equation in terms of streamfunction}. When the background shear $(U,0)$ is replaced by $(U_{m,\gamma},0)$, followed by taking the Fourier transform in $x$ of \eqref{Linearized Euler Equations}, we obtain
\[
\dfrac{-1}{ik}\partial_t \hat{\omega}=\mathcal{R}_{m,\gamma,k} \hat{\omega},
\] where
$\mathcal{R}_{m,\gamma,k}=(U_{m,\gamma}(y)-U''_{m,\gamma}(y)(\partial_y^2-k^2)^{-1})$.
If $c=c_r+ic_i \in \mathbb{C}$ is an eigenvalue of the operator $\mathcal{R}_{m,\gamma,k}$, then there exists a function $\Omega(y,m,\gamma,c,k) \in L^2(-1,1)$ such that 
\begin{equation}\label{operator R}
   \mathcal{R}_{m,\gamma,k}\Omega(y,m,\gamma,c,k)  =c \Omega(y,m,\gamma,c,k),
\end{equation}
where
\[
\Delta \Phi(y,m,\gamma,c,k)=\Omega(y,m,\gamma,c,k), \qquad \Phi(\pm1,m,\gamma,c,k)=0.
\]
 This leads us to the Rayleigh equation
\begin{equation}\label{Rayleigh Equation}
    -\partial^2_y \Phi(y,m,\gamma,c,k)+\dfrac{U''_{m,\gamma}}{U_{m,\gamma}-c}\Phi(y,m,\gamma,c,k)=-k^2\Phi(y,m,\gamma,c,k),
\end{equation} 
with $\Phi(-1,m,\gamma,c,k)=\Phi(1,m,\gamma,c,k)=0$. The homogeneous Rayleigh equation is well studied in \cite{WeiZhangZhao2015, limasmoudizhao2022} for a fixed shear flow. In this paper, the shear flow is changing as the parameters $m$ and $\gamma$ vary. For this reason, we derive some uniform estimates related to the solutions of \eqref{Rayleigh Equation} on these parameters. See Lemma~\ref{phi1 phi2} in Appendix~\ref{Related Results}. 

Our main strategy is to study the eigenvalues of the Rayleigh operator \eqref{operator R} which is, in some sense, equivalent to studying the eigenvalue of the Schr\"{o}dinger-type operator found in \eqref{Rayleigh Equation}. We provide Remarks~\eqref{connection between c and k^2} and \eqref{TFAE statement} to highlight connections between the eigenvalues of the Rayleigh operator \eqref{linearized operator}, $\mathcal{R}_{m,\gamma,k}$ \eqref{operator R} and the Schr\"{o}dinger-type operator in \eqref{Rayleigh Equation}. 

Understanding the eigenvalue problem for Schr\"{o}dinger-type operator \eqref{Rayleigh Equation} for $c=0$ plays an undoubtedly pivotal role in studying the transition threshold problem considered here. For convenience, the eigenvalue problem of the Schr\"{o}dinger-type operator for $c=0$ is recast alternatively as
\begin{equation}
    \mathcal{H}_{m,\gamma}\Phi=-k^2 \Phi,
\end{equation}
where $ \mathcal{H}_{m,\gamma}$ is given by
\begin{equation}\label{operator H}
     \mathcal{H}_{m,\gamma}:=-\partial^2_y+Q_{m,\gamma}(y),
\end{equation}
 with
 \begin{equation}\label{Q}
 Q_{m,\gamma}(y):=\dfrac{U''_{m,\gamma}(y)}{U_{m,\gamma}(y)}.
 \end{equation}
  \begin{remark}\label{connection between c and k^2}We would like to direct readers to the following important points regarding the connection between eigenvalues of $\mathcal{R}_{m,\gamma,k}$ and $\mathcal{H}_{m,\gamma}$:
     \begin{itemize}
         \item In light of equation \eqref{operator R}, if $c=0$ is an embedded eigenvalue of the operator $\mathcal{R}_{m,\gamma,k}$, then $-k^2$ is an eigenvalue of the operator $\mathcal{H}_{m,\gamma}$ with $-1\geq -k^2 \geq \lambda_{m,\gamma}$.
         \item Moreover, if $\lambda_{m,\gamma}\geq -1$, then $-k^2$ with $k\in \mathbb{Z}$ cannnot be an eigenvalue of $\mathcal{H}_{m,\gamma}$ which implies that $c=0$ is not an embedded eigenvalue of the operator $\mathcal{R}_{m,\gamma,k}$.
     \end{itemize}
 \end{remark}
\subsubsection{Determining $m_*$}
To determine $m_*$, we study the Schr\"{o}dinger-type  operator $  \mathcal{H}_{m,\gamma}$. 
Let $\lambda_{m,\gamma}$ be the minimal eigenvalue associated with $ \mathcal{H}_{m,\gamma}$, then 
\begin{equation}\label{lambda m gamma}
    \lambda_{m,\gamma}= \min_{\substack{\lVert \Psi\rVert_{L^2}=1 \\ \Psi \in H^1_0}}(\mathcal{H}_{m,\gamma} \Psi,\Psi)=\min_{\substack{\lVert \Psi\rVert_{L^2}=1 \\ \Psi \in H^1_0}}\int_{-1}^{-1} |\partial_y\Psi|^2+\int_{-1}^1 Q_{m,\gamma}(y)|\Psi|^2\;dy.
\end{equation}
 We write the associated quadratic form as
\begin{equation}\label{quadratic form 2}
 H_{m,\gamma} \Psi:=\lVert \Psi' \rVert_{L^2}^2 + \int_{-1}^{1} Q_{m,\gamma}(y) |\Psi|^2 \; dy.
 \end{equation}
 Let $\Psi(y,m,\gamma,c,\lambda_{m,\gamma})$ be the eigenfuntion associated with the minimal eigenvalue $\lambda_{m,\gamma}$. From \eqref{lambda m gamma}, one can see that the function $\Psi(y,m,\gamma,\lambda_{m,\gamma})$ solves the Rayleigh equation \eqref{Rayleigh Equation} along with the boundary conditions for $c=0$ and $-k^2$ being replaced by $\lambda_{m,\gamma}$.

Furthermore, as $\gamma \to 0$, we obtain 
\[
(\mathcal{H}_{m,\gamma} \Psi,\Psi)\to (\mathcal{H}_{m,0} \Psi,\Psi):= \int_{-1}^{-1} |\partial_y\Psi|^2+\int_{-1}^{1}\dfrac{U''}{U}|\Psi|^2\;dy-\dfrac{m|\Psi(0)|^2}{U'(0)}.
\]
Similarly, let $\lambda_{m,0}$ be the minimal eigenvalue associated with $ \mathcal{H}_{m,0}$, then 
\[
\lambda_{m,0}= \min_{\substack{\lVert \Psi\rVert_{L^2}=1 \\ \Psi \in H^1_0}}(\mathcal{H}_{m,0} \Psi,\Psi).
\]
Additionally, the corresponding quadratic form is given by 
\begin{equation}\label{quadratic form 1}
 H_{m,0} \Psi:=\lVert \Psi' \rVert_{L^2}^2 + \int_{-1}^{1} \dfrac{U''}{U} |\Psi|^2 \; dy-\dfrac{m |\Psi(0)|^2}{U'(0)}.
 \end{equation} 
 Hence, the eigenfunction $\Psi(y,m,\gamma,\lambda_{m,0})$ associated with the minimal eigenvalue $\lambda_{m,0}$ is the solution of the associated Euler–Lagrange equation 
\begin{equation}\label{eq:Euler-Lagrange}
    -\partial^2_y \Psi +\dfrac{U''(y)}{U(y)}\Psi-\dfrac{m\Psi(0)\delta(0)}{U'(0)}=\lambda_{m,0}\Psi,\quad \Psi(y=\pm1)=0.
\end{equation}

In Lemma~\ref{Monotonicity}, we obtain an explicit formula given by  the function $\mathfrak{M}$ which reveals the relationship  between $m$ and $\lambda$ so that \eqref{eq:Euler-Lagrange} admits a nontrivial solution. Moreover, we show that $\mathfrak{M}$ is monotonic. Additionally, we prove that $\lambda_{m,\gamma}$ of $\mathcal{H}_{m,\gamma}$ is also monotonic with respect to the parameter $m$, see Lemma~\ref{monotonicity of lambda m gamma}. As a result, we obtain an inequality that connects the two minimal  eigenvalues
\[
|\lambda_{m,\gamma}-\lambda_{m,0}|\lesssim m \gamma. 
\] 
It will be clear later that from this inequality, for any fixed $\gamma$ small, one can determine $m_*$ such that $\lambda_{m_*,\gamma}=-1$.

\subsubsection{Stability for $m<m_*$} Since the eigenvalue problem \eqref{Rayleigh Equation} is an ODE (Ordinary Differential Equation) problem, it is natural to introduce the associated Wronskian. We recall that the Wronskian $\mathcal{W}$ of two solutions $f$ and $g$ to a second order ODE is given by 
\[
\mathcal{W}[f,g]=fg'-f'g.
\] If $\mathcal{W}[f,g]=0$, then $f$ and $g$ are linearly dependent, i.e. $f=Cg$, for some constant $C$. Let $\varphi^{-}(y,m,\gamma,c,\lambda)$ and $\varphi^{+}(y,m,\gamma,c,\lambda)$ be two solutions of the Rayleigh equation \eqref{Rayleigh Equation} with $-k^2=\lambda$ satisfying the conditions:
\begin{equation}\label{boundary conditions for varphi}
    \varphi^{-}(-1,m,\gamma,c,\lambda)=0=\varphi^{+}(1,m,\gamma,c,\lambda)\text{ and }(\varphi^{-})'(-1,m,\gamma,c,\lambda)=1=(\varphi^{+})'(1,m,\gamma,c,\lambda).
\end{equation}
Direct computation yields the Wronskian of these two solutions is $y-$independent, namely
\[
\partial_y \mathcal{W}[\varphi^{-},\varphi^{+}](y,m,\gamma,c,\lambda)=\partial_y\bigg(\varphi^{-}(\varphi^{+})'-(\varphi^{-})'\varphi^{+}\bigg)=0.
\]

As a result, upon computing the Wronskian of $\varphi^{-}$ and $\varphi^{+}$ and using the boundary conditions of $\varphi^{+}$ and $(\varphi^{+})'$ at $y=1$, we can infer 
\begin{equation}\label{wronskian of generic varphi}
   \mathcal{W}[\varphi^{-},\varphi^{+}](m,\gamma,c,\lambda)=\varphi^{-}(1,m,\gamma,c,\lambda). 
\end{equation}

By far, we have introduced three operators, eigenvalue problem associated with each of them and the notion of Wronskian of two solutions $\varphi^-$ and $\varphi^+$. To clear some things up and highlight the connections between them, we present the following remark:
\begin{remark}\label{TFAE statement}
    For any $k\in \mathbb{Z}\setminus\{0\}$. The following statements are equivalent for $c\notin \mathbb{R}$:
    \begin{itemize}
        \item $\kappa=ikc$ is an eigenvalue of the Rayleigh operator $\mathcal{L}_{U_{m,\gamma}}$;
        \item $c$ is an eigenvalue of the operator $\mathcal{R}_{m,\gamma,k}$;
        \item There is a nontrivial solution $\Phi$ solving the Rayleigh equation \eqref{Rayleigh Equation};
        \item $-k^2$ is an eigenvalue of the Schr\"{o}dinger-type operator $-\partial^2_{y}+\frac{U''_{m,\gamma}}{U_{m,\gamma}-c}$;
         \item $c$ is a zero of $\mathcal{W}[\varphi_1,\varphi_2](m,\gamma,c,\lambda)$;
        \item $\varphi^{-}(1,m,\gamma,c,\lambda)=0$.
    \end{itemize}
\end{remark}
It is well-known that due the Howard’s semicircle theorem \cite{Howard1961}, the eigenvalues $c$ of the operator $\mathcal{R}_{m,\gamma,k}$ belong to a semicircle on the complex plane. In light of that, we introduce the following two sets:
\begin{equation}\label{sets}
\begin{aligned}
    &\mathfrak{B}:=\Bigl\{c=c_r+ic_i:\bigg(c_r-\dfrac{U_{m,\gamma}(-1)+U_{m,\gamma}(1)}{2}\bigg)^2+(c_i)^2 \leq \dfrac{\big(U_{m,\gamma}(1)-U_{m,\gamma}(-1)\big)^2}{2}\Bigr\}\subset \mathbb{C},\\
    &\mathfrak{D}_{\epsilon_0}:=\Bigl\{c=c_r+ic_i: c_r \in \text{Ran}(U_{m,\gamma})\text{ and } 0\leq |c_i|<\epsilon_0 \Bigr\}\subset \mathbb{C}.
    \end{aligned}
\end{equation} 
In the aforementioned paper, the theorem essentially asserts that for $c\notin \mathfrak{B}$, $\mathcal{W}[\varphi_1,\varphi_2](m,\gamma,c,\lambda)\neq 0$. It is also easy to check that for $|\lambda|>C_0$ with $C_0>0$ large enough, $\mathcal{W}[\varphi_1,\varphi_2](m,\gamma,c,\lambda)\neq 0$ for all $c$. Thus we only focus on the case $\big\{(c,\lambda):~|\lambda|\leq C_0,\ c\in \mathfrak{B}\big\}$ which is compact. Moreover, we have an explicit formula for the Wronskian when $c\in  \mathfrak{D}_{\epsilon_0} \setminus \text{Ran}(U_{m,\gamma})$, see equation \eqref{wronskian of varphi}.

To prove the stability for the case $m<m_*$, we apply the Rouche’s theorem from complex analysis and use a contradiction argument. We first show that the Rayleigh operator has no embedded eigenvalue. Secondly, we show the nonexistence of eigenvalue $c\in \mathfrak{B}$ with $c_i \neq 0$. The argument goes as follows: if there is an eigenvalue in $\mathfrak{B}$ with $c_i\neq0$ for $m<m_*$, then such eigenvalue should be generated by an embedded eigenvalue. However, since there is no embedded eigenvalue hence there is no such eigenvalue with $c_i \neq 0$. The nonexistence of eigenvalue is equivalent to the nonexistence of growing mode in the regime $m<m_*$.

\subsubsection{Instability for $m>m_*$} To show the existence of growing mode, we study the zeros of the Wronskian $\mathcal{W}[\varphi^{-},\varphi^{+}](m,\gamma, c, -1)$. We prove that as $m$ getting large slightly away from $m_*$, the eigenvalue $c$ moves away from zero, namely, there exists $c:=c_r(m)+ic_i(m)$ such that $c(m_*)=0$ and 
$$\mathcal{W}[\varphi^{-},\varphi^{+}](m,\gamma, c_r+ic_i,-1)=0,$$
for $m_*\leq m<m_*+\delta_0$. This is done by studying the partial derivatives of the modified Wronskian $W(m,\gamma,c_r+ ic_i,-1)$ with respect to $c_i$, $m,$ and applying an ODE-type argument.

\subsubsection{Nonlinear asymptotic instability for $m$ near $m_*$}
To prove nonlinear asymptotic instability, our argument relies on the Crandall--Rabinowitz local bifurcation theory. Using the observation that any stream function that solves $\Delta \psi = \widetilde G \circ \psi$ gives rise to a steady solution of the nonlinear problem \eqref{Euler Equations}, we apply the bifurcation theory to the Poisson problem. Due to the lack of evenness in our background stream function $\psi_0$, we work with the modified $\psi_0$ given by $\widetilde \psi_0$. The map $\widetilde G$ is given by the  modified background stream function via $\widetilde G(\widetilde\psi_0)=U'(y)$.  Such map is well-defined at $y=0$. However, it is not obvious that its higher derivatives are still continuous at $y=0$. Thanks to the vanishing assumptions on even derivatives of $U$ at $y=0$, we show that $\widetilde G$ is sufficiently smooth for all values of $y$. This result is recorded in Lemma~\ref{regularity lemma}. Upon checking all the assumptions to apply the local bifurcation theory, we deduce the existence of local curve of solutions bifurcating from the background shear $(U,0)$.

For fixed $\gamma \in (0,\gamma_0)$, one can choose  $m_1$ and $m_2$ close enough to $m_*$ so that $\lambda_{m_1,\gamma}<-1<\lambda_{m_2,\gamma}$. We then pick $m^\circ \in (m_1,m_2)$. Instead of bifurcating from the generic background shear $(U,0)$, we bifurcate our local curve of solutions from  $(U_{m^\circ,\gamma},0)$. The $m^\circ$ is chosen such that we obtain a local curve of solutions that exhibits a nontrivial (non-sheared) steady solution with $k=1$. By doing this, we have confirmed an existence of a nontrivial steady solution that is $2 \pi$-periodic in $x$ which equivalently proves the nonlinear asymptotic instability. 
\section{Preliminary}\label{Preliminary}

\subsection{Plan of the article}
 Let us now outline the general structure of the paper. In Section~\ref{Preliminary}, we record some estimates concerning distance between eigenvalues and monotonicity properties. As will be seen later, these estimates are crucial in proving our two main results. In the same section, we have provided some more discussions on the Wronskian. 

The main results of the present work are then proved in Section~\ref{proof of main results}. The section is started with proof Theorem~\ref{linear}. It is partitioned into into four parts, each devoted to \ref{Stability}, \ref{Neutral mode}, \ref{Instability} and the estimate in \eqref{estimate for lambda r}. Following that is the proof of Theorem~\ref{nonlinear}. As an intermediate step, we have included a lemma on a bifurcation result that is important to close our argument in Theorem~\ref{nonlinear}. 

Finally, we have included two appendices: Appendix~\ref{Related Results} and Appendix~\ref{Regularity}. The first appendix contains some estimates concerning  functions in the expression of the regular solution of the Rayleigh equation. The second appendix is devoted to discuss the regularity of the the map $\widetilde{G}$.

\subsection{Estimates} The following subsection is aimed to present important estimates on the location of eigenvalues and monotonicity results that are fundamental in our analysis. The whole idea hinges on the careful study of the Rayleigh equation \eqref{Rayleigh Equation} and an ODE-type of argument. In particular, at $c=0$, we manage to understand how eigenvalue transitions as we fix one parameter and vary the rest. 
The Rayleigh equation \eqref{Rayleigh Equation} has rather a delicate limiting structure as the parameter $\gamma \to 0$. To this end, we show the convergence of the the Rayleigh equation \eqref{Rayleigh Equation}. In light of our interest in steady solutions, taking $c=0$, we obtain 
\begin{equation}\label{Rayleigh equation at c=0}
    -\partial^2_y \Phi(y,m,\gamma,\lambda)+Q_{m,\gamma}\Phi(y,m,\gamma,\lambda)=\lambda \Phi(y,m,\gamma,\lambda),
\end{equation}
where $Q(y,m,\gamma,\lambda)$ is as in \eqref{Q}. Explicity, it takes the form
\[
  Q(y,m,\gamma,\lambda)= \dfrac{U''(y)}{U_{m,\gamma}(y)}+\dfrac{m\Gamma'(y/\gamma)}{\gamma(y/\gamma)}\dfrac{1}{\frac{U(y)}{y}+\gamma m \widetilde{\Gamma}(y/\gamma)/(y/\gamma)}.
 \] 
 
It is easy to check that $\Phi(y,m,\gamma, \lambda)$ solves \eqref{Rayleigh equation at c=0}, if and only if the following identity
\begin{equation}\label{weak formulation}
    -\int_{-1}^1\Phi(y,m,\gamma, \lambda)\partial_{yy}\varphi(y)
    +Q_{m,\gamma}\Phi(y,m,\gamma, \lambda)\varphi(y)\;dy=\lambda\int_{-1}^1\Phi(y,m,\gamma, \lambda)\varphi(y)\;dy
\end{equation}
holds for any test function $\varphi\in H_0^1(-1,1)$.
Taking $\gamma \to 0$ leads to
\[
 Q(y,m,\gamma,\lambda)\to \dfrac{U''(y)}{U(y)}-\dfrac{m\delta(0)}{U'(0)}. 
 \]
Hence, for all test functions $\varphi \in H^1_0(-1,1)$ we have 
\begin{equation}\label{Rayleigh equation at gamma 0}
   \int_{-1}^1 -\Phi(y,m,\lambda)\partial^2_y \varphi(y)  + \dfrac{U''(y)}{U(y)} \Phi(y,m,\lambda)\varphi(y)-\lambda  \Phi(y,m,\lambda)\varphi(y) dy =\dfrac{m\Phi (0,m,\lambda)\varphi(0)}{U'(0)},
\end{equation}
 In other words, when $\gamma=0$, the Rayleigh equation becomes a second order differential equation, namely
 \begin{equation}\label{second order DE}
     \partial^2_y \Phi(y,m,\lambda) - \dfrac{U''(y)}{U(y)}\Phi(y,m,\lambda)+\lambda\Phi(y,m,\lambda)=-\dfrac{m}{U'(0)} \Phi(0,m,\lambda)\delta(0).
 \end{equation}
 
The following lemma records a sufficient and necessary condition for the differential equation in \eqref{second order DE} to admit a nontrivial solution. Moreover, it also asserts that  $m$ which is proved to be $\lambda$-dependent satisfies some monotonicity condition with respect to $\lambda$. 
\begin{lemma}[Monotonicity of $m$]\label{Monotonicity} The differential equation \eqref{second order DE} admits a nontrivial solution for  fixed $m$ with $\lambda<0$ if and only if $m=\mathfrak{M}(\lambda)$, for some function $\mathfrak{M}$. Moreover, $m$ is strictly decreasing with respect to $\lambda$, 
\begin{equation}
    \partial_{\lambda}\mathfrak{M}<0.
\end{equation}
\end{lemma}

\begin{proof}
   Our proof relies on the representation of solution of \eqref{second order DE}. As it will be clear soon that such representation gives us a precise description on how the eigenvalue $\lambda$ and the parameter $m$ are related from which the monotonicity result then follows. 
 
 Naturally, in order to solve \eqref{second order DE} one would want to solve the homogeneous version of it namely
\begin{equation}\label{homogeneous DE}
\widetilde{\phi}''(y,\lambda)-\dfrac{U''(y)}{U(y)} \widetilde{\phi}(y,\lambda) +\lambda \widetilde{\phi}(y,\lambda)=0.
\end{equation} 

In view of \eqref{Rayleigh equation at c=0}, the homogeneous problem \eqref{homogeneous DE} can be obtained by setting $m$ to equal 0 in \eqref{Rayleigh equation at c=0}. By fixing the value for of $m$, the function $\widetilde{\phi}$ in \eqref{homogeneous DE} and the homogeneous equation itself become $m$-independent. Thanks to Lemma~\ref{phi1 phi2} in Appendix~\ref{Related Results}, the homogeneous differential equation \eqref{homogeneous DE} has a regular solution which takes the form $\widetilde{\phi}(y,\lambda)=U(y)\phi_1(y,\lambda)$, where 
\begin{equation}\label{phi1}
   \phi_1(y,\lambda)= 1+\int_0^{y} \dfrac{-\lambda} {U^2(w)} \int_0^{w}\phi_1(z,\lambda) U^2(z) \;dz\;dw,
\end{equation}
and it satisfies the conditions
\begin{equation}\label{boundary conditions}
\phi_1(0,\lambda)=1, \qquad \phi'_1(0,\lambda)=0. 
\end{equation} We would like to mention that, alternatively, to obtain the expression of $\phi_1$, one can solve the following differential equation

\begin{equation}\label{differential equation for phi1}
    (U^2(y) \phi'_1(y,\lambda))'=-\lambda \phi_1(y,\lambda)U^2,
\end{equation}
which is obtained by simply plugging $\widetilde{\phi}$ into \eqref{homogeneous DE}.

It is important to note that since $c=0$, then $\phi_2$ in Lemma~\ref{phi1 phi2} is uniformly equal to 1. That is the reason why we do not see it in the expression of $\widetilde{\phi}$.
As a consequence, one can check that
\begin{equation}\label{phi tilde}
    \Phi(y,\lambda):=\Xi^{-}(y,\lambda)\chi_{[-1,0)}(y)+\Xi^{+}(y,\lambda)\chi_{(0,1]}(y),
\end{equation}

with
 \begin{equation}\label{expansion}
 \begin{aligned}
 \Xi^{\mp}(y,\lambda)&:= U(y) \phi_1(y,\lambda) \int_{\mp1}^y \dfrac{1}{U^2(z)}\bigg(\dfrac{1}{\phi_1^2(z,\lambda)} -1\bigg)\;dz
                            \\ &\quad + U(y) \phi_1(y,\lambda) \int_{U(\mp1)}^{U(y)} \dfrac{(U^{-1})'(u_1)-(U^{\mp1})'(0)-(U^{-1})''(0)(u_1)}{(u_1)^2}\;du_1
                            \\ &\quad -\dfrac{U(y)\phi_1(y,\lambda)}{U'(0)U(z)}\Big|_{z=\mp1}^{z=y}
 \end{aligned}
 \end{equation}
solves the homogeneous problem \eqref{homogeneous DE} for $y \in [-1,1]\setminus \{0\}$, is continuous at $y=0$, and satisfies the boundary conditions: $\Phi(\pm 1,\lambda)=0$.

In view of the inhomogeneous differential equation \eqref{second order DE}, the function $\Phi(y,\lambda)$ \eqref{phi tilde} can be regarded as a weak solution of \eqref{second order DE} in the following sense
\begin{equation}
   -\int_{-1}^1 \Phi(y,\lambda) \varphi''(y)  + \dfrac{U''(y)}{U(y)} \Phi(y,\lambda)\varphi(y)-\lambda  \Phi(y,\lambda)\varphi(y) dy =\dfrac{m\Phi(0,\lambda)\varphi(0)}{U'(0)},
\end{equation}
for all test functions $\varphi(y)\in H^{1}_0(-1,1).$
Via integration by parts and using the fact that $\Phi$ solves the homogeneous problem away from 0, we arrive at the following equation
\[
-\partial_y \Xi^{+}(0,\lambda)+\partial_y \Xi^{-}(0,\lambda)=m\Phi(0,\lambda).
\]

Moreover, using the expression of $\Xi^{\mp}$ \eqref{expansion} and the facts that $U(0)=0,\phi_1(0,\lambda)=1$, and $\phi'_1(0,\lambda)=0$, we obtain
\[
\begin{aligned}
    \mathfrak{M}(\lambda):=m&=\dfrac{\partial_y \Xi^{-}(0,\lambda)-\partial_y \Xi^{+}(0,\lambda)}{\Phi(0,\lambda)}\\
    &=U'(0) \dfrac{\int_{-1}^{1} \dfrac{1}{U^2(y)}\bigg(\dfrac{1}{\phi^2_1(y,\lambda)}-1\bigg)\;dy}{\Phi(0,\lambda)}.
\end{aligned}
\]
Notice that the term $\Phi(0,\lambda)$ in the denominator of $\mathfrak{M}$ can be expressed as 
\[
\begin{aligned}
\Phi(0,\lambda)&=\lim_{y\to 0^-}\Xi^{-}(y,\lambda)=\dfrac{-1}{U'(0)}.
\end{aligned}
\]
Altogether, we conclude that $\mathfrak{M}$ takes the form
\begin{equation}\label{expression of m}
    \mathfrak{M}(\lambda)=-(U'(0))^2\int_{-1}^{1} \dfrac{1}{U^2(y)}\bigg(\dfrac{1}{\phi^2_1(y,\lambda)}-1\bigg)\;dy.
\end{equation}

Having derived the expression for the function $\mathfrak{M}$, we now continue to prove its monotonicity with respect to $\lambda$. We focus our attention into looking at how $\mathfrak{M}$ varies with respect to $\lambda$. This can be done by first differentiating \eqref{expression of m} with respect to $\lambda$. Direct computation yields
\begin{equation} \label{monotonicity of m}
\partial_{\lambda} \mathfrak{M}=-(U'(0))^2 \int_{-1}^{1} \dfrac{-2}{(U(y)\phi_1(y,\lambda))^2}\dfrac{\partial_\lambda\phi_1(y,\lambda)}{\phi_1(y,\lambda)} \;dy=:-(U'(0))^2 \text{I}.
\end{equation}
 By applying Lemma~\ref{F} to $\partial_\lambda \phi_1/\phi_1$, we infer $ -(U'(0))^2 \text{I}<0$ which, as a result, implies
\begin{equation}\label{partial lambda m}
    \partial_{\lambda}\mathfrak{M}<0.
\end{equation}
Hence, we conclude that $\mathfrak{M}$ is monotonically decreasing with respect to $\lambda$. 
\end{proof}

 In the next lemma, we derive an estimate on the distance between two minimal eigenvalues of $\mathcal{H}_{m,0}$ and $\mathcal{H}_{m,\gamma}$, respectively.

\begin{lemma}\label{distance between minimal eigenvalues} For any $m>0$ and $\gamma \neq 0$, we have
\[
|\lambda_{m,0}-\lambda_{m,\gamma}|\lesssim m\gamma.
\]
\end{lemma}
\begin{proof}
    Via the quadratic forms \eqref{quadratic form 2} and \eqref{quadratic form 1}, we can infer that 
 \begin{equation}\label{inequality for distance}
 \begin{aligned}
 &\lambda_{m,0} \leq H_{m,0}\Psi(y,m,\gamma,\lambda_{m,\gamma})\\
                    &= H_{m,\gamma}\Psi(y,m,\gamma,\lambda_{m,\gamma})-\int_{-1}^{1}\Bigg(\dfrac{U''(y)}{U_{m,\gamma}(y)}-\dfrac{U''(y)}{U(y)}\Bigg)|\Psi(y,m,\gamma,\lambda_{m,\gamma})|^2\;dy\\
                    &\quad -\int_{-1}^{1} \dfrac{m}{\gamma} \sigma(y/\gamma) \dfrac{|\Psi(y,m,\gamma,\lambda_{m,\gamma})|^2}{\frac{U(y)}{y}+\gamma m \widetilde{\Gamma}(y/\gamma)/(y/\gamma)} \; dy-\dfrac{m |\Psi(0,m,\gamma,\lambda_{m,\gamma})|^2} {U'(0)} \\
                    &= \lambda_{m,\gamma}-\int_{-1}^{1}\Bigg(\dfrac{U''(y)}{U_{m,\gamma}(y)}-\dfrac{U''(y)}{U(y)}\Bigg)|\Psi(y,m,\gamma,\lambda_{m,\gamma})|^2\;dy\\
                    &\quad- m\int_{-1/\gamma}^{1/\gamma} \sigma(y) \bigg(\dfrac{\gamma y}{U(\gamma y)}|\Psi(\gamma y,m,\gamma,\lambda_{m,\gamma})|^2-\dfrac{|\Psi(0,m,\gamma,\lambda_{m,\gamma})|^2}{U'(0)}\bigg)\;dy \\&\quad -m\int_{|y| \geq 1/ \gamma} \sigma(y) \dfrac{|\Psi(0,m,\gamma,\lambda_{m,\gamma})|^2}{U'(0)}\; dy +m \int_{-1}^{1} \dfrac{\sigma(y/\gamma)y^2}{\gamma U^2(y)} \dfrac{\gamma m \widetilde{\Gamma}(y/\gamma)/(y/\gamma)|\Psi(y,m,\gamma,\lambda_{m,\gamma})|^2}{1+\frac{y}{U(y)}\gamma m \widetilde{\Gamma}(y/\gamma)/(y/\gamma)} \; dy\\
                    &=:\lambda_{m,\gamma} +S_1+S_2+S_3+S_4,
 \end{aligned}
 \end{equation}
 where $\sigma(y)=\Gamma'(y)/y$.
 Let us now take a closer look at the terms $S_1,S_2,S_3,$ and $S_4$ and derive their estimates. We begin by deriving an estimate for $|S_1|$,
\begin{equation}
\begin{aligned}
|S_1|&\leq \int_{-1}^{1}\Bigg|\dfrac{U''(y)U(y)-U''(y)U_{m,\gamma}(y)}{U(y)U_{m,\gamma}(y)}\Bigg||\Psi(y,m,\gamma,\lambda_{m,\gamma})|^2\;dy\\&= \int_{-1}^{1}\Bigg|\dfrac{U''(y)(U(y)-U_{m,\gamma}(y))}{U(y)U_{m,\gamma}(y)}\Bigg||\Psi(y,m,\gamma,\lambda_{m,\gamma})|^2\;dy\\
         &\lesssim \int_{-1}^{1} \Bigg|\dfrac{m\gamma^2\widetilde{\Gamma}(y/\gamma)}{U_{m,\gamma}(y)}\Bigg||\Psi(y,m,\gamma,\lambda_{m,\gamma})|^2\;dy\lesssim \int_{-1}^{1} \Bigg| \dfrac{m\gamma \widetilde{\Gamma}(y/\gamma)}{y/\gamma}\Bigg||\Psi(y,m,\gamma,\lambda_{m,\gamma})|^2\;dy \lesssim m \gamma.
\end{aligned}
\end{equation}
Next is an estimate for $|S_2|$,
\begin{equation}
\begin{aligned}
|S_2| &\leq m \int_{-1/\gamma}^{1/\gamma} \sigma(y) \bigg|\dfrac{\gamma y}{U(\gamma y)}-\dfrac{1}{U'(0)}\bigg| |\Psi(\gamma y,m,\gamma,\lambda_{m,\gamma})|^2\;dy\\&\quad
+m \int_{-1/\gamma}^{1/\gamma} \sigma(y)\bigg||\Psi(\gamma y,m,\gamma,\lambda_{m,\gamma})|^2-|\Psi(0,m,\gamma,\lambda_{m,\gamma})|^2\bigg|\dfrac{1}{U'(0)} \;dy\\
&=:\mathbb{I}+\mathbb{II}.
\end{aligned}
\end{equation}
Observe that
\[
\begin{aligned}
 \mathbb{I}&=m\int_{-1/\gamma}^{1/\gamma}\sigma(y) \bigg|\dfrac{\gamma y}{U(\gamma y)}-\dfrac{1}{U'(0)}\bigg| |\Psi(\gamma y,m,\gamma,\lambda_{m,\gamma})|^2\;dy \\
&\lesssim m \int_{-1/\gamma}^{1/\gamma} \dfrac{\sigma(y) |\Psi(\gamma y,m,\gamma,\lambda_{m,\gamma})|^2 \gamma^3 y^3}{\gamma y}\;dy\\
&\lesssim m \int_{-1/\gamma}^{1/ \gamma} \sigma(y) |\Psi(\gamma y,m,\gamma,\lambda_{m,\gamma})|^2 \gamma^2 y^2\;dy\lesssim m \gamma^2\lesssim m \gamma.
\end{aligned}
\] Additionally,
\[
\begin{aligned}
\mathbb{II}&=m \int_{-1/\gamma}^{1/\gamma} \sigma(y)\bigg||\Psi(\gamma y,m,\gamma,\lambda_{m,\gamma})|^2-|\Psi(0,m,\gamma,\lambda_{m,\gamma})|^2 \bigg|\dfrac{1}{U'(0)} \;dy \\
&\lesssim m \int_{-1/\gamma}^{1/\gamma} \sigma(y) \bigg|\Psi(\gamma y,m,\gamma,\lambda_{m,\gamma})+\Psi(0,m,\gamma,\lambda_{m,\gamma}) \bigg| \bigg|\Psi(\gamma y,m,\gamma,\lambda_{m,\gamma})-\Psi(0,m,\gamma,\lambda_{m,\gamma}) \bigg|\;dy\\
        &\lesssim m  \int_{-1/\gamma}^{1/\gamma} \sigma(y) \norm{\Psi(\gamma y,m,\gamma,\lambda_{m,\gamma})}_{L^{\infty}}\dfrac{1}{\gamma y} \Bigg(\int_{0}^{\gamma y} |\Psi'(z,m,\gamma,\lambda_{m,\gamma})|\;dz\Bigg) \gamma y\; dy\\
        &\lesssim m \norm{\sigma(y)}_{L^{\infty}} \norm{\Psi(y,m,\gamma,\lambda_{m,\gamma})}_{L^{\infty}}\int_{-1/\gamma}^{1/\gamma}\dfrac{1}{\gamma y}\int_{0}^{\gamma y} \Bigg(|\Psi'(z,m,\gamma,\lambda_{m,\gamma})|\;dz\Bigg) \gamma y\; dy\\&
        \lesssim m \gamma.                                                    
\end{aligned}
\] Hence, combining the two estimates yields 
$|S_2| \lesssim m \gamma$.

Further, it is easy to see that $|S_3|$ satisfies the following inequality
\[
\begin{aligned}
|S_3|= m\bigg|\int_{|y| \geq 1/ \gamma} \sigma(y) \dfrac{|\Psi(0,m,\gamma)|^2}{U'(0)}\; dy\bigg| \lesssim m \gamma.
\end{aligned}
\]
Lastly, it remains to obtain an estimate on $|S_4|$,
\[
\begin{aligned}
|S_4|&= m \bigg|\int_{-1}^{1} \dfrac{\sigma(y/\gamma)}{\gamma} \dfrac{y^2}{U^2(y)} \dfrac{(\gamma m \widetilde{\Gamma}(y/\gamma)/(y/\gamma))}{1+\frac{y}{U(y)}\gamma m \widetilde{\Gamma}(y/\gamma)/(y/\gamma)}\; dy\bigg|\\
         &\leq m  \bigg|\int_{-1/ \gamma}^{1/ \gamma} \sigma(y) \dfrac{y^2\gamma^2}{U^2(y\gamma)} \dfrac{(\gamma m \widetilde{\Gamma}(y)/(y))}{1+\frac{y\gamma}{U(y\gamma)}\gamma m \widetilde{\Gamma}(y)/(y)}\; dy\bigg| \lesssim m \gamma \bigg|\int_{\mathbb{R}} \sigma(y) \;dy\bigg| \approx m \gamma.
\end{aligned}
\] Note that to get the estimate for $|S_4|$,  we have used the fact $\widetilde{\Gamma}(z)/(z)$ is bounded.
 Hence, we can conclude that 
\begin{equation}\label{upper bound m gamma}
\lambda_{m,0}-\lambda_{m,\gamma}\lesssim m\gamma.
\end{equation}

 The other direction of the estimate above, namely $\lambda_{m,\gamma}-\lambda_m \lesssim m \gamma$ can be derived similarly as follows:
 \begin{align*}
 \lambda_{m,\gamma} &\leq H_{m,\gamma}\Psi(y,m,\lambda_{m,0})(y)\\
                    &= H_{m,0}\Psi(y,m,\lambda_{m,0})(y)+\int_{-1}^{1}\Bigg(\dfrac{U''(y)}{U_{m,\gamma}(y)}-\dfrac{U''(y)}{U(y)}\Bigg)|\Psi(y,m,\lambda_{m,0})|^2\;dy\\
                     &\quad+\int_{-1}^{1} \dfrac{m}{\gamma} \sigma(y/\gamma) \dfrac{1}{\frac{U(y)}{y}+\gamma m \widetilde{\Gamma}(y/\gamma)/(y/\gamma)}|\Psi(y,m,\lambda_{m,0})|^2 \; dy
                    +\dfrac{m |\Psi(0,m,\lambda_{m,0})|^2} {U'(0)} \\
                    &= \lambda_{m,0}+\int_{-1}^{1}\Bigg(\dfrac{U''(y)}{U_{m,\gamma}(y)}-\dfrac{U''(y)}{U(y)}\Bigg)|\Psi(y,m,\lambda_{m,0})|^2\;dy\\
                    &\quad+m \int_{-1/\gamma}^{1/\gamma} \sigma(y) \bigg(\dfrac{\gamma y}{U(\gamma y)}|\Psi(\gamma y,m,\lambda_{m,0})|^2-\dfrac{|\Psi(0,m,\lambda_{m,0})|^2}{U'(0)}\bigg)\;dy \\&\quad +m\int_{|y| \geq 1/ \gamma} \sigma(y) \dfrac{|\Psi(0,m,\lambda_{m,0})|^2}{U'(0)}\; dy -m \int_{-1}^{1} \dfrac{\sigma(y/\gamma)y^2}{\gamma U^2(y)}  \dfrac{\gamma m \widetilde{\Gamma}(y/\gamma)/(y/\gamma)|\Psi(y,m,\lambda_{m,0})|^2}{1+\frac{y}{U(y)}\gamma m \widetilde{\Gamma}(y/\gamma)/(y/\gamma)}\; dy\\
                    &=:\lambda_{m,0} +\widetilde{S}_1+\widetilde{S}_2+\widetilde{S}_3+\widetilde{S}_4.
 \end{align*}
 As before, we can also obtain the bounds for all $\widetilde{S}_i$ for $i=1,2,3,4$, which eventually lead to the inequality 
 \begin{equation}\label{lower bound m gamma}
     \lambda_{m,\gamma}-\lambda_{m,0}\lesssim m \gamma.
 \end{equation} Finally, combining both estimates in \eqref{upper bound m gamma} and \eqref{lower bound m gamma} yields
 \begin{equation}\label{distance between eigenvalues}
      |\lambda_{m,0}-\lambda_{m,\gamma}|\lesssim m \gamma.
 \end{equation}
 The proof is then complete. \qedhere
\end{proof}

In addition to the estimate in Lemma~\ref{distance between minimal eigenvalues}, we also need a monotonicity property of the minimal eigenvalues of both operators $\mathcal{H}_{\widetilde{m},\gamma}$ and $\mathcal{H}_{\bar{m},\gamma}$ for $0<\widetilde{m}<\bar{m}$. This is the content of the next lemma. 
\begin{lemma}\label{monotonicity of lambda m gamma}
For any $0<\widetilde{m} <\bar{m}$ and $\gamma>0$, we have
\[
\lambda_{\bar{m},\gamma}-\lambda_{\widetilde{m},\gamma}<0.
\]
\end{lemma}
\begin{proof}
      Via the quadratic form \eqref{quadratic form 2}, we deduce the following inequality
 \begin{equation}\label{lambda m tilde gamma}
 \begin{aligned}
     \lambda_{\bar{m},\gamma} &\leq \lVert \Psi'(y,\widetilde{m},\gamma) \rVert_{L^2}^2 + \int_{-1}^{1} Q_{\bar{m},\gamma} |\Psi(y,\widetilde{m},\gamma)|^2   \; dy\\
     &=\lambda_{\widetilde{m},\gamma}+\int_{-1}^1\bigg(\dfrac{U''_{\bar{m},\gamma}(y)}{U_{\bar{m},\gamma}(y)}-\dfrac{U''_{\widetilde{m},\gamma}(y)}{U_{\widetilde{m},\gamma}(y)}\bigg)|\Psi(y,\widetilde{m},\gamma)|^2  \; dy.
 \end{aligned}
 \end{equation}
 
 Let us focus on the integrand on the right hand side of \eqref{lambda m tilde gamma}. Notice that $\Psi^2(y,\bar{m},\gamma)$ is positive. Therefore, it suffices to check the sign of the remaining term in the integrand, that is $Q_{\bar{m},\gamma}(y)-Q_{\widetilde{m},\gamma}(y)$. We claim that this difference is strictly negative. This is equivalent to proving the derivative of $Q_{m,\gamma}(y)$ with respect to $y$ is uniformly bounded above by a negative constant. Indeed, via standard computations, we obtain
 \begin{equation}
 \begin{aligned}
     \partial_m(\dfrac{U''_{m,\gamma}(y)}{U_{m,\gamma}(y)})&=\dfrac{U(y)\Gamma'(y/\gamma)-U''(y)\gamma^2 \widetilde{\Gamma}(y/\gamma)}{U^2_{m,\gamma}(y)}= \dfrac{U(y)y\bigg(\frac{\Gamma'(y/\gamma)
}{y}-\frac{U''(y)}{U(y)}\gamma^2\frac{1}{y}\widetilde{\Gamma}(y/\gamma)\bigg)}{U^2_{m,\gamma}(y)}\\
&\lesssim \bigg(\dfrac{-C_0}{\gamma}+C_{1}\gamma^2 \lVert \Gamma \rVert_{L^{\infty}}\bigg) \lesssim \dfrac{-C_0}{\gamma}
 \end{aligned}
 \end{equation}
 for $\gamma>0$ small enough. Thus, $\partial_m(\dfrac{U''_{m,\gamma}}{U_{m,\gamma}})<0$ which then implies
 \[
 \bigg(\dfrac{U''_{\bar{m},\gamma}}{U_{\bar{m},\gamma}}-\dfrac{U''_{\widetilde{m},\gamma}}{U_{\widetilde{m},\gamma}}\bigg)<0,
 \]
 for sufficiently small $\gamma>0$.
 This inequality in concert with the one in \eqref{lambda m tilde gamma} lead us to
 \begin{equation}\label{inequality lambda m gamma}
     \lambda_{\bar{m},\gamma}-\lambda_{\widetilde{m},\gamma}<0.
 \end{equation}
 This completes the proof. \qedhere
\end{proof}

In the next three remarks, we provide more discussions on the Wronskian. 
\begin{remark}
We would like to mention that we have an explicit formula for the Wronskian for $c \in \mathfrak{B} \cap \mathfrak{D}_{\epsilon_0}\setminus \mathbb{R}$. Observe that we obtain a representation formula for each $\varphi^{-}$ and $\varphi^{+}$ mentioned in \eqref{wronskian of generic varphi}
\[
\begin{aligned} \label{varphis}
  &\varphi^{-}(y,m,c,\lambda)=\phi(-1,m,c,\lambda)\phi(y,m,c,\lambda) \int_{-1}^{y} \dfrac{1}{\phi^2(z,m,c,\lambda)}dz,\\
  &\varphi^{+}(y,m,c,\lambda)=\phi(1,m,c,\lambda)\phi(y,m,c,\lambda) \int_{1}^{y} \dfrac{1}{\phi^2(z,m,c,\lambda)}dz,
\end{aligned}
\]
where $\phi(y,m,c,\lambda)$ is defined  as in Lemma \ref{phi1 phi2}. It is also not hard to see that  $\varphi^{-}$ and $\varphi^{+}$ satisfy the boundary conditions \eqref{boundary conditions for varphi}.
By \eqref{wronskian of generic varphi}, we know that the Wronskian of $\varphi^{-}$ and  $\varphi^{+}$ is given by
\begin{equation}\label{wronskian of varphi}
\mathcal{W}[\varphi^{-},\varphi^{+}](m,\gamma,c,\lambda)= \phi(-1,m,c,\lambda)\phi(1,m,c,\lambda) \int_{-1}^{1} \dfrac{1}{\phi^2(y,m,c,\lambda)}\;dy.
\end{equation}
\end{remark}

 One can verify from the definition of $\phi(y,m,c,\lambda)$ that both terms $\phi(-1,m,c,\lambda)$ and $\phi(1,m,c,\lambda)$ in the Wronskian $\mathcal{W}$ are away from zero for all $c\in \mathfrak{B} \cap \mathfrak{D}_{\epsilon_0}\setminus \mathbb{R}$. This leads us to state the next remark.
\begin{remark}\label{equivalent condition for c eval and wronskian being 0}
    Let $c\in \mathfrak{B} \cap \mathfrak{D}_{\epsilon_0}\setminus \mathbb{R}$. the constant $c$ is an eigenvalue of operator $\mathcal{R}_{m,\gamma,k}$ \eqref{operator R} if and only if 
    \[
    W(m,\gamma,c,\lambda):=\int_{-1}^{1} \dfrac{1}{\phi^2(y,m,\gamma,c,\lambda)}\;dy=0.
    \]
More precisely, this gives us a criterion to decide whether or not $c$ is an eigenvalue of $\mathcal{R}_{m,\gamma,k}$. That is to say one only needs to check if $W$ is zero at $c$. For the rest of the paper $W$ will be called ``modified Wronskian".

 One may observe that the Wronskian $\mathcal{W}$ and modified Wronskian $W$ share the same zeros. However, it is important to note that although  $\mathcal{W}[\varphi^{-},\varphi^{+}](m,\gamma,c,\lambda)=\varphi^{-}(1,m,\gamma,c,\lambda)$ is analytic in $c$, the modified Wronskian $W$ may not be so for some $c$. Therefore, when the analyticity  is needed in our argument, we shall work with $\mathcal{W}$, otherwise it suffices to only work with the modified Wronskian $W$.
\end{remark}

\begin{remark}\label{continuity of wronskian}
    It is proved in \cite[Lemma 6.3]{WeiZhangZhao2015} (also see \cite{limasmoudizhao2022}) that for any given $m,\gamma$ (namely fixed shear flow) and $\lambda$, the Wronskian $\mathcal{W}[\varphi^{-},\varphi^{+}](m,\gamma,c,\lambda)$ is one-sided continuous to the boundary $\mathrm{Ran}\, U_{m,\gamma}$ from both $\mathfrak{Im}\, c>0$ and $\mathfrak{Im}\, c<0$. Namely, there are real functions $A(m,\gamma,\mathfrak{Re}\,c,\lambda)$ and $B(m,\gamma,\mathfrak{Re}\,c,\lambda)$ such that 
    \begin{align*}
    \lim_{\mathfrak{Im}\, c\to 0_{\pm}}
    \mathcal{W}[\varphi^{-},\varphi^{+}](1,m,\gamma,c,\lambda)
    =A(m,\gamma,\mathfrak{Re}\,c,\lambda)\mp iB(m,\gamma,\mathfrak{Re}\,c,\lambda)
    \end{align*}
    where 
    \begin{equation*}
        \begin{aligned}
        &B(m,\gamma, \mathfrak{Re}\,c,\lambda)\\
        &=\frac{U''_{m,\gamma}(U^{-1}_{m,\gamma}(\mathfrak{Re}\, c))}{U'_{m,\gamma}(U^{-1}_{m,\gamma}(\mathfrak{Re}\, c))^3}(U_{m,\gamma}(-1)-\mathfrak{Re}\, c)(U_{m,\gamma}(1)-\mathfrak{Re}\, c)\phi_1(1,m,\mathfrak{Re}\, c,\lambda)\phi_1(-1,m,\mathfrak{Re}\,c,\lambda)
        \end{aligned}
    \end{equation*}
    It is easy to check that $B(m,\gamma,0,\lambda)=0$, which means that $\mathcal{W}[\varphi^{-},\varphi^{+}](1,m,c,\lambda)$ is continuous at $0$. 

    In the same paper \cite{WeiZhangZhao2015}, the authors also proved the following statement:
    The pure imaginary number $ik c$ with $0\neq k\in \mathbb{Z}$ and $c\in \mathrm{Ran}\, U_{m,\gamma}$ is an embedded eigenvalue of the Rayleigh operator namely, the Rayleigh equation \eqref{Rayleigh Equation} with $c\in \mathrm{Ran}\, U_{m,\gamma}$, has nontrivial solutions, if and only if
    \begin{equation}
        A(m,\gamma,c,-k^2)^2+B(m,\gamma,c,-k^2)^2=0.
    \end{equation}
\end{remark}
Having established the tools and machinery for our argument, we now present our proofs. 
\section{Proofs}\label{proof of main results}
This section is devoted to proving both Theorems~\ref{linear} and \ref{nonlinear}. It consists of two parts. The first portion is composed of a sequence of proof, each corresponds to the regime when $m<m_*, m=m_*,$ and $m>m_*$. Following that would be the proof of Theorem~\ref{nonlinear}.

Let us now establish a simple yet crucial fact for our proofs. Recall in \eqref{expression of m}, we know that for a given $\lambda$, one can find $m$ such that $\lambda$ is the corresponding minimal eigenvalue of $\mathcal{H}_{m,0}$. Let $\lambda_{m_0,0}=-1$ be the minimal eigenvalue of $\mathcal{H}_{m_0,0}$ (in other words for $m=m_0$). We shall use this fact as a building block in our arguments.

We are now in the position of proving Theorem~\ref{linear}~\ref{Neutral mode}.
\begin{proof}[\textbf{Proof of Theorem~\ref{linear}\ref{Neutral mode}}]
By using inequality \eqref{distance between eigenvalues} we have just derived, for $m=m_0$, there exists $\gamma_0>0$ such that 
\[
|\lambda_{m_0,\gamma}-\lambda_{m_0,0}|\lesssim m_0 \gamma,
\] 
for all $\gamma \in (0,\gamma_0)$.

Let $\gamma_* \in (0,\gamma_0)$. Thanks to the expression of $m$ in \eqref{expression of m} and its monotonicity in  \eqref{partial lambda m} , there exists $m=m_1$ such that $\lambda_{m_1,0}<-1-Cm_1\gamma_*$ with $m_1>m_0$. Moreover, via inequality \eqref{distance between eigenvalues}, we know that 
\[
\lambda_{m_1,\gamma_*}\leq Cm_1\gamma_* +\lambda_{m_1,0}<-1.
\]
By similar argument, let $m_2<m_0$. Again, using the formula of $m$ in \eqref{expression of m} and the monotonicity of $m$ in \eqref{partial lambda m}, we can choose $m=m_2$, such that $\lambda_{m_2,0}>-1+Cm_2\gamma_*$. Further, via inequality \eqref{distance between eigenvalues} we have
\[
\lambda_{m_2,\gamma_*}\geq \lambda_{m_2,0}-Cm_2\gamma_*>-1. 
\]
Since, $\lambda_{m_2,\gamma_*}>-1$ and $\lambda_{m_1,\gamma_*}<-1$, then there exists $m=m_*$ such that $m_1>m_*>m_2$ and
\[
\lambda_{m_*,\gamma_*}=-1.
\]
This implies that $c=0$ is the an embedded eigenvalue of the Rayleigh operator. Since there is no other inflection point of $U_{m_*,\gamma_*}$ apart from $y=0$, then $c=0$ is the only embedded eigenvalue of $\mathcal{L}_{m_*,\gamma_*}$. This proves part\ref{Neutral mode}.
\end{proof}
Next, we prove part~\ref{Stability} of Theorem~\ref{linear}. 
\begin{proof}[\textbf{Proof of Theorem~\ref{linear}\ref{Stability}}]
Observe that for any $m<m_*$, where $m_*$ is the one mentioned before, we have
\begin{equation} \label{spectra condition for m<m_*}
\lambda_{m,\gamma_*}>-1=\lambda_{m_*,\gamma_*}.
\end{equation}
But we know that $-k^2\leq-1$ which means  $c=0$ cannot be an embedded eigenvalue of $\mathcal{R}_{m,\gamma,k}$ for $m<m_*$. Via Lemma~\ref{monotonicity of lambda m gamma}, the nonexistence of embedded eigenvalue also holds true for any $m<m_*$ and $\gamma\in (0,\gamma_0)$. 

To that end, we look at the following Rayleigh equation for $k\neq 0$,
\begin{equation}\label{RE}
\Big(U_{m,\gamma_*}-c\Big)(\partial^2_y-k^2)\phi-U''_{m,\gamma_*} \phi=0.
\end{equation} We would like to note again that for $k>C_0$, where $C_0$ some generic constant, the Wronskian $\mathcal{W}(m,\gamma_*,c,-k^2) \neq 0$ for all $c$. Hence, one can apply the same argument below even for all $C_0>|k|\geq 1$.  We show that the Rayleigh equation contains no nontrivial solution $\phi\in H^1_0(-1,1).$ This is equivalent to proving that the Wronskian $\mathcal{W}(m,\gamma_*,c,-k^2)\neq 0$. 

Suppose that \eqref{RE} has a nontrivial solution for some $c$. We define the following set $\mathcal{M}_k$ and its infimum
\begin{align*}
\mathcal{M}_k:&=\{m: \exists\, c \in S\setminus R,\  \mathcal{W}(m,\gamma_*,c,-k^2)=0\}
, \quad m_{k,\infty}:=\inf_{m\in \mathcal{M}_k}m,
\end{align*}
where $R=\bigcup\limits_{m<m_*}[U_{m,\gamma_*}(-1),U_{m,\gamma_*}(1)]=\mathrm{Ran}\, U_{m,\gamma_*}$ and
$$
S=\left\{c\in \mathbb{C}:~\mathrm{dist}\,(c,R)\leq d_0\right\},
$$
with \[d_0=\sup\limits_{m<m_*}\frac{U_{m,\gamma_*}(1)-U_{m,\gamma_*}(-1)}{2}.\]
Note that $\mathcal{W}(m,\gamma_*,c,-k^2)\neq 0$ for $c\notin S$ for all $k$.

Since $\mathcal{M}_k=\emptyset$ for large $k$, we only focus on a finite number of $k$'s. The following argument can be done for each $k$. So we drop the $k$ in the notations and let $\mathcal{W}(m,c)=\mathcal{W}(m,\gamma_*,c,-k^2)$, $W(m,\gamma_*,c)=W(m,\gamma_*,c,-k^2)$,  $\mathcal{M}=\mathcal{M}_k$, and $m_{\infty}=m_{k,\infty}$ for convenience. 
Our strategy is to show that $\mathcal{M}$ is open and closed in the natural topology. As a consequence, it has to be either the empty set or the whole real line. The later one cannot be true since $0\notin \mathcal{M}$, hence it has to be empty. 

Let us start by showing the open property of $\mathcal{M}$. This can be done by proving that for any $m\in \mathcal{M}$, there exists $\delta>0$ such that for all $x \in B_{\delta}(m)$ we have $x \in \mathcal{M}$.
By way of contradiction, we suppose that $\mathcal{M}$ is not open. Then there is $m\in\mathcal{M}$, such that for all $\delta>0$ there exists $x \in B_{\delta}(m)$ such that $x \notin \mathcal{M}$. Let $c_0$ be one zero of $\mathcal{W}(m,\gamma_*,c)$. Next, we focus on the closed neighborhood of $c_0$: $N_{\delta_1}(c_0)=\{z\in S\setminus R:~\mathrm{dist}(c_0,z)\leq \delta_1\}$ with $\delta_1$ fixed depending on $m$ so that $\mathcal{W}(m,\gamma_*,c) \neq 0$ for all $c \in \partial N_{\delta_1}(c_0)$. By the continuity of $\mathcal{W}(x,\gamma_*,c)$ in $c$ and $x$,  there exists $\epsilon$ small enough such that for all $|x-m|<\epsilon$, 
\[
\begin{aligned}
   |\mathcal{W}(m,\gamma_*,c)-\mathcal{W}(x,\gamma_*,c)|<|\mathcal{W}(m,\gamma_*,c)|\quad \text{for any} \quad c\in \partial N_{\delta_1}(c_0).
\end{aligned}
\] 
However, by Rouche's theorem, we know therefore that the number of zeros of $\mathcal{W}(m,\gamma_*,c)$ is equal to that of $\mathcal{W}(x,\gamma_*,c)$ in $N_{\delta_1}(c_0)$. However, $c_0$ is a zero of $\mathcal{W}(m,\gamma_*,c)$, but $W(x,\gamma_*,c) \neq 0$ in $N_{\delta_1}(c_0)$. Hence, we arrive at a contradiction. This proves that $\mathcal{M}$ has to be open.

It remains, therefore, to show that $\mathcal{M}$ is closed. Consider the sequence $\{(c_n,m_n)\}$ with 
\[\mathcal{W}(m_n,\gamma_*,c_n)=0=W(m_n,\gamma_*,c_n).
\] Recycling the same notation, there is a subsequence denoted by $(c_n,m_n)$ such that $(c_n,m_n) \to (c_\infty, m_{\infty})$ as $n\to \infty$. Since we are studying the limit of the subsequence $c_n$ to a real number, it suffices to only work with $c_n \in \mathfrak{B}\cap \mathfrak{D}_{\epsilon_0}$. This allows us to use the Wronskian $\mathcal{W}$. We claim that  $c_\infty$ does not belong to $R$. This then allows us to say that $c_\infty$ must belong to $S \setminus R$.

By way of contradiction, suppose that $c_{\infty} \in R$. Recall that, via the result in Lemma~\ref{phi1 phi2}, we know that the Rayleigh equation \eqref{RE} admits a regular solution of the form (by taking $\lambda=-k^2$)
    \[
     \phi(y,m,\gamma_*,c_n)=(U_{m,\gamma_*}(y)-c_n){\phi}_1(y,m,\gamma_*,\mathfrak{Re}c_n)\phi_2(y,m,\gamma_*,c_n). 
    \]
    We also would like to mention the following identity:
    \[
    \begin{aligned}
        \lim_{(m_n,c_n)\to(m_\infty,c_\infty)}\mathcal{W}(m_n,\gamma_*,c_n)&=\lim_{c_n\to c_\infty}\lim_{m_n\to m_\infty}\mathcal{W}(m_n,\gamma_*,c_n)\\&
        =\lim_{c_n\to c_\infty}\mathcal{W}(m_\infty,\gamma_*,c_n).
    \end{aligned}
    \]
    For fixed $m=m_\infty$, by Remark~\ref{continuity of wronskian}, we know that 
    \[
    \mathcal{W}(m_\infty,\gamma_*,c_n)\to A(m_\infty,\gamma_*,c_\infty)\pm iB(m_\infty,\gamma_*,c_\infty),\quad \text{as } \mathfrak{Im}\; c_{n}\to 0. 
    \]
    However, due to the continuity of $\mathcal{W}$ with respect to $c$ and the fact that $\mathcal{W}(m_n,c_n)=0$, we have $\mathcal{W}(m_\infty,c_\infty)=0$. This implies that
    \[
    A^2(m_{\infty},\gamma_*,c_\infty)+B^2(m_{\infty},\gamma_*,c_\infty)=0.
    \]
    By the limit identity above, we conclude that 
    \[
    \lim_{c_n\to c_\infty}\mathcal{W}(m_\infty,\gamma_*,c_n)=0.
    \]
    Hence, $c_\infty$ must be an embedded eigenvalue of the operator $\mathcal{R}_{m_\infty,\gamma_*,k}$. This contradicts the fact that the operator $\mathcal{R}_{m,\gamma_*,k}$ has no embedded eigenvalue for $m<m_*$. Hence, $c_\infty \notin R$. Hence, $c_{\infty} \in S\setminus R$. Therefore, $\mathcal{M}$ must be  closed. Using the openness and closedness property of $\mathcal{M}$ along with the fact that $0\notin \mathcal{M}$, we conclude that $\mathcal{M}$ has to be empty. The proof is now complete.\qedhere
\end{proof}

\begin{proof}[\textbf{Proof of Theorem~\ref{linear}\ref{Instability} and estimate \eqref{estimate for lambda r}}]
 From the proof of Theorem~\ref{linear}~\ref{Neutral mode}, we know that  there exist $m_*$ and $\gamma_*$ such that $\lambda_{m_*,\gamma_*}=-1$. This fact in combination with Lemma~\ref{monotonicity of lambda m gamma} tells us that for any $m>m_*$, we have $\lambda_{m,\gamma_*}<\lambda_{m_*,\gamma_*}=-1$ which together with the existence of growing mode result for the $\mathcal{K}^+-$flow proved by Lin \cite{Lin2003} gives the existence of the eigenvalues of $\mathcal{R}_{m,\gamma,1}$. In the present work, we give a new proof of such existence which helps us understand how the growing modes emerge with respect to the parameter $m$. This is in fact the main motivation on why we work with the Wronskian $\mathcal{W}$. By Remark~\ref{TFAE statement}, it reduces to finding the zeros of $\mathcal{W}$. If a given zero is an interior point, one would use the Implicit Function Theorem (IFT) to construct a curve of solutions branching out from such interior point. However, the given zero $(0,0)=(c_r(m_*),c_i(m_*))$ here is at the boundary of the domain $\{(c_r,c_i):c_i\geq 0\}$. This obstructs a direct use of IFT. As an alternative, we first use an ODE argument to find an interior zero point in the aforementioned domain. Following that, we then apply the IFT to show the existence of a curve of the zeros of $\mathcal{W}$ originating from the interior zero point. Finally, we derive some estimates to guarantee that the constructed curve will not loop back and touch the boundary.
Indeed we will show that for a given $m>m_*$, the zero of $\mathcal{W}$ belongs to the set
\begin{equation}\label{set sigma}
    \Sigma:=\{(c_r,c_i):|c_r|\leq C^2 \gamma \epsilon_0  \text{ and } 0\leq c_i \leq C \gamma \epsilon_0\},
\end{equation}
for sufficiently small $\gamma$ and $\epsilon_0$. A direct calculation tells us that if the following system of differential equations
\begin{equation}\label{ODE}
\begin{aligned}
    &-\dfrac{d c_i}{d m}=\dfrac{\partial_{c_i}\mathcal{W}_r\partial_m \mathcal{W}_r+\partial_{c_i}\mathcal{W}_i \partial_m \mathcal{W}_i}{(\partial_{c_i}\mathcal{W}_i)^2+(\partial_{c_i}\mathcal{W}_r)^2}=:-F(y,m,\gamma,c,-1),\\&
    -\dfrac{ d c_r}{dm}=\dfrac{\partial_{c_i}\mathcal{W}_i\partial_m \mathcal{W}_r-\partial_{c_i}\mathcal{W}_r \partial_m \mathcal{W}_i}{(\partial_{c_i}\mathcal{W}_i)^2+(\partial_{c_i}\mathcal{W}_r)^2}=:G(y,m,\gamma,c,-1)
\end{aligned}
\end{equation}
has a solution inside the set $\Sigma$, then we obtain a curve $(c_r(m),c_i(m))$ that starts from $(0,0)=(c_r(m_*),c_i(m_*))$ along which the Wronskian $\mathcal{W}(m,\gamma,c,-1)=0$. Due to the fact that $F$ and $G$ are Lipschitz in $c$, a standard ODE argument gives us the existence of solution to the system above. It now remains to show that such a solution lies inside $\Sigma$. To accomplish that, we need the following  priori estimates: $C^{-1}\gamma\leq F\leq C \gamma$ and $|G|\leq C \gamma$ for $(c_r,c_i)\in \Sigma$. We have divided our computations in three steps.

As a starter, we define the coupling between $\phi_1$ and $\phi_2$ as follows
\[
\phi_1^{E}(y,m,\gamma,c,\lambda):=\phi_1(y,m,\gamma,c_r,\lambda)\phi_2(y,m,\gamma,c,\lambda),
\]
where $\phi_1$ and $\phi_2$ are as stated in Lemma~\ref{phi1 phi2}.
Hence, the modified Wronskian takes the form
 \begin{equation}\label{modified wronskian at cr 0}
     W(m,\gamma,c,-1)=\underbrace{\int_{-1}^1\dfrac{1}{(U_{m,\gamma}(y)-c)^2}\;dy}_{\text{I}} + \underbrace{\int_{-1}^1\dfrac{1}{(U_{m,\gamma}(y)-c)^2} \Bigg( \dfrac{1}{(\phi_1^E)^2(y,m,\gamma,c,-1)}-1\Bigg)\;dy}_{\text{II}},
\end{equation}
where $c \notin \text{Ran}(U_{m,\gamma}).$

As $W$ is a complex-valued function, it is natural to recast it in the form $W=W_{r}+iW_{i}$ where $W_{r}$ is its real part while $W_{i}$ is its imaginary part. More precisely, in light of the splitting in \eqref{modified wronskian at cr 0}, we have
 \[
 W_r=\text{I}_r+\text{II}_r,\quad W_i=\text{I}_i+\text{II}_i \text{ with } \text{I}=\text{I}_r+i\text{I}_i \text{ and } \text{II}=\text{II}_r+i\text{II}_i.
 \]
 For later use, we re-express $\text{I}$ in terms of its real and imaginary parts as follows
\begin{equation}\label{Expression of I}
     \begin{aligned}
     \text{I}&=\dfrac{-1}{(U_{m,\gamma}(y)-c)U'_{m,\gamma}(y)}\bigg|_{-1}^{1}+\int_{U_{m,\gamma}(-1)}^{U_{m,\gamma}(1)}\dfrac{\partial^2_v(U_{m,\gamma}^{-1})(v)}{(v-c_r-ic_i)}\;dv\\&
        =\dfrac{-(U_{m,\gamma}(y)-c_r)}{((U_{m,\gamma}(y)-c_r)^2+c_i^2)U'_{m,\gamma}(y)}\bigg|_{-1}^{1}+\int_{U_{m,\gamma}(-1)}^{U_{m,\gamma}(1)}\dfrac{(v-c_r)\partial^2_v(U_{m,\gamma}^{-1})(v)}{((v-c_r)^2+c_i^2)}\;dv\\&
        \qquad +i\bigg(\int_{U_{m,\gamma}(-1)}^{U_{m,\gamma}(1)}\dfrac{c_i\partial^2_v(U_{m,\gamma}^{-1})(v)}{((v-c_r)^2+c_i^2)}\;dv-\dfrac{c_i}{((U_{m,\gamma}(y)-c_r)^2+c_i^2)U'_{m,\gamma}(y)}\bigg|_{-1}^{1}\bigg)
        \\&
        =:\text{I}_r+i\text{I}_i.
        \end{aligned}
\end{equation}
\textbf{STEP 1:} Here, our ultimate goal is to derive inequalities concerning $\partial_{c_i} W$ for its real and imaginary parts. To achieve that, we compute a series of estimates that are recorded in the lemmas below.
\begin{lemma}[Estimate of $W$]\label{estimate of W}
    The modified Wronskian obeys the following estimate
    \[
    |W|\leq C \text{ for all } (c_r,c_i) \in \Sigma.
    \]
\end{lemma}
\begin{proof}
    From \eqref{Expression of I}, the boundary term in the expression of $\text{I}_r$ is clearly bounded above by some constant $C$. Similarly, the remaining integral term in $\text{I}_r$ can be estimated using Lemma~\ref{Supremum estimate}
    \[
    \begin{aligned}
    \bigg|\int_{U_{m,\gamma}(-1)}^{U_{m,\gamma}(1)}\dfrac{(v-c_r)\partial^2_v(U_{m,\gamma}^{-1})(v)}{((v-c_r)^2+c_i^2)}\;dv\bigg|&\lesssim \norm{\partial^2_v(U_{m,\gamma}^{-1})(v)}_{L^2}^{1/2}\bigg(\norm{\partial^3_v(U_{m,\gamma}^{-1})(v)}_{L^2}^{1/2}+\norm{\partial^2_v(U_{m,\gamma}^{-1})(v)}_{L^\infty}^{1/2}\bigg)\\& \lesssim 1.
    \end{aligned}
    \]
    Hence, $|\text{I}_r|\leq C$. One can repeat the same argument in order to estimate the imaginary part of $\text{I}$ and obtain 
    \[
    |\text{I}_i|\leq C.
    \]
    Furthermore, via the estimates in Lemma~\ref{phi1 phi2}, one can deduce that 
    \[
    |\text{II}| \leq \int_{-1}^1 \dfrac{\min\{c_i,(y-y_c)^2\}}{(y-y_c)^2+(c_i)^2}\;dy\leq C.
    \]
    Combining both estimates from $|\text{I}|$ and $|\text{II}|$ yields the desired inequality for $|W|$. \qedhere 
\end{proof}
\begin{lemma}[Estimate of $\partial_{c_i}\text{II}$]\label{estimate partial ci II}
    The partial derivative of \textup{II} with respect to $c_i$ obeys the following estimate
    \[
    |\partial_{c_i}\textup{II}|\leq C.
    \]
\end{lemma}
\begin{proof}
    To obtain the bounds for $\partial_{c_i}\text{II}$, we first define the so-called ``good" derivative
\begin{equation}\label{good derivative}
    \partial_G:=\bigg(\partial_{c_i}+i\dfrac{\partial_y}{U_{m,\gamma}'(y_c)}\bigg),
\end{equation}
where $y_c=U_{m,\gamma}^{-1}(c_r)$.
With this in mind, in terms of the good derivative we can express $\partial_{c_i}$ as follows
\[
 \partial_{c_i}:=\partial_G -i\dfrac{\partial_y}{U_{m,\gamma}'(y_c)}.
\]

Let us now derive an estimate for $\partial_{c_i}\text{II}$. In order to accomplish this, we need to use the results in Lemma~\ref{phi1 phi2}-Lemma~\ref{partial m phi1}. It is not hard to show that
\[
\begin{aligned}
\partial_{c_i}\text{II}&=\text{II}'_1+\text{II}'_2-i\dfrac{1}{U'_{m,\gamma}(y_c)}\Bigg( \dfrac{1}{(U_{m,\gamma}(y)-c_r-ic_i)^2} \bigg( \dfrac{1}{(\phi_1^E)^2(y,m,\gamma,c_r+ic_i,-1)}-1\bigg)\Bigg)\Bigg|_{-1}^{1},
\end{aligned}
\]
where
\[
\begin{aligned}
&\text{II}'_1=\int_{-1}^1\bigg( \dfrac{1}{(\phi_1^E)^2(y,m,\gamma,c_r+ic_i,-1)}-1\bigg)\Bigg[\partial_G \bigg( \dfrac{1}{(U_{m,\gamma}(y)-c_r-ic_i)^2} \bigg)\Bigg]\;dy, \\
&\text{II}'_2=\int_{-1}^1\bigg( \dfrac{1}{(U_{m,\gamma}(y)-c_r-ic_i)^2} \bigg)\Bigg[\partial_G \bigg( \dfrac{1}{(\phi_1^E)^2(y,m,\gamma,c_r+ic_i,-1)}-1\bigg)\Bigg]\;dy.
\end{aligned}
\]

The boundary term in the expression of $\partial_{c_i}\text{II}$  is finite. In other words,
\[
\bigg| \dfrac{1}{U'_{m,\gamma}(y_c)}\Bigg( \dfrac{1}{(U_{m,\gamma}(y)-c_r-ic_i)^2} \bigg(\dfrac{1}{(\phi_1^E)^2(y,m,\gamma,c_r+ic_i,-1)}-1\bigg)\Bigg)\bigg|_{-1}^{1}\bigg| < C.
\]
The term in the integrand of $\text{II}'_1$ involving $\partial_G$ can be estimated as follows
\[
\begin{aligned}
\bigg|\partial_G \bigg( \dfrac{1}{(U_{m,\gamma}(y)-c_r-ic_i)^2}\bigg)\bigg|&=\bigg|\dfrac{2\bigg(-i+i\dfrac{U'_{m,\gamma}(y)}{U'_{m,\gamma}(y_c)}\bigg)}{(U_{m,\gamma}(y)-c_r-ic_i)^{3}}\bigg|\\&\leq C\dfrac{|y-y_c|\lVert U''_{m,\gamma}(y)\rVert_{L^{\infty}}}{((U_{m,\gamma}(y)-c_r)^2+(c_i)^2)^{3/2}}\\&\leq C\dfrac{\lVert U''_{m,\gamma}(y)\rVert_{L^{\infty}}}{(y-y_c)^2+(c_i)^2}.
\end{aligned}
\]

Via the above inequality along with estimates in Lemma~\ref{phi1 phi2}, we obtain
\[
|\text{II}'_1| \leq C \int_{-1}^{1}\dfrac{\min\{1,(y-y_c)^2\}}{(y-y_c)^2+(c_i)^2}\;dy,
\]
from which we can say that $|\text{II}'_1|\leq C$. 

Next, we shall show, similarly, that  $\text{II}'_2$ is also uniformly bounded by a constant independent of $m,\gamma,c_r, c_i$. We begin by obtaining an estimate for $\partial_G((1-(\phi_1^E)^2)/(\phi_1^E)^2)$ which explicitly (after an expansion) reads
\[
\begin{aligned}
\partial_G\Big(\dfrac{1}{(\phi_1^E)^2(y,m,\gamma,c_r+ic_i,-1)}-1\Big)&=\dfrac{-2\partial_G\phi_1^E(y,m,\gamma,c_r+ic_i,-1)}{(\phi_1^E)^3(y,m,\gamma,c_r+ic_i,-1)}\\&
=\dfrac{-2\partial_G\phi_1^E(y,m,\gamma,c_r+ic_i,-1)}{(\phi_1^E(y,m,\gamma,c_r+ic_i,-1))}\dfrac{1}{(\phi_1^E)^2(y,m,\gamma,c_r+ic_i,-1)}\\&
\leq C (y-y_c)^2
\end{aligned}
\]
where we have applied Lemma~\ref{phi1 phi2} and Lemma~\ref{partial G phi1} to obtain the last line in the above inequality. Therefore, we have
\[
|\text{II}'_2|\leq C \int_{-1}^{1}\dfrac{(y-y_c)^2}{(y-y_c)^2+(c_i)^2}\;dy\leq C.
\]   
Combining all the estimates, we can say that 
\begin{equation} \label{partial ci II}|\partial_{c_i}\text{II}|<C.  \qedhere
\end{equation}
\end{proof}

We use the result of the previous lemma to derive an estimate for $\partial_{c_i}W_i$ which is recorded below.

\begin{lemma}[Estimate on $\partial_{c_i}W_i$] \label{estimate partial ci Wi}
    The partial derivative of the Imaginary part of the Wronskian, namely $W_i$, with respect to $c_i$ obeys the following estimate
    \[
    |\partial_{c_i}W_i|\leq \dfrac{C}{\gamma}\text{ for all } (c_r,c_i) \in \Sigma.
    \]
\end{lemma}
\begin{proof}

Differentiating $\text{I}_i$ with respect to $c_i$ yields
 \[
 \begin{aligned}
     \partial_{c_i}\text{I}_i&=\int_{U_{m,\gamma}(-1)}^{U_{m,\gamma}(1)} \partial_{c_i}\Big(\dfrac{c_i}{((v-c_r)^2+c_i^2)}\Big)\partial^2_v(U_{m,\gamma}^{-1})(v)\;dv-\partial_{c_i}\bigg(\dfrac{c_i}{((U_{m,\gamma}(y)-c_r)^2+c_i^2)U'_{m,\gamma}(y)}\bigg|_{-1}^{1}\bigg)\\&
     =\int_{U_{m,\gamma}(-1)}^{U_{m,\gamma}(1)} \partial_{v}\Big(\dfrac{-(v-c_r)}{((v-c_r)^2+c_i^2)}\Big)\partial^2_v(U_{m,\gamma}^{-1})(v)\;dv\\& \qquad \qquad \qquad \qquad \qquad \qquad-\bigg(\dfrac{((U_{m,\gamma}(y)-c_r)^2+c_i^2)U'_{m,\gamma}(y)-2c_i^2(U'_{m,\gamma}(y))}{((U_{m,\gamma}(y)-c_r)^2+c_i^2)^2(U'_{m,\gamma}(y))^2}\bigg|_{-1}^{1}\bigg)\\&
     =\int_{U_{m,\gamma}(-1)}^{U_{m,\gamma}(1)} \dfrac{(v-c_r)}{((v-c_r)^2+c_i^2)}\partial^3_v(U_{m,\gamma}^{-1})(v)\;dv+\dfrac{(U_{m,\gamma}(y)-c_r)}{((U_{m,\gamma}(y)-c_r)^2+c_i^2)}\dfrac{U''_{m,\gamma}(y)}{(U'_{m,\gamma})^3(y)}\bigg|_{-1}^{1}\\&
     \qquad -\bigg(\dfrac{((U_{m,\gamma}(y)-c_r)^2+c_i^2)U'_{m,\gamma}(y)-2c_i^2(U'_{m,\gamma}(y))}{((U_{m,\gamma}(y)-c_r)^2+c_i^2)^2(U'_{m,\gamma}(y))^2}\bigg|_{-1}^{1}\bigg).
 \end{aligned}
 \]
{Applying Lemma~\ref{Supremum estimate}} shows that the first integral in the expression of $\partial_{c_i}\text{I}_i$ obeys the estimate below

\begin{equation}\label{partial ci imaginary I}
\begin{aligned}
\int_{U_{m,\gamma}(-1)}^{U_{m,\gamma}(1)} \dfrac{(v-c_r)\partial^3_v(U_{m,\gamma}^{-1})(v)}{((v-c_r)^2+c_i^2)}\;dv&\leq \sup_{(v,c_i)\in \Sigma}\bigg|\int_{U_{m,\gamma}(-1)}^{U_{m,\gamma}(1)} \dfrac{1}{v-c_r-ic_i}\partial^3_v(U_{m,\gamma}^{-1})(v)\;dv\bigg|\\&\lesssim \norm{\partial^3_v(U_{m,\gamma}^{-1})(v)}_{L^2}^{1/2} \bigg(\norm{\partial^4_v(U_{m,\gamma}^{-1})(v)}_{L^2}^{1/2} + \norm{\partial^3_v(U_{m,\gamma}^{-1})(v)}_{L^\infty}^{1/2} \bigg)\\&
\lesssim \dfrac{1}{\gamma}.
\end{aligned}
\end{equation}

It therefore remains to control the boundary terms in the expression of $\partial_{c_i}\text{I}_i$. It is straightforward to see that they are indeed bounded uniformly in $\gamma$. Hence, gathering all the estimates together results in 
\[
|\partial_{c_i}\textup{I}_i|\lesssim \dfrac{1}
{\gamma}.
\]
Via the estimate in Lemma~\ref{estimate partial ci II}, we conclude that 
\[
|\partial_{c_i} W_i|=|\partial_{c_i} \textup{I}_i+\partial_{c_i} \textup{II}_i|\leq |\partial_{c_i} \textup{I}_i|+|\partial_{c_i} \textup{II}| \leq \dfrac{C}{\gamma}. \qedhere
\]
\end{proof}

Having derived an estimate for $\partial_{c_i}W_i$, we now proceed and present similar estimates but for $\partial_{c_i}W_r$. We record it in the lemma below.
\begin{lemma}[Estimate on $\partial_{c_i} W_r$]\label{estimate partial ci Wr}
    The partial derivative of the real part of the Wronskian, namely $W_r$, with respect to $c_i$ obeys the following estimate
    \[
    \dfrac{-1}{C\gamma}\leq \partial_{c_i}W_r\leq \dfrac{-C}{\gamma}\text{ for all } (c_r,c_i) \in \Sigma.
    \]
\end{lemma}
\begin{proof}

   Notice that \[
 \begin{aligned}
     \partial_{c_i}\text{I}_r&=\partial_{c_i}\bigg(\dfrac{-(U_{m,\gamma}(y)-c_r)}{((U_{m,\gamma}(y)-c_r)^2+c_i^2)U'_{m,\gamma}(y)}\bigg|_{-1}^{1}\bigg)+\int_{U_{m,\gamma}(-1)}^{U_{m,\gamma}(1)}\partial_{c_i}\Big(\dfrac{(v-c_r)}{((v-c_r)^2+c_i^2)}\Big)\partial^2_v(U_{m,\gamma}^{-1})(v)\;dv\\&
     =\bigg(\dfrac{c_i(U_{m,\gamma}(y)-c_r)2U'_{m,\gamma}(y)}{((U_{m,\gamma}(y)-c_r)^2+c_i^2)^2(U'_{m,\gamma}(y))^2}\bigg|_{-1}^{1}\bigg)+\int_{U_{m,\gamma}(-1)}^{U_{m,\gamma}(1)}\partial_{v}\Big(\dfrac{c_i}{((v-c_r)^2+c_i^2)}\Big)\partial^2_v(U_{m,\gamma}^{-1})(v)\;dv\\&
     :=\partial_{c_i}\text{I}^1_r+\partial_{c_i}\text{I}^2_r.
 \end{aligned}
 \]
 It is straightforward to see using the assumption that $c_i \approx \gamma$ and upon evaluation at $y=-1,1$, the boundary term $\partial_{c_i}\text{I}^1_r$ above can be estimated as follows
 \begin{equation}\label{partial ci I 1 r}
     C^{-1}\gamma\leq \partial_{c_i}\text{I}^1_r \leq C\gamma. 
 \end{equation}

Additionally, the term $\partial_{c_i}\text{I}^2_r$ satisfies the inequality 
 \begin{equation}\label{sign of partial c_i real part of I}
\begin{aligned}  &|\partial_{c_i}\text{I}^2_r+\pi \partial_v^3(U^{-1}_{m,\gamma})(c_r)|\\&\leq \int_{|v-c_r|\leq K_0c_i}\bigg|\dfrac{c_i}{(v-c_r)^2+c_i^2}\bigg(\partial^3_v(U_{m,\gamma}^{-1})(v)-\partial^3_v(U_{m,\gamma}^{-1})(c_r)\bigg)\bigg|\;dv\\&
\qquad +\int_{|v-c_r|\geq K_0c_i}\bigg|\dfrac{c_i}{(v-c_r)^2+c_i^2}\partial^3_v(U_{m,\gamma}^{-1})(c_r)\bigg|\;dv\\&
    \qquad +\int_{|v-c_r|\geq K_0c_i}\bigg|\dfrac{c_i}{(v-c_r)^2+c_i^2}\partial^3_v(U_{m,\gamma}^{-1})(v)\bigg|\;dv +\bigg|\dfrac{c_iU''_{m,\gamma}(y)}{((U_{m,\gamma}(y)-c_r)^2+c_i^2)(U'_{m,\gamma}(y))^3}\Big|_{y=-1}^{1}\bigg|\\&
    \leq \bigg(\int_{|v-c_r|\leq K_0c_i}\Big|\dfrac{c_i(v-c_r)}{(v-c_r)^2+c_i^2}\Big|^2\;dv\bigg)^{1/2} \bigg(\int_{|v-c_r|\leq K_0c_i} |\partial^4_v(U_{m,\gamma}^{-1})(v)|^2\;dv\bigg)^{1/2}+C\Big|\frac{\pi}{2}-\arctan(K_0)\Big|\frac{1}{\gamma}\\&
    \qquad 
    +\bigg|\dfrac{c_iU''_{m,\gamma}(y)}{((U_{m,\gamma}(y)-c_r)^2+c_i^2)(U'_{m,\gamma}(y))^3}\Big|_{y=-1}^{1}\bigg|\\
    &\leq C|K_0c_i^2|^{\frac12}\gamma^{-\frac{3}{2}}+C\Big|\frac{\pi}{2}-\arctan(K_0)\Big|\frac{1}{\gamma}+C|c_i|.
\end{aligned}
\end{equation}
Moreover, one may observe that 
\begin{equation*}
\begin{aligned}
&\Big|\partial_v^3U_{m,\gamma}^{-1}(c_r)-\partial_v^3U_{m,\gamma}^{-1}(0) \Big|\leq \Big| \int_{0}^{c_r} \partial_v^4 U_{m,\gamma}^{-1}(v)\; dv\Big|\\& 
\leq |y_c| \norm{\dfrac{-U_{m,\gamma}^{(4)}(y)}{(U'_{m,\gamma}(y))^4}+\dfrac{4U_{m,\gamma}^{'''}(y)U_{m,\gamma}^{''}(y)}{(U_{m,\gamma}^{'}(y))^5}+\dfrac{6U_{m,\gamma}^{''}(y)U_{m,\gamma}^{'''}(y)}{(U_{m,\gamma}^{'}(y))^5}-\dfrac{15(U_{m,\gamma}^{'''}(y))^3}{(U_{m,\gamma}^{'}(y))^6}}_{L^{\infty}}\\&
\leq C \frac{|y_c|}{\gamma^2}.
\end{aligned}
\end{equation*}
Now, recalling the definition of $U_{m,\gamma}$ in \eqref{perturbed shear}, $|y_c|\leq \epsilon_0 \gamma$ with $\epsilon_0 << \gamma$ and choosing $\gamma$ small enough, we can then infer that
\begin{equation}
\begin{aligned}
\Big|\partial_v^3U_{m,\gamma}^{-1}(c_r)-\partial_v^3U_{m,\gamma}^{-1}(0) \Big|&\leq C\frac{\epsilon_0}{\gamma}.
\end{aligned}
\end{equation}

Using the above inequality and \eqref{sign of partial c_i real part of I} along with taking $K_0$ to be sufficiently large and letting $c_i<\epsilon_0 K_0 \gamma$, we can infer that  
\[
   \partial_{c_i}\text{I}^2_r \leq -\pi \partial_v^3U_{m,\gamma}^{-1}(0)+ \pi |\partial_v^3U_{m,\gamma}^{-1}(c_r)-\partial_v^3U_{m,\gamma}^{-1}(0)| \leq -\pi\frac{1+C\epsilon_0}{\gamma} \leq -\frac{C}{\gamma},
\]
\[
\partial_{c_i}\text{I}^2_r \geq -\pi \partial_v^3U_{m,\gamma}^{-1}(0)- \pi |\partial_v^3U_{m,\gamma}^{-1}(c_r)-\partial_v^3U_{m,\gamma}^{-1}(0)| \geq -\pi\frac{1-C\epsilon_0}{\gamma} \leq -\frac{C^{-1}}{\gamma}.
\]
Finally, combining the above estimates with \eqref{partial ci I 1 r}, we conclude that

\begin{equation}\label{partial ci real I}
\dfrac{-C^{-1}}{\gamma}\leq \partial_{c_i}\text{I}_r \leq \dfrac{-C}{\gamma}.
\end{equation}

Via the estimate in Lemma~\ref{estimate partial ci II}, we arrive at
\[
\dfrac{-1}{C \gamma}\leq \partial_{c_i} W_r=\partial_{c_i} \textup{I}_r+\partial_{c_i} \textup{II}_r\leq \dfrac{-C}{\gamma}. \qedhere
\]
\end{proof}

\textbf{Step 2:}
To reach our goal in proving Theorem~\ref{linear}\ref{Instability} and the estimate displayed in \eqref{estimate for lambda r}, it remains to derive estimates for $\partial_m W_r$ and $\partial_m W_i$.  
Recall that for any $(m,c)=(m,c_r+ic_i)$, the modified Wronskian is given by the expression in \eqref{modified wronskian at cr 0}.

In order to get an estimate of $\partial_m W(m,\gamma,
c_r+ic_i,-1)$, we alternatively recast the modified Wronskian in a different way, that is

\begin{equation}\label{partial m W}
  \begin{aligned}
 W(m,\gamma,c_r+ic_i,-1)&= \bigg(\dfrac{-1}{U'_{m,\gamma}(y_c)} \int_{-1}^{1} \dfrac{U'_{m,\gamma}(y)-U_{m,\gamma}'(y_c)}{(U_{m,\gamma}(y)-U_{m,\gamma}(y_c)-ic_i)^2}\;dy\bigg)\\& \quad -\dfrac{1}{U'_{m,\gamma}(y_c)}\bigg[\dfrac{1}{U_{m,\gamma}(1)-U_{m,\gamma}(y_c)-ic_i}-\dfrac{1}{U_{m,\gamma}(-1)-U_{m,\gamma}(y_c)-ic_i}\bigg]\\&
 \quad+ \bigg(\int_{-1}^1 \dfrac{1}{(U_{m,\gamma}(y)-U_{m,\gamma}(y_c)-ic_i)^2}\bigg(\dfrac{1}{(\phi_1^E)^2(y,m,\gamma,c,-1)}-1\bigg)\;dy\bigg).
\end{aligned}  
\end{equation}
Taking a partial derivative with respect to $m$ yields
\[
\partial_m W(m,\gamma,c_r+ic_i,-1)= \mathfrak{W}^1+\mathfrak{W}^2+\mathfrak{W}^3+\mathfrak{W}^4,
\]
where the definition of each term $\mathfrak{W}^i$ with $ i=1,2,3,4,$ is given below
\[
\begin{aligned}
\mathfrak{W}^1&:=\partial_m\bigg(\dfrac{-1}{U'_{m,\gamma}(y_c)} \int_{-1}^{1} \dfrac{U'_{m,\gamma}(y)-U_{m,\gamma}'(y_c)}{(U_{m,\gamma}(y)-U_{m,\gamma}(y_c)-ic_i)^2}\;dy\bigg),\\
\mathfrak{W}^2&:=\partial_m\bigg(\dfrac{1}{U'_{m,\gamma}(y_c)}\bigg)\bigg[\dfrac{1}{U_{m,\gamma}(1)-U_{m,\gamma}(y_c)-ic_i}-\dfrac{1}{U_{m,\gamma}(-1)-U_{m,\gamma}(y_c)-ic_i}\bigg],\\
\mathfrak{W}^3&:=\dfrac{1}{U'_{m,\gamma}(y_c)}\partial_m\bigg[\dfrac{1}{U_{m,\gamma}(1)-U_{m,\gamma}(y_c)-ic_i}-\dfrac{1}{U_{m,\gamma}(-1)-U_{m,\gamma}(y_c)-ic_i}\bigg],\\
\mathfrak{W}^4&:=\partial_m \bigg(\int_{-1}^1 \dfrac{1}{(U_{m,\gamma}(y)-U_{m,\gamma}(y_c)-ic_i)^2}\bigg(\dfrac{1}{(\phi_1^E)^2(y,m,\gamma,c_r+ic_i,-1)}-1\bigg)\;dy\bigg).
\end{aligned}
\]

It is important to note that in connection with the splitting of $W$ in \eqref{modified wronskian at cr 0}, we have
\begin{equation}\label{partial m I}
\mathfrak{W}^1+\mathfrak{W}^2+\mathfrak{W}^3=\partial_m \textup{I},
\end{equation}
and 
\begin{equation}
\mathfrak{W}^4 =\partial_m \textup{II}.
\end{equation}
With this in mind, we derive estimates for $\partial_m\textup{I}$ and $\partial_m\textup{II}$ which are recorded in the next two lemmas.

\begin{lemma}[Estimate on $\partial_{m} \textup{I}$]\label{estimate partial m I}
    The partial derivative of $\textup{I}$ with respect to $m$ obeys the following estimate
    \[
    C^{-1}\leq \partial_m \textup{I}_r \leq C, \quad  |\partial_{m} \textup{I}_i|\leq C \epsilon_0, \text{ for all } (c_r,c_i) \in \Sigma.
    \]
\end{lemma}
\begin{proof}
We start by analyzing the boundary terms $\mathfrak{W}^2$ and $\mathfrak{W}^3$. In order to get an estimate for $\mathfrak{W}^2$, we first rewrite it as 
\[
\mathfrak{W}^2=\dfrac{-\gamma\Gamma(y_c/\gamma)}{(U'_{m,\gamma}(y_c))^2}\bigg[\dfrac{U_{m,\gamma}(y)-U_{m,\gamma}(y_c)+ic_i}{(U_{m,\gamma}(y)-U_{m,\gamma}(y_c))^2+(c_i)^2}\bigg]\bigg|_{-1}^{1}.
\]
It is not hard to see that we obtain the following bound
\[
|\mathfrak{W}^2|\leq C\gamma,
\]
where we have used the differentiability of $U_{m,\gamma}$ and the fact that $y_c$ is away from $1$ and $-1$.
The next term we shall discuss is $\mathfrak{W}^3$. Notice that
\[
\begin{aligned}
\mathfrak{W}^3&=\dfrac{\gamma^2}{U'_{m,\gamma}(y_c)}\bigg[\dfrac{\Big(\widetilde{\Gamma}(y/\gamma)-\widetilde{\Gamma}(y_c/\gamma)\Big)}{\Big(U_{m,\gamma}(y)-U_{m,\gamma}(y_c)\Big)^2+(c_i)^2}\\&
\qquad \qquad \quad-\dfrac{2\Big(\widetilde{\Gamma}(y/\gamma)-\widetilde{\Gamma}(y_c/\gamma)\Big)\Big(U_{m,\gamma}(y)-U_{m,\gamma}(y_c)\Big) \Big(U_{m,\gamma}(y)-U_{m,\gamma}(y_c)+ic_i\Big)}{\Big((U_{m,\gamma}(y)-U_{m,\gamma}(y_c))^2+(c_i)^2\Big)^2}\bigg]\bigg|_{-1}^{1}.
\end{aligned}
\]
Again, employing the same argument as in estimating $\mathfrak{W}^2$, we can infer that 
\[
|\mathfrak{W}^3|\leq C\gamma^2.
\]

Now, we direct our attention to $\mathfrak{W}^1$. Via integration by parts, $\mathfrak{W}^1$ can be written in the following way
\[
\begin{aligned}
    \mathfrak{W}^1&=\partial_m\Bigg(\dfrac{-1}{U'_{m,\gamma}(y_c)} \int_{-1}^{1} \dfrac{U'_{m,\gamma}(y)-U_{m,\gamma}'(y_c)}{(U_{m,\gamma}(y)-U_{m,\gamma}(y_c)-ic_i)^2}\;dy\Bigg)\\&=\partial_m\Bigg(\dfrac{1}{U'_{m,\gamma}(y_c)}\bigg[\dfrac{(U_{m,\gamma}(y)-U_{m,\gamma}(y_c)+ic_i)}{(U_{m,\gamma}(y)-U_{m,\gamma}(y_c))^2+(c_i)^2}\bigg(\dfrac{U'_{m,\gamma}(y)-U'_{m,\gamma}(y_c)}{U'_{m,\gamma}(y)}\bigg)\bigg|_{-1}^{1}\bigg]\\&
    \qquad-\dfrac{1}{U'_{m,\gamma}(y_c)}\int_{-1}^{1} \dfrac{U'_{m,\gamma}(y_c)U''_{m,\gamma}(y)(U_{m,\gamma}(y)-U_{m,\gamma}(y_c)+ic_i)}{(U'_{m,\gamma}(y))^2((U_{m,\gamma}(y)-U_{m,\gamma}(y_c))^2+(c_i)^2)}\;dy\Bigg).
\end{aligned}
\]
As a consequence, we can subdivide $\mathfrak{W}^1$ into three different parts, namely 
\[
\begin{aligned}
\mathfrak{W}^1=\mathfrak{W}^1_1+\mathfrak{W}^1_2+\mathfrak{W}^1_3,
\end{aligned}
\]
where
\begin{align*}
&\mathfrak{W}^1_1=\dfrac{\gamma \Gamma(y_c)}{(U'_{m,\gamma}(y_c))^2}\bigg[\dfrac{(U_{m,\gamma}(y)-U_{m,\gamma}(y_c)+ic_i)}{(U_{m,\gamma}(y)-U_{m,\gamma}(y_c))^2+(c_i)^2}\bigg(\dfrac{U'_{m,\gamma}(y)-U'_{m,\gamma}(y_c)}{U'_{m,\gamma}(y)}\bigg)\bigg|_{-1}^{1}\bigg]\\&
+\dfrac{\gamma}{U'_{m,\gamma}(y_c)}\Bigg[\dfrac{(\gamma\widetilde{\Gamma}(y/\gamma)-\gamma\widetilde{\Gamma}(y_c/\gamma))(U'_{m,\gamma}(y)-U'_{m,\gamma}(y_c))}{((U_{m,\gamma}(y)-U_{m,\gamma}(y_c))^2+(c_i)^2)(U'_{m,\gamma}(y))}\\&+\dfrac{(U_{m,\gamma}(y)-U_{m,\gamma}(y_c)+ic_i)(\Gamma(y/\gamma)-\Gamma(y_c/\gamma))}{((U_{m,\gamma}(y)-U_{m,\gamma}(y_c))^2+(c_i)^2)(U'_{m,\gamma}(y))}\\&
 +\dfrac{\Big((2\gamma \widetilde{\Gamma}(y/\gamma)-2\gamma \widetilde{\Gamma}(y_c/\gamma))(U_{m,\gamma}(y)-U_{m,\gamma}(y_c))\Big(U_{m,\gamma}(y)-U_{m,\gamma}(y_c)+ic_i\Big)\Big(U'_{m,\gamma}(y)-U'_{m,\gamma}(y_c)\Big)}{((U_{m,\gamma}(y)-U_{m,\gamma}(y_c))^2+(c_i)^2)^2(U'_{m,\gamma}(y))}\\&
 +\dfrac{\Gamma(y/\gamma)((U_{m,\gamma}(y)-U_{m,\gamma}(y_c))^2+(c_i)^2)\Big(U_{m,\gamma}(y)-U_{m,\gamma}(y_c)+ic_i\Big)\Big(U'_{m,\gamma}(y)-U'_{m,\gamma}(y_c)\Big)}{((U_{m,\gamma}(y)-U_{m,\gamma}(y_c))^2+(c_i)^2)^2(U'_{m,\gamma}(y))^2}\Bigg|_{-1}^{1}\Bigg],\\&
\mathfrak{W}^1_2= -\dfrac{\gamma}{U'_{m,\gamma}(y_c)}\int_{-1}^{1}\bigg[\dfrac{\Gamma(y_c/\gamma)U''_{m,\gamma}(y)(U_{m,\gamma}(y)-U_{m,\gamma}(y_c)+ic_i)}{(U'_{m,\gamma}(y))^2((U_{m,\gamma}(y)-U_{m,\gamma}(y_c))^2+(c_i)^2)}\\&\qquad\qquad \qquad \qquad \qquad \qquad+\dfrac{U'_{m,\gamma}(y_c)U''_{m,\gamma}(y)(\gamma\widetilde{\Gamma}(y/\gamma)-\gamma\widetilde{\Gamma}(y_c/\gamma))}{(U'_{m,\gamma}(y))^2((U_{m,\gamma}(y)-U_{m,\gamma}(y_c))^2+(c_i)^2)}\\&
\qquad -\dfrac{2\gamma (\widetilde{\Gamma}(y/\gamma)-\widetilde{\Gamma}(y_c/\gamma))(U_{m,\gamma}(y)-U_{m,\gamma}(y_c))U'_{m,\gamma}(y_c)U''_{m,\gamma}(y)(U_{m,\gamma}(y)-U_{m,\gamma}(y_c)+ic_i)}{((U_{m,\gamma}(y)-U_{m,\gamma}(y_c))^2+(c_i)^2)^4(U'_{m,\gamma}(y))^2}\\&
\qquad-\dfrac{2\gamma \Gamma(y/\gamma)((U_{m,\gamma}(y)-U_{m,\gamma}(y_c))^2+(c_i)^2)U'_{m,\gamma}(y_c)U''_{m,\gamma}(y)(U_{m,\gamma}(y)-U_{m,\gamma}(y_c)+ic_i)}{((U_{m,\gamma}(y)-U_{m,\gamma}(y_c))^2+(c_i)^2)^4(U'_{m,\gamma}(y))^3} \bigg]\;dy,\\&
\mathfrak{W}^1_3=-\dfrac{1}{U'_{m,\gamma}(y_c)}\int_{-1}^{1}\dfrac{U'_{m,\gamma}(y_c)\Gamma'(y/\gamma)(U_{m,\gamma}(y)-U_{m,\gamma}(y_c)+ic_i)}{(U'_{m,\gamma}(y))^2((U_{m,\gamma}(y)-U_{m,\gamma}(y_c))^2+(c_i)^2)}\;dy.
\end{align*}

It is not hard to see that the boundary terms in the expression of $\mathfrak{W}^1_1$ yields the estimate
\[
|\mathfrak{W}^1_1|\leq C \gamma.
\]
 In addition to that, via Lemma~\ref{Supremum estimate} and Remark~\ref{remark A7}, the term $\mathfrak{W}^1_2$ obeys the estimate
\[
|\mathfrak{W}^1_2|\leq C \gamma.
\]

It remains to estimate $\mathfrak{W}^1_3$. Notice that $\mathfrak{W}^1_3$ can be decomposed into its real and imaginary parts. which reads
\[
\begin{aligned}
\mathfrak{W}^1_3&=-\int_{-1}^{1}\dfrac{\Gamma'(y/\gamma)(U_{m,\gamma}(y)-U_{m,\gamma}(y_c))}{(U'_{m,\gamma}(y))^2((U_{m,\gamma}(y)-U_{m,\gamma}(y_c))^2+(c_i)^2)}\;dy\\&
\quad-ic_i\int_{-1}^{1}\dfrac{\Gamma'(y/\gamma)}{(U'_{m,\gamma}(y))^2((U_{m,\gamma}(y)-U_{m,\gamma}(y_c))^2+(c_i)^2)}\;dy\\&=:\mathfrak{Re}(\mathfrak{W}^1_3)(y_c)+i\mathfrak{Im}(\mathfrak{W}^1_3)(y_c).
\end{aligned}
\]
Observe that when $y_c=0$, the real part of $\mathfrak{W}^1_3$ reads
\begin{equation}\label{real part of pi 1,3}
\begin{aligned}
\mathfrak{Re}(\mathfrak{W}^1_3)(0)&=-\int_{0\leq|y|<\gamma}\dfrac{\Gamma'(y/\gamma)U_{m,\gamma}(y)}{(U'_{m,\gamma}(y))^2((U_{m,\gamma}(y))^2+(c_i)^2)}\;dy\\& \quad-\int_{1\geq|y|\geq \gamma}\dfrac{\Gamma'(y/\gamma) U_{m,\gamma}(y)}{(U'_{m,\gamma}(y))^2((U_{m,\gamma}(y))^2+(c_i)^2)}\;dy.
\end{aligned}
\end{equation}
The integral first integral of $\mathfrak{Re}(\mathfrak{W}^1_3)(0)$ with interval of integration $0\leq|y|<\gamma$ can be estimated as follows
\[
\begin{aligned}
   C^{-1}&\leq-\frac{1}{\gamma}\int_{\substack{c_i\leq|y|\leq\gamma}}\dfrac{\Gamma'(y/\gamma)}{(y/\gamma)}\;dy\lesssim-\int_{0\leq|y|<\gamma}\dfrac{\Gamma'(y/\gamma)(U_{m,\gamma}(y))}{(U'_{m,\gamma}(y))^2((U_{m,\gamma}(y))^2+(c_i)^2)}\;dy
   \\&\lesssim -\int_{\substack{0\leq|y|<c_i}}\dfrac{\dfrac{\Gamma'(y/\gamma)}{(y/\gamma)}(y^2/\gamma)}{(y)^2+(c_i)^2}\;dy-\int_{\substack{c_i\leq |y|<\gamma}}\dfrac{\dfrac{\Gamma'(y/\gamma)}{(y/\gamma)}(y^2/\gamma)}{(y)^2+(c_i)^2}\;dy\\&
   \lesssim \int_{\substack{0\leq |y|< c_i}} \dfrac{y^2}{(c_i)^2}\dfrac{1}{\gamma}\;dy-\frac{1}{\gamma}\int_{\substack{c_i\leq|y|\leq\gamma}}\dfrac{\Gamma'(y/\gamma)}{(y/\gamma)}\;dy\\
   &\lesssim \frac{\gamma-|c_i|}{\gamma}\leq C \quad \text{for}\quad |c_i|\leq \epsilon_0\gamma.
\end{aligned}
\]
For the second integral of $\mathfrak{Re}(\mathfrak{W}^1_3)(0)$ with interval of integration $\gamma\leq|y|<1$, one can recast it in the following way
\begin{align}
    0&< -\int_{1\geq|y|\geq \gamma}\dfrac{\Gamma'(y/\gamma)(U_{m,\gamma}(y))}{(U'_{m,\gamma}(y))^2((U_{m,\gamma}(y))^2+(c_i)^2)}\;dy\\&
    \lesssim -\int_{1\geq|y|\geq \gamma}\dfrac{\Gamma'(y/\gamma)}{(y/\gamma)}\dfrac{y^2}{((y)^2+c_i^2)}\frac{1}{\gamma}dy\leq C.
\end{align}
As a result, we obtain 
\begin{equation}\label{estimate real W13 at 0}
    C^{-1}<\mathfrak{Re}(\mathfrak{W}^1_3)(0)<C.
\end{equation}

Furthermore, one can check that

\begin{align*}
&\partial_{y_c}(\mathfrak{Re}(\mathfrak{W}^1_3)(y_c))=U'_{m,\gamma}(y_c)\partial_{c_r}\mathfrak{Re}(\mathfrak{W}^1_3)(c_r)\\&=U'_{m,\gamma}(y_c) \partial_{c_r} \bigg( -\int_{-1}^{1}\dfrac{\Gamma'(y/\gamma)(U_{m,\gamma}(y)-c_r)}{(U'_{m,\gamma}(y))^2((U_{m,\gamma}(y)-c_r)^2+(c_i)^2)}\;dy \bigg) \\&=U'_{m,\gamma}(y_c)\int_{U_{m,\gamma}(-1)}^{U_{m,\gamma}(1)} (\partial_{c_r}+\partial_{v}) \bigg( \dfrac{(v-c_r)}{\Big((v-c_r)^2+c_i^2\Big)} \Gamma'(U^{-1}_{m,\gamma}(v)/\gamma)(\partial_v(U_{m,\gamma}^{-1}(v)))^3 \bigg)\;dv\\&
\qquad -\dfrac{(v-c_r)}{\Big((v-c_r)^2+c_i^2\Big)} \Gamma'(U^{-1}_{m,\gamma}(v)/\gamma)(\partial_v(U_{m,\gamma}^{-1}(v)))^3\bigg|_{U_{m,\gamma}(-1)}^{U_{m,\gamma}(1)} U'_{m,\gamma}(y_c) \\&
=U'_{m,\gamma}(y_c)\int_{U_{m,\gamma}(-1)}^{U_{m,\gamma}(1)} \dfrac{(v-c_r)}{\Big((v-c_r)^2+c_i^2\Big)} 
 \partial_v \bigg(\Gamma'(U^{-1}_{m,\gamma}(v)/\gamma)(\partial_v(U_{m,\gamma}^{-1}(v)))^3\bigg)\;dv\\&
 \qquad-\dfrac{(v-c_r)}{\Big((v-c_r)^2+c_i^2\Big)} \Gamma'(U^{-1}_{m,\gamma}(v)/\gamma)(\partial_v(U_{m,\gamma}^{-1}(v)))^3\bigg|_{U_{m,\gamma}(-1)}^{U_{m,\gamma}(1)}U'_{m,\gamma}(y_c)\\&
 \leq \int_{-1}^{1} \dfrac{(U_{m,\gamma}(y)-c_r)}{\Big((U_{m,\gamma}(y)-c_r)^2+c_i^2\Big)} 
  \bigg(\dfrac{1}{\gamma}\dfrac{\Gamma''(y/\gamma)}{(U'_{m,\gamma}(y))^3}-3\dfrac{\Gamma'(y/\gamma)U''_{m,\gamma}(y)}{(U'_{m,\gamma}(y))^4}\bigg)U'_{m,\gamma}(y_c)\;dy+C\\&
 \lesssim \norm{\dfrac{1}{\gamma}\dfrac{\Gamma''(y/\gamma)}{(U'_{m,\gamma}(y))^3}-3\dfrac{\Gamma'(y/\gamma)U''_{m,\gamma}(y)}{(U'_{m,\gamma}(y))^4}}_{L^2}^{1/2}\\&
 \quad \times \bigg(\norm{\partial_y\Big(\dfrac{1}{\gamma}\dfrac{\Gamma''(y/\gamma)}{(U'_{m,\gamma}(y))^3}-3\dfrac{\Gamma'(y/\gamma)U''_{m,\gamma}(y)}{(U'_{m,\gamma}(y))^4}\Big)}_{L^2}^{1/2}+\norm{ \Big(\dfrac{1}{\gamma}\dfrac{\Gamma''(y/\gamma)}{(U'_{m,\gamma}(y))^3}-3\dfrac{\Gamma'(y/\gamma)U''_{m,\gamma}(y)}{(U'_{m,\gamma}(y))^4}\Big)}_{L^{\infty}}^{1/2}\bigg)+C\\&
 \lesssim \dfrac{1}{\gamma}. 
\end{align*}

By the Mean Value theorem, it holds that
\[
|\mathfrak{Re}(\mathfrak{W}^1_3)(y_c)-\mathfrak{Re}(\mathfrak{W}^1_3)(0)|\leq |y_c| \norm{\partial_{y_c}\mathfrak{Re}(\mathfrak{W}^1_3)}_{L^{\infty}}\lesssim \dfrac{|y_c|}{\gamma}\lesssim \epsilon_0.
\]
Therefore, using the previous inequality with $\epsilon_0 \ll 0$ and together with the estimate derived in \eqref{estimate real W13 at 0}, we can infer that there exists $C>0$ such that
\[
C^{-1}\leq \mathfrak{Re}(\mathfrak{W}^1_3)(y_c)\leq C
\] which is equivalent to saying  $C^{-1}\leq \partial_m \text{I}_r \leq C$.

Next, we go through a similar process and derive an estimate for the imaginary part of $\mathfrak{W}^1_3$. The argument for this patterns the one presented previously when estimating the real part of $\mathfrak{W}^1_3$. Direct computation tells us that

\begin{align*}
|\mathfrak{Im}(\mathfrak{W}^1_3)(0)|&=\Big|c_i\int_{-1}^{1}\dfrac{\Gamma'(y/\gamma)}{(U'_{m,\gamma}(y))^2((U_{m,\gamma}(y))^2+(c_i)^2)}\;dy\Big|\\&\leq \int_{-1}^{1}\dfrac{|c_i||\frac{\Gamma'(y/\gamma)}{(y)/\gamma}||y/\gamma|}{(U'_{m,\gamma}(y))^2((U_{m,\gamma}(y))^2+(c_i)^2)}\;dy\\&\lesssim \dfrac{c_i}{\gamma}
\int_{U_{m,\gamma}(-1)}^{U_{m,\gamma}(1)}\dfrac{U^{-1}_{m,\gamma}(v)}{((v)^2+c_i^2)} (\partial_v(U^{-1}_{m,\gamma}(v)))^3\;dv\\&
\lesssim \dfrac{c_i}{\gamma}\bigg(\lVert \partial_v(U^{-1}_{m,\gamma}(v))^3 \rVert^{1/2}_{L^2}\bigg(\lVert \partial_v(\partial_v(U^{-1}_{m,\gamma}(v))^3) \rVert^{1/2}_{L^2}+\lVert \partial_v(U^{-1}_{m,\gamma}(v))^3 \rVert^{1/2}_{L^\infty}\bigg)\bigg)\lesssim \dfrac{c_i}{\gamma} \lesssim \epsilon_0.
\end{align*}

Additionally, we obtain the following estimate for $\partial_{y_c}(\mathfrak{Im}(\mathfrak{W}^1_3)(y_c))$,

\begin{align*}
&\partial_{y_c}(\mathfrak{Im}(\mathfrak{W}^1_3)(y_c))=U'_{m,\gamma}(y_c)\partial_{c_r}\mathfrak{Im}(\mathfrak{W}^1_3)\\&=U'_{m,\gamma}(y_c) \partial_{c_r} \bigg(\int_{-1}^{1}\dfrac{\Gamma'(y/\gamma)c_i}{(U'_{m,\gamma}(y))^2((U_{m,\gamma}(y)-c_r)^2+(c_i)^2)}\;dy \bigg) \\&=U'_{m,\gamma}(y_c)\int_{U_{m,\gamma}(-1)}^{U_{m,\gamma}(1)} (\partial_{c_r}+\partial_{v}) \bigg( \dfrac{c_i}{\Big((v-c_r)^2+c_i^2\Big)} \Gamma'(U^{-1}_{m,\gamma}(v)/\gamma)(\partial_v(U_{m,\gamma}^{-1}(v)))^3 \bigg)\;dv\\&
\qquad -\dfrac{(c_i)}{\Big((v-c_r)^2+c_i^2\Big)} \Gamma'(U^{-1}_{m,\gamma}(v)/\gamma)(\partial_v(U_{m,\gamma}^{-1}(v)))^3\bigg|_{U_{m,\gamma}(-1)}^{U_{m,\gamma}(1)} U'_{m,\gamma}(y_c) \\&
=U'_{m,\gamma}(y_c)\int_{U_{m,\gamma}(-1)}^{U_{m,\gamma}(1)} \dfrac{c_i}{\Big((v-c_r)^2+c_i^2\Big)} 
 \partial_v \bigg(\Gamma'(U^{-1}_{m,\gamma}(v)/\gamma)(\partial_v(U_{m,\gamma}^{-1}(v)))^3\bigg)\;dv\\&
 \qquad-\dfrac{c_i}{\Big((v-c_r)^2+c_i^2\Big)} \Gamma'(U^{-1}_{m,\gamma}(v)/\gamma)(\partial_v(U_{m,\gamma}^{-1}(v)))^3\bigg|_{U_{m,\gamma}(-1)}^{U_{m,\gamma}(1)}U'_{m,\gamma}(y_c)\\&
 \leq \int_{-1}^{1} \dfrac{c_i}{\Big((U_{m,\gamma}(y)-c_r)^2+c_i^2\Big)} 
  \bigg(\dfrac{1}{\gamma}\dfrac{\Gamma''(y/\gamma)}{(U'_{m,\gamma}(y))^3}-3\dfrac{\Gamma'(y/\gamma)U''_{m,\gamma}(y)}{(U'_{m,\gamma}(y))^4}\bigg)U'_{m,\gamma}(y_c)\;dy+C\\&
 \lesssim \norm{\dfrac{1}{\gamma}\dfrac{\Gamma''(y/\gamma)}{(U'_{m,\gamma}(y))^3}-3\dfrac{\Gamma'(y/\gamma)U''_{m,\gamma}(y)}{(U'_{m,\gamma}(y))^4}}_{L^2}^{1/2}\\&
 \quad \times \bigg(\norm{\partial_y\Big(\dfrac{1}{\gamma}\dfrac{\Gamma''(y/\gamma)}{(U'_{m,\gamma}(y))^3}-3\dfrac{\Gamma'(y/\gamma)U''_{m,\gamma}(y)}{(U'_{m,\gamma}(y))^4}\Big)}_{L^2}^{1/2}+\norm{ \Big(\dfrac{1}{\gamma}\dfrac{\Gamma''(y/\gamma)}{(U'_{m,\gamma}(y))^3}-3\dfrac{\Gamma'(y/\gamma)U''_{m,\gamma}(y)}{(U'_{m,\gamma}(y))^4}\Big)}_{L^{\infty}}^{1/2}\bigg)+C\\&
 \lesssim \dfrac{1}{\gamma}. 
\end{align*}

Similar as before, via the Mean Value theorem we can infer that 
\begin{equation}\label{estimate Imaginary part W13}
|\mathfrak{Im}(\mathfrak{W}^1_3)(y_c)|\lesssim \epsilon_0.
\end{equation}
Having derived both estimates for the real and imaginary parts of $\mathfrak{W}^1_3$, we therefore conclude that
\[
C^{-1}<|\mathfrak{W}^1_3|<C.
\]
Recalling \eqref{partial m I} and choosing $\gamma$ to be sufficiently small yield
\[
C^{-1}\leq|\partial_m \textup{I}|\leq C. \qedhere
\]
\end{proof}

Last but not least, let us estimate the term $\mathfrak{W}^4$. 

\begin{lemma}[Estimate on $\partial_{m} \textup{II}$]\label{estimate partial m II}
We have the following estimate
\[
|\partial_{m}\textup{II}| \leq C \gamma.
\]
\end{lemma}
\begin{proof}
First of all, notice that $\partial_m\textup{II}=\mathfrak{W}^4$. We write $\mathfrak{W}^4$ explicitly as follows
\[
\begin{aligned}
\mathfrak{W}^4&=\int_{-1}^1\dfrac{-2\gamma^2\Big(\widetilde{\Gamma}(y/\gamma)-\widetilde{\Gamma}(y_c/\gamma)\Big)}{\Big(U_{m,\gamma}(y)-U_{m,\gamma}(y_c)-ic_i\Big)^3}\bigg(\dfrac{1}{(\phi_1^E)^2(y,m,\gamma,c_r+ic_i,-1)}-1\bigg)\\&
\qquad +\dfrac{1}{\Big(U_{m,\gamma}(y)-U_{m,\gamma}(y_c)-ic_i\Big)^2}\partial_m\bigg(\dfrac{1}{(\phi_1^E)^2(y,m,\gamma,c_r+ic_i,-1)}-1\bigg)\;dy\\&=\mathfrak{W}^4_1+\mathfrak{W}^4_2.
\end{aligned}
\]

The term $\mathfrak{W}^4_1$ obeys the following inequality
\[
\begin{aligned}
|\mathfrak{W}^4_1|&\leq \int_{-1}^1\dfrac{\gamma|y-y_c|}{((U_{m,\gamma}(y)-c_r)^2+(c_i)^2)^{3/2}}\lVert \Gamma(y/\gamma)\rVert_{L^{\infty}}\bigg| \dfrac{1}{(\phi_1^E)^2(y,m,\gamma,c_r+ic_i,-1)}-1\bigg|\;dy\\&
\leq C \gamma \int_{-1}^1 \dfrac{\min\{c_i,(y-y_c)^2\}}{(y-y_c)^2+(c_i)^2}\;dy\leq C \gamma,
\end{aligned}
\]
where we have used the bounds in Lemma~\ref{phi1 phi2} for $\phi_1$ and $\phi_2$.

Next, in estimating the other portion of $\mathfrak{W}^4$, namely $\mathfrak{W}^4_2$, we use the estimate found in Lemma~\ref{partial m phi1}, more precisely 
\[
|\mathfrak{W}^4_2|\lesssim \int_{-1}^1 \dfrac{\gamma (y-y_c)^2}{(y-y_c)^2+(c_i)^2}\;dy\lesssim \gamma \Big(\tan^{-1}(\frac{c_i}{1-y_c})-\tan^{-1}(\frac{c_i}{-1-y_c})\Big)\lesssim \gamma^2,
\]
where we have used the Taylor expansion of $\tan^{-1}$ and the fact that $c_i\approx \gamma$. Gathering both estimates for $\mathfrak{W}^4_1$ and $\mathfrak{W}^4_2$ and choosing the constant $C$ sufficiently large, we conclude that $|\mathfrak{W}^4|\leq C \gamma$. In light of the expression \text{II} in \eqref{modified wronskian at cr 0} and the expression of $\mathfrak{W}^4$ defined earlier, the above estimate is equivalent to 
\begin{equation}\label{partial m II}
|\partial_m \text{II}|\leq C \gamma.\qedhere
\end{equation}
\end{proof}

\begin{remark}\label{remark partial m W}
    As a consequence of Lemma~\ref{estimate partial m I} and Lemma~\ref{estimate partial m II}, we have 
    \[
    C^{-1}\leq \partial_{m}W_r\leq C, \qquad  |\partial_{m}W_i|\leq C \epsilon_0.
    \]
\end{remark}

\textbf{Step 3:} Finally, we would like gather and unify of all the estimates we have obtained to determine the sign and bounds of $\partial_m c_i$ and $|\partial_m c_r|$. First of all recall that the Wronskian $\mathcal{W}=\phi(-1)\phi(1) W$  is complex analytic. Using the notations in Lemma~\ref{phi1 phi2} and Lemma~\ref{estimate on the good derivative of phi1E}, we can write $\phi(-1)\phi(1)=(U_{m,\gamma}(-1)-c)\phi_1^E(-1)(U_{m,\gamma}(1)-c)\phi_1^E(1)$. Since the reciprocal of $(U_{m,\gamma}(-1)-c)(U_{m,\gamma}(1)-c)$ is in itself analytic, we define another modified Wronskian where we mod out the terms $(U_{m,\gamma}(-1)-c)(U_{m,\gamma}(1)-c)$. Abusing notation, let us reuse $\mathcal{W}$ to represent such modified Wronskian, simply put
\[
\mathcal{W}=\phi_1^E(-1)\phi_1^E(1)W=:gW.
\]
Since $g$ is complex, we shall use $g_r$ and  $g_i$ to denote the real and imaginary parts of $g$ respectively. As a result, $\mathcal{W}$ can be decomposed into its real and imaginary parts as follows
\[
\begin{aligned}
    &\mathcal{W}_r=g_r W_r-g_i W_i,\\&
    \mathcal{W}_i=g_r W_i+g_i W_r.
\end{aligned}
\]
Applying $\partial_{c_i}$ and $\partial_{m}$ to the above equations yields the following
\[
\begin{aligned}
&\partial_{c_i}\mathcal{W}_r=(\partial_{c_i}g_r)W_r+g_r \partial_{c_i}W_r-(\partial_{c_i}g_i)W_i-g_i\partial_{c_i}W_i,\\&
\partial_{c_i}\mathcal{W}_i=(\partial_{c_i}g_r)W_i+g_r \partial_{c_i}W_i+(\partial_{c_i}g_i)W_r+g_i\partial_{c_i}W_r,\\&
\partial_{m}\mathcal{W}_r=(\partial_{m}g_r)W_r+g_r \partial_{m}W_r-(\partial_{m}g_i)W_i-g_i\partial_{m}W_i,\\&
\partial_{m}\mathcal{W}_i=(\partial_{m}g_r)W_i+g_r \partial_{m}W_i+(\partial_{m}g_i)W_r+g_i\partial_{m}W_r.
\end{aligned}
\]
Taking advantage of estimates of $\phi_2$ in \eqref{phi1 phi2}, we obtain
\[
|\mathfrak{Re}(\phi_2 (-1)\phi_2(1))-1|\lesssim c_i.
\]
Choosing $c_i$ to be sufficiently small permits us to infer that 
\[
C^{-1}\leq g_r=\phi_1(-1)\phi_1(1)+\mathfrak{Re}(\phi_2 (-1)\phi_2(1))\leq C.
\] Additionally, 
\[
|g_i|\lesssim c_i.
\]
These observations combined with our estimates in Lemma~\ref{estimate of W}, \ref{estimate partial ci Wi}, \ref{estimate partial ci Wr}, \ref{phi1 phi2}, \ref{partial m phi1} and \ref{partial c phi1} allow us to infer
\begin{equation}\label{various estimates for derivative of W}
\begin{aligned}
\dfrac{-1}{C\gamma}\leq &\partial_{c_i} \mathcal{W}_r \leq \dfrac{-C}{\gamma},\quad
|\partial_{c_i} \mathcal{W}_i| \leq \dfrac{C}{\gamma},\quad
 C^{-1} \leq \partial_{m} \mathcal{W}_r \leq C,\quad
|\partial_{m} \mathcal{W}_i| \leq C \epsilon_0.
\end{aligned}
\end{equation}
Consequently, using the above inequalities and applying them to the system in \eqref{ODE} yield the desired estimates
\begin{equation}\label{ineq partial m ci}
C^{-1} \gamma \leq -F \leq C \gamma, \qquad |G| \leq C \gamma.
\end{equation}

Additionally, one can check that the determinant of the Jacobian matrix
\[
J:=\begin{pmatrix}
\partial_m \mathcal{W}_r & \partial_{c_i} \mathcal{W}_r \\
\partial_m \mathcal{W}_i  & \partial_{c_i} \mathcal{W}_r 
\end{pmatrix},
\]
at $(m_*,\gamma_*,
0,-1)$ is given by 
\[
|J(m_*,\gamma_*,
0,-1)|=(\partial_{c_i}  \mathcal{W}_r)^2(m_*,\gamma_*,
0,-1) +(\partial_{c_i}  \mathcal{W}_i)^2(m_*,\gamma_*,
0,-1)\gtrsim \frac{1}{\gamma^2}.
\]
The Theorem~\ref{linear}\ref{Instability} therefore follows directly from an application of IFT as described earlier. Meanwhile, estimate \eqref{estimate for lambda r} is simply a straight consequence of inequalities in \eqref{ineq partial m ci}.\qedhere
\end{proof}

Having completed the proof for Theorem~\ref{linear}. We are now ready to present our argument for Theorem~\ref{nonlinear}. This is exactly the content of what is coming next.

\subsection{Nontrivial steady solution}\label{Cat's eyes}
In this subsection, we show an existence of a flow that solves \eqref{original euler} with non-shearing structure. As mentioned at the outset, this portion of our work is fully inspired by the seminal work of \cite{LinZeng2011}. In the aforementioned paper, both authors, Lin and Zeng, showed an existence of a steady nontrivial (non-sheared) flow near Couette. In contrast, our result applies to a general class of monotonic background shear flows (not only Couette). To that end, we begin with the following lemma.
\begin{lemma}\label{Bifurcation}
Let $\mathcal{U}$ and $G$ be as stated in Lemma~\ref{regularity lemma}. If the operator
\begin{equation}\label{Operator L}
\mathcal{X}:=-\dfrac{d^2}{dy^2}+\dfrac{\mathcal{U}''}{\mathcal{U}},\; H^{2}(-1,1) \to L^{2}(-1,1),
\end{equation} with Dirichlet boundary conditions at $y\in \{\pm1\}$ has a negative eigenvalue $-k_0^2$, then there exists  $\epsilon_0>0$ such that for any $0 < \epsilon<\epsilon_0$, there exist a nontrivial (non-sheared) steady solution $(u_\epsilon(x,y),v_\epsilon(x,y))$ with period $T_\epsilon$ such that 
\begin{equation}\label{norm nonsheard and background shear}
 \norm{(u_{\epsilon}(x,y),v_{\epsilon}(x,y))-(\mathcal{U},0)}_{H^{2}(\mathbb{T}_{2\pi}\times[-1,1])}=\epsilon.
\end{equation} As $\epsilon \to 0$, the period $T_\epsilon \to \frac{2 \pi}{k_0}$. Moreover, 
\begin{equation}\label{higher exponent norm}
 \norm{(u_{\epsilon}(x,y),v_{\epsilon}(x,y))-(\mathcal{U},0)}_{H^{5/2-\tau+N}(\mathbb{T}_{2\pi}\times[-1,1])}\leq C(\gamma,N) \epsilon.
\end{equation}
Here the constant $C(\gamma,N)$ depends on $\gamma, N$. 
\end{lemma}
\begin{proof} 
We recall the definition of $\widetilde{\psi}_0$ in \eqref{definition of tilde psi}. 
For arbitrary stream function, we introduce the extension of the map $G\in C^1(\min{\widetilde{\psi}_0},\max{\widetilde{\psi}_0})$ via $\widetilde{G}\in C^1_0(\mathbb{R})$ to the entire real line where $\widetilde{G} \equiv G$ on $[\min{\widetilde{\psi}_0},\max{\widetilde{\psi}_0}]$. One can check that any stream function $\psi $ that solves
\begin{equation}\label{Poisson-type eq}
\Delta \psi= \widetilde{G}(\psi),
\end{equation}
is indeed also a solution of \eqref{Euler Equations}. In preparation of implementing a bifurcation framework, we introduce a new horizontal variable $\zeta:=kx$. The variable $k$ will play a role as the bifurcation parameter in our construction. In the new coordinate $(\zeta,y)$, the stream function $\psi$ becomes $\bar{\psi}$. Therefore, equation \eqref{Poisson-type eq} now reads 

\begin{equation}\label{Poisson-type eq in new variable}
k^2\dfrac{\partial^2}{\partial_\zeta^2}\bar{\psi}(\zeta,y) +\dfrac{\partial^2}{\partial_y^2}\bar{\psi}(\zeta,y) = \widetilde{G}(\bar{\psi}(\zeta,y)).
\end{equation} 

Our goal is to construct ``near-by solutions"  to the background flow. In light of that, we consider the perturbation of the background stream function 
\[
\psi_{\text{per}}(\zeta, y)=\bar{\psi}(\zeta,y)-\psi_0(y),
\] where the subscript ``per" refers to perturbation. 

We define the following two spaces
\[
P:=\{\psi_{\text{per}}(\zeta,y)\in H^3(\mathbb{T}_{2 \pi}\times [-1,-1]), \psi_{\text{per}}(\zeta,-1)=0=\psi_{\text{per}}(\zeta,1),2\pi-\text{periodic}\},
\]
and
\[
D:=\{\psi_{\text{per}}(\zeta,y)\in H^1(\mathbb{T}_{2 \pi}\times [-1,-1]), \psi_{\text{per}}(\zeta,-1)=0=\psi_{\text{per}}(\zeta,1),2\pi-\text{periodic}\}.
\]
Via \eqref{Poisson-type eq in new variable}, we define an operator
\[
\mathcal{F}(\psi_{\text{per}},k^2):P\times \mathbb{R}^+ \to D,
\]
by
\begin{equation}
\begin{aligned}
\mathcal{F}(\psi_{\text{per}},k^2):&=k^2\dfrac{\partial^2}{\partial_\zeta^2}\psi_{\text{per}}+\dfrac{\partial^2}{\partial_y^2}\psi_{\text{per}} -\Big(\widetilde{G}(\psi_{\text{per}}+\psi_0)-\widetilde{G}(\psi_0)\Big),
\end{aligned}
\end{equation}

Observe that for any $k\in \mathbb{R}$, $\mathcal{F}(0,k^2)=0$. We then linearize the operator $\mathcal{F}$ around $(\psi_{\text{per}},k^2)=(0,k_0^2)$, where $k_0$ is the wave number that corresponds to the negative eigenvalue of the Sturm--Liouville problem. The resulting linearized operator around $(0,k_0^2)$ now reads
\begin{equation}
\begin{aligned}
\mathfrak{L}:=\mathcal{F}_{\psi_{\text{per}}}(0,k_0^2)&=k_0^2 \dfrac{\partial^2}{\partial_\zeta^2}+\dfrac{\partial^2}{\partial_y^2}-\widetilde{G}'(\psi_0)\\
                    &=k_0^2 \dfrac{\partial^2}{\partial_\zeta^2}+\dfrac{\partial^2}{\partial_y^2}-\dfrac{\mathcal{U}''}{\mathcal{U}}.\end{aligned}
\end{equation}  Let $\phi_0$ be the associated positive eigenvector of the eigenvalue $-k_0^2$ of $\mathcal{X}$. One can check that the kernel of $\mathfrak{L}$ given by
\[
\textup{Ker}(\mathfrak{L})=\textup{span}(\phi_0(y) \cos{\zeta}),
\] is one dimensional. Further, since $\mathfrak{L}$ is self-adjoint, then $\phi_0(y) \cos{\zeta} \notin \text{Ran}(\mathfrak{L})$. One can also see that 
\begin{equation}\label{mixed derivative}
\lim_{k \to k_0}\dfrac{\partial_{\psi_{\text{per}}}\mathcal{F}(0,k^2)(\phi_0(y) \cos{\zeta})-\partial_{\psi_{\text{per}}}\mathcal{F}(0,k_0^2)(\phi_0(y) \cos{\zeta})}{k^2-k_0^2}=-\phi_0(y) \cos{\zeta} \notin \text{Ran}(\mathfrak{L}).
\end{equation}Formally speaking, the mixed-higher derivative is well-defined at $(0,k_0^2)$ in the direction of the kernel.

All of these facts together allow us to employ the Crandall--Rabinowitz local bifurcation theorem \cite{CrandallRabinowitz1971} from which we obtain a one-parameter local curve of solutions
\[
\mathcal{C}^\epsilon_{\loc}:=\{\psi_{\text{per}}(\epsilon),k^2(
\epsilon): \mathcal{F}(\psi_{\text{per}}(\epsilon),k^2(\epsilon))=0,\; 0\leq \epsilon<\epsilon_0\} \subset H^3 \times \mathbb{R},
\]
for some $\epsilon_0>0.$
 It is important to remark that the curve $\mathcal{C}^\epsilon_{\loc}$ contains the information of the perturbation. As $\epsilon \to 0$, then $(\psi_{\text{per}}(\epsilon),k^2(\epsilon))\to (0,k^2_0)$; this is the case when perturbation size shrinks to zero. Hence, to the leading order, the perturbed stream function takes the form
 \begin{equation}
 \bar{\psi}(\zeta,y)=\psi_0(y)+\epsilon \phi_0(y) \cos{\zeta}+o(\epsilon),
 \end{equation}
 where $\phi_0$ is the positive eigenfunction associated with eigenvalue $-k_0^2$ of the operator $\mathcal{X}$. This gives rise to the non-sheared solution and the norm in \eqref{norm nonsheard and background shear}.
 
 In order to obtain the estimate on the higher regularity \eqref{higher exponent norm}, we first recall from Lemma~\ref{regularity lemma} that the map $\widetilde{G}\in C^{N+1}$. Using the equation
 \[
 \Delta \psi_{\text{per}}(\zeta,y)=\widetilde{G}(\psi_{\text{per}}(\zeta,y))=\psi_{\text{per}}(\zeta,y)\int_{0}^1\widetilde{G}'(\psi_0+s(\psi_{\text{per}}(\zeta,y)))\;ds.
 \]
 and \eqref{norm nonsheard and background shear},
along with some elliptic bootstrap argument, we get the desired estimate \eqref{higher exponent norm}.
 \qedhere
 \end{proof}
 
 \begin{remark}
We would like to mention that here we do not require $\mathcal{F}$ to belong to $C^2$. Rather, we are content with $\mathcal{F}\in C^1$ as we only need the mix derivative \eqref{mixed derivative} to exist at the point $(0,k_0^2)$ which is one of the sufficient requirements to use the bifurcation theory of \cite{CrandallRabinowitz1971}. This level of regularity condition on $\mathcal{F}$ is indeed weaker than the one used in \cite{LinZeng2011}.
\end{remark}

The remaining portion of this section will be devoted to proving Theorem~\ref{nonlinear}. Most of the ingredients for the proof has been established in Lemma~\ref{Bifurcation}. Additionally, our shear flow $(\mathcal{U},0)$ now takes the form $(U_{m,\gamma},0))$, that is the perturbed shear flow of $(U,0)$. The local bifurcation curve will start at the perturbed background shear. Note that the assumptions on $\mathcal{U}$ are also satisfied by $U_{m,\gamma}$.  Furthermore, the spectral condition required in Lemma~\ref{Bifurcation} for $\mathcal{X}$ is also satisfied by the operator  $\mathcal{H}_{m,\gamma}$.

\begin{proof}[\textbf{Proof of Theorem~\ref{nonlinear}}]
  First of all, we know that from the proof 
 of Theorem~\ref{linear}, for fixed $\gamma_* \in (0,\gamma_0),$ there exists $m_*$, such that $\lambda_{m_*,\gamma_*}=-1$. Moreover, due to the monotonicity of $\lambda_{m,\gamma}$, we can choose $m$ on the interval $|m-m_*|\lesssim \gamma^N$ such that the corresponding operator $\mathcal{H}_{m,\gamma_*}$ has a minimal eigenvalue $\lambda_{m,\gamma_*}$ located close to $-1$ with the associated eigenfunction satisfying the Dirichlet conditions at $y=\pm 1$. 
 
 From Lemma~\ref{Bifurcation}, it is known that for any $m\in (m_*-\gamma^N,m_*+\gamma^N)$, we have a local curve of solutions of the 2-D Euler equations parameterized by $\epsilon$ with $0<\epsilon<\epsilon_0$ \eqref{Euler Equations}  bifurcating from the shear flow $(U_{m,\gamma}(y),0)$ where a nontrivial (non-sheared) steady flow belongs. 

 Via the monotonicity of eigenvalue $\lambda_{m,\gamma}$, there exist $m_1\in(m_*-\gamma^N,m_*)$,  $m_2\in(m_*,m_*+\gamma^N)$, such that 
\[
\lambda_{m_1,\gamma_*;\epsilon}<-1<\lambda_{m_2,\gamma_*;\epsilon}.
\]
Moreover, for any $m\in(m_1,m_2)$, we can show that when $0 \leq s< 5/2+N$,
\[
\norm{U_{m,\gamma}(y)-U(y)}_{H^s(-1,1)}\to 0 \text{ as } \gamma\to 0.
\]
First of all, it is not hard to see that 
\[
U_{m,\gamma}(y)-U_{m_*,\gamma}(y)=(m-m*)\gamma^2\widetilde\Gamma(y/\gamma).
\]
Hence, 
\[
\norm{U_{m,\gamma}(y)-U_{m_*,\gamma}(y)}_{H^s(-1,1)}\lesssim \norm{\gamma^{2+N}\widetilde\Gamma(y/\gamma)}_{H^s(\mathbb{R})}.
\]
But, 
\[
\norm{\gamma^{2+N}\widetilde\Gamma(y/\gamma)}_{L^2(\mathbb{R})} \lesssim \gamma^{3+N}, \text{ and } \norm{\gamma^{2+N}\widetilde\Gamma(y/\gamma)}_{\dot{H}^s(\mathbb{R})} \lesssim \gamma^{5/2-s+N}.
\]
Therefore, we conclude that 
\begin{equation}\label{distance between background shears}
    \norm{U_{m,\gamma}(y)-U_{m_*,\gamma}(y)}_{H^{(\frac{5}{2}-s+N)}(-1,1)}\to 0 \text{ as } \gamma\to 0,
\end{equation}
for $0 \leq s< 5/2+N$.

Observe that, since $\lambda_{m,\gamma;\epsilon}$ is continuous with respect to $m$, the above inequality tells us that there exists a pair $(m^{\circ},\epsilon^{\circ})\in (m_1,m_2)\times(0,\epsilon_0)$ such that following along the local curve bifurcating from $U_{m^{\circ},\gamma_*}$, we find a nontrivial steady solution at the parameter value $\epsilon=\epsilon^{\circ}$ with  $k=\pm 1$ and satisfies the norm in \eqref{norm nonsheard and background shear}.
By choosing even smaller $m\in (m_*-m^{\circ},m_*+m^{\circ})$ and using the estimate \eqref{higher exponent norm}, one obtains
\[\norm{(u_{\epsilon}(x,y),v_{\epsilon}(x,y))-(U_{m,\gamma}(y),0)}_{H^{5/2-\tau+N}(\mathbb{T}_{2\pi}\times[-1,1])}\leq \epsilon,
\]
for sufficiently small $\epsilon>0$. This is precisely the desired inequality in \eqref{norm of the distance background and nontrivial}. Hence, this completes the proof of Theorem~\ref{nonlinear}. \qedhere
\end{proof}
\appendix

\section{Regular solution to the homogeneous Rayleigh equations}\label{Related Results}
The following appendix contains key lemmas that are used in some parts of the proofs throughout the paper. The first lemma below provides us with a representation of solution of the Rayleigh equation \eqref{Rayleigh Equation} along with some estimates of functions involved in he representation. Since this lemma has appeared in previous works, we have decided not to present its proof here. Following that is a sequence of lemma containing some derivative estimates. These estimates are used mostly in proving the results in Theorem~\ref{linear} to obtain bounds on various derivatives of the Wronskian.

\begin{lemma}\label{phi1 phi2}
For any $\lambda$, there exists $0<\epsilon_0\leq 1$, such that for any $c \in \mathfrak{B} \cap \mathfrak{D}_{\epsilon_0}\subset \mathbb{C}$, the Rayleigh equation \eqref{Rayleigh Equation} has a regular solution,
\[
\phi(y,m,\gamma,c,\lambda)=(U_{m,\gamma}(y)-c)\phi_1(y,m,\gamma,c_r,\lambda)\phi_2(y,m,\gamma,c,\lambda),
\]
which satisfies the conditions: $\phi(y_c,m,\gamma,c,\lambda)=0, \phi'(y_c,m,\gamma,c,\lambda)=U'_{m,\gamma}(y_c),$
where $U_{m,\gamma}(y_c)=c_r$. The function $\phi_1$ is a real-valued function that solves
\begin{equation}
    \left\{ \begin{aligned}
    &\partial_y\Big((U_{m,\gamma}(y)-c_r)^2\phi'_1(y,m,\gamma,c_r,\lambda)\Big)+\lambda\phi_1(y,m,\gamma,c_r,\lambda)(U_{m,\gamma}(y)-c_r)^2=0,\\
    & \phi_1(y_c,m,\gamma,c_r,\lambda)=1\qquad \phi'_1(y_c,m,\gamma,c_r,\lambda)=0. 
\end{aligned} \right.
\end{equation}
Meanwhile, the function $\phi_2$ is a complex-valued function that solves
\begin{equation}\label{DE for phi2}
    \left\{ \begin{aligned}
    &\partial_y\Big((U_{m,\gamma}(y)-c)^2\phi^2_1(y,m,\gamma,c_r,\lambda)\phi'_2(y,m,\gamma,c,\lambda)\Big)\\
    &\qquad \qquad \qquad +\dfrac{2ic_iU'_{m,\gamma}(y)(U_{m,\gamma}(y)-c)}{U_{m,\gamma}(y)-c_r}\phi_1(y,m,\gamma,c_r,\lambda)\phi'_1(y,m,\gamma,c_r,\lambda)\phi_2(y,m,\gamma,c,\lambda)=0,\\
    & \phi_2(y_c,m,\gamma,c,\lambda)=1\qquad \phi'_2(y_c,m,\gamma,c,\lambda)=0. 
\end{aligned} \right.
\end{equation}
Explicitly, $\phi_1$ and $\phi_2$ take the form
\begin{equation}\label{general expression of phi1}
\phi_1(y,m,c_r,\lambda)=1+\int_{y_c}^y \dfrac{-\lambda}{(U_{m,\gamma}(w)-c_r)^2}\int_{y_c}^{w}\phi_1(z,\lambda)(U_{m,\gamma}(z)-c_r)^2\;dz\;dw.
\end{equation}
and
\[
\begin{aligned}
    &\phi_2(y,m,\gamma,c,\lambda)=1\\
    &-2ic_i\int_{y_c}^y \dfrac{1}{(U_{m,\gamma}(w)-c)^2\phi^2_1(w,m,\gamma,c_r,\lambda)}\\
    &\quad\quad\quad\times \int_{y_c}^{w}\dfrac{U'_{m,\gamma}(z)(U_{m,\gamma}(z)-c)}{(U_{m,\gamma}(z)-c_r)}\phi'_1(z,m,\gamma,c,\lambda)\phi_1(z,m,\gamma,c,\lambda)\phi_2(z,m,\gamma,c,\lambda)\;dz\;dw.
\end{aligned}
\]
Moreover, we have the following estimates for $\phi_1$, $\phi_2$ and their derivatives,
\begin{align*}
    &1\leq \phi_1(y_1,m,\gamma,c_r,\lambda)\leq \phi_1(y_2,m,\gamma,c_r,\lambda)\quad \text{for}\quad |y_2-y_c|\geq |y_1-y_c|,\\
    &C^{-1}e^{C^{-1}\sqrt{-\lambda}|y-y_c|}\leq \phi_1(y,m,\gamma,c_r,\lambda)\leq Ce^{C\sqrt{-\lambda}|y-y_c|}\\
    &\phi_1(y,m,\gamma,c_r,\lambda)-1\leq C\min \{|\lambda||y-y_c|^2,1\}\phi_1(y,m,\gamma,c,\lambda)\\
    &\left|\frac{\partial_y\phi_1(y,m,\gamma,c_r,\lambda)}{\phi_1(y,m,\gamma,c_r,\lambda)}\right|\leq C\sqrt{-\lambda}|y-y_c|,\quad \left|\frac{\partial_{yy}\phi_1(y,m,\gamma,c_r,\lambda)}{\phi_1(y,m,\gamma,c_r,\lambda)}\right|\leq C|\lambda|
\end{align*}
and
\begin{align}
    &|\phi_{2}(y,m,\gamma,c,\lambda)-1|\le C\min \{\sqrt{-\lambda}|c_i|, |\lambda||y-y_c|^2\},\label{eq-phi2-est-1}\\
      &|\partial_y\phi_{2}(y,m,\gamma,c,\lambda)|\le C|\lambda|\min \{c_i, |y-y_c|\},\label{eq-phi2-est-2}\\
      & \|\partial_{yy}\phi_{2}(y,m,\gamma,c,\lambda)\|_{L^\infty_y}\le C|\lambda|,\label{eq-phi2-est-3}
  \end{align}

The constant $C$ is independent of $\gamma, y, y_c, m, \lambda,c$. 
\end{lemma}
\begin{proof}
    The proof is omitted here. We refer reader to \cite[Proposition 5.3] {limasmoudizhao2022} for the estimate of $\phi_2(y,m,c,\lambda)$ and to \cite{WeiZhangZhao2015} for more accurate estimate of $\phi_1(y,m,c_r,\lambda)$. There, the authors presented a complete argument to obtain all the estimates. Note that the estimates relays only on the upper and lower bound of $U'_{m,\gamma}$ and $\|U''_{m,\gamma}(y)\|_{L^{\infty}_y}$. Thus the constant is independent of $\gamma$. 
\end{proof}

Earlier in \eqref{good derivative} we define the notion of good derivative denoted by $\partial_G$. Here, we present an estimate when the good derivative acts on the function $\phi_2$ from Lemma~\ref{phi1 phi2}.

Let us introduce $\phi_1^{E}(y,m,c,\lambda)=\phi_1(y,m,\gamma,c_r,\lambda)\phi_2(y,m,\gamma,c,\lambda)$. Note that $\phi_1^{E}(y,m,\gamma,c,\lambda)=\phi_1(y,m,\gamma,c_r,\lambda)$ for $c=c_r$, since $\phi_2(y,m,\gamma,c,\lambda)\equiv 1$ for $c=c_r$. One can regard $\phi_1^{E}(y,m,\gamma,c,\lambda)$ as the complex extension of $\phi_1(y,m,\gamma,c_r,\lambda)$. The next lemma gives a good property of $\phi_1^{E}(y,m,\gamma,c,\lambda)$ when the `good derivative' $\partial_{c_i}+i\frac{\partial_y}{U_{m,\gamma}'(0)}$ acts on it. We present the lemma by taking $c_r=0$ and $\lambda=-1$, which is used in the proof. 

\begin{lemma}[Estimates on $\partial_c \phi$ ]\label{partial c phi1}
    Let $\phi$ be the function displayed in Lemma~\ref{phi1 phi2}. For $c=c_r+ic_i$ with $U_{m,\gamma}(y_c)=c_r$ and $\lambda=-1$, we have
\begin{equation}
\left| \partial_c \phi(y,m,\gamma,c_r+ic_i,-1)\right| \leq C,
\end{equation}
where the constant $C$ is independent of $\gamma,m,c$ and $\partial_c:=\partial_{c_r}+i\partial_{c_i}$.
\end{lemma}
\begin{proof}
As stated before in Lemma~\ref{phi1 phi2}, we know  that $\phi(y,m,\gamma,c_r+ic_i,\lambda)$ solves the Rayleigh equation \eqref{Rayleigh Equation} with boundary conditions $\phi(y_c,m,\gamma,c,\lambda)=0, \phi'(y_c,m,\gamma,c,\lambda)=U'_{m,\gamma}(y_c),$
where $U_{m,\gamma}(y_c)=c_r$. 

Applying $\partial_c$ to both sides of $\eqref{Rayleigh Equation}$ gives us
\begin{equation}\label{eq satisfied by partial c phi}
    -\partial^2_y \partial_c\phi(y,m,\gamma,c,k)+\dfrac{U''_{m,\gamma}}{U_{m,\gamma}-c}\partial_c\phi(y,m,\gamma,c,k)+\dfrac{U''_{m,\gamma}}{(U_{m,\gamma}-c)^2}\phi(y,m,\gamma,c,k)=-k^2\Phi(y,m,\gamma,c,k),
\end{equation} 
Consider an ODEtz for \eqref{eq satisfied by partial c phi}  which takes the form $\partial_c(\phi)= A \phi$ for some $A$ to be determined. By plugging it into \eqref{eq satisfied by partial c phi} and using the fact that $\phi$ solves \eqref{Rayleigh Equation}, we are able to say that 
\[
\partial_c\phi(y,m,\gamma,c,k)=\phi(y,m,\gamma,c,k)\int_{y_c}^{y}\dfrac{\int_{y_c}^{w}-U''_{m,\gamma}(z)(\phi_1^E)^2(z,m,\gamma,c,k) \;dz}{(U_{m,\gamma}(w)-c)^2(\phi_1^E)^2(w,m,\gamma,c,k)}\;dw,
\]
where $\phi_1^E(y,m,\gamma,c,\lambda)=\phi_1(y,m,\gamma,c_r,\lambda) \phi_2(y,m,\gamma,c,\lambda)$.

Taking advantage of the above representation and using the boundedness of $\phi_1$ and $\phi_2$ in Lemma~\ref{phi1 phi2}, we obtain
\[
|\partial_c\phi| \leq C.
\]
\end{proof}

\begin{lemma} [Estimates on $\partial_G \phi_1^E$ ]\label{partial G phi1}
Let $\phi_1^E(y,m,\gamma,c,\lambda)=\phi_1(y,m,\gamma,c_r,\lambda) \phi_2(y,m,\gamma,c,\lambda)$. For $c=c_r+ic_i$ with $U_{m,\gamma}(y_c)=c_r$ and $\lambda=-1$, we have
\begin{equation}\label{estimate on the good derivative of phi1E}
\left| \frac{\partial_G \phi_1^{E}(y,m,\gamma,c_r+ic_i,-1) }{\phi_1^{E}(y,m,\gamma,c_r+ic_i,-1)}\right|\leq C|y-y_c|^2,
\end{equation}
where the constant $C$ is independent of $\gamma,m,c$ and $\partial_G:=\partial_{c_i}+i\frac{\partial_y}{U'_{m,\gamma}(y_c)}$.
\end{lemma}

\begin{proof}
From Lemma~\ref{phi1 phi2}, we have $\phi_1^E(y,m,\gamma,c_r+ic_i,-1)=\phi(y,m,\gamma,c_r+ic_i,-1) /(U_{m,\gamma}(y)-U_{m,\gamma}(y_c)-ic_i)$. As a consequence, $\phi_1^{E}(y,m,\gamma,c_r+ic_i,-1)$ solves 

\begin{equation}\label{eq satisfied by phi 1E}
\bigg((U_{m,\gamma}(y)-U_{m,\gamma}(y_c)-ic_i)^2(\phi_1^E)'(y,m,\gamma,c,-1)\bigg)'=\phi_1^E(y,m,\gamma,c,-1)(U_{m,\gamma}(y)-U_{m,\gamma}(y_c)-ic_i)^2.
\end{equation}
Applying the good derivative $\partial_G$ to both sides of the above equation,
we arrive at 
\begin{equation}\label{partial G DE}
    \begin{aligned}
   &\bigg((U_{m,\gamma}(y)-U_{m,\gamma}(y_c)-ic_i)^2\partial_G(\phi_1^E)'(y,m,\gamma,c_r+ic_i,-1)\bigg)'\\&\quad=(\partial_G(\phi_1^E(y,m,\gamma,c_r+ic_i,-1)))(U_{m,\gamma}(y)-U_{m,\gamma}(y_c)-ic_i)^2\\&\qquad- \bigg(\dfrac{\partial_G\Big((U_{m,\gamma}(y)-U_{m,\gamma}(y_c)-ic_i)^2\Big)}{(U_{m,\gamma}(y)-U_{m,\gamma}(y_c)-ic_i)^2}\bigg)'(U_{m,\gamma}(y)-U_{m,\gamma}(y_c)-ic_i)^2(\phi_1^E)'(y,m,\gamma,c_r+ic_i,-1).
\end{aligned}
\end{equation}

Moreover, it is equivalent to saying that
\begin{equation}
    \begin{aligned}
   &\bigg((U_{m,\gamma}(y)-U_{m,\gamma}(y_c)-ic_i)^2\left(\frac{\partial_G\phi_1^E(y,m,\gamma,c_r+ic_i,-1)}{\phi_1^E(y,m,\gamma,c_r+ic_i,-1)}\right)'(\phi_1^E)^2(y,m,\gamma,c_r+ic_i,-1)\bigg)'\\&\quad=- \bigg(\dfrac{\partial_G\Big((U_{m,\gamma}(y)-U_{m,\gamma}(y_c)-ic_i)^2\Big)}{(U_{m,\gamma}(y)-U_{m,\gamma}(y_c)-ic_i)^2}\bigg)'(U_{m,\gamma}(y)-U_{m,\gamma}(y_c)-ic_i)^2\\&\qquad \qquad \times (\phi_1^E)'(y,m,\gamma,c_r+ic_i,-1)\phi_1^E(y,m,\gamma,c_r+ic_i,-1).
\end{aligned}
\end{equation}
Solving for $\frac{\partial_G\phi_1^E}{\phi_1^E}$ gives us
\begin{equation}\label{good derivative phi1E}
\begin{aligned}
    &\frac{\partial_G\phi_1^E(y,m,\gamma,c_r+ic_i,-1)}{\phi_1^E(y,m,\gamma,c_r+ic_i,-1)}
    \\&\quad=-\int_{y_c}^y\dfrac{1}{(U_{m,\gamma}(w)-U_{m,\gamma}(y_c)-ic_i)^2(\phi_1^E)^2(w,m,\gamma,c_r+ic_i,-1)}\\
    &\quad\quad\times\int_{y_c}^w\bigg(\dfrac{\partial_G\Big((U_{m,\gamma}(z)-U_{m,\gamma}(y_c)-ic_i)^2\Big)}{(U_{m,\gamma}(z)-U_{m,\gamma}(y_c)-ic_i)^2}\bigg)'(U_{m,\gamma}(z)-U_{m,\gamma}(y_c)-ic_i)^2\\& \quad \quad \times (\phi_1^E)'(z,m,\gamma,c_r+ic_i,-1)\phi_1^E(z,m,\gamma,c_r+ic_i,-1)\;dz\;dw.
\end{aligned}
\end{equation}
Let us take a closer look at the term above involving the good derivative,
\begin{equation}\label{expression of good derivative}
\begin{aligned}
&\bigg|\bigg(\dfrac{\partial_G\Big((U_{m,\gamma}(z)-U_{m,\gamma}(y_c)-ic_i)^2\Big)}{(U_{m,\gamma}(z)-U_{m,\gamma}(y_c)-ic_i)^2}\bigg)'\bigg|
\\&=\bigg|\dfrac{2\bigg(i\frac{U''_{m,\gamma}(z)}{U'_{m,\gamma}(y_c)}(U_{m,\gamma}(z)-U_{m,\gamma}(y_c)-ic_i)-(-i+i\frac{U'_{m,\gamma}(z)}{U'_{m,\gamma}(y_c)})U'_{m,\gamma}(z)\bigg)}{(U_{m,\gamma}(z)-U_{m,\gamma}(y_c)-ic_i)^2}\bigg|\\&
=\bigg|2\dfrac{\bigg(c_i\frac{U''_{m,\gamma}(z)}{U'_{m,\gamma}(y_c)}+i\bigg(\frac{U''_{m,\gamma}(z)}{U'_{m,\gamma}(y_c)}U_{m,\gamma}(z)+U'_{m,\gamma}(z)-\frac{(U'_{m,\gamma}(z))^2}{U'_{m,\gamma}(y_c)}-\frac{U_{m,\gamma}(y_c)U''_{m,\gamma}(z)}{U'_{m,\gamma}(y_c)}\bigg)}{(U_{m,\gamma}(z)-U_{m,\gamma}(y_c)-ic_i)^2}\bigg|\\&
\lesssim \frac{\lVert U''_{m,\gamma}(z)\rVert_{L^{\infty}}\lVert U'_{m,\gamma}(z)\rVert_{L^{\infty}}}{U'_{m,\gamma}(y_c)}\dfrac{1}{|z-y_c|+|c_i|}.
\end{aligned}
\end{equation}

It is clear that the expression on the right hand side of \eqref{expression of good derivative} has one singularity at $(z,c_i)=(y_c,0)$. When plugging this term back into the integral in \eqref{good derivative phi1E}, the simple singularity of \eqref{expression of good derivative} in the integrand can be cancelled by the zero of $(\phi_1^E)'$ and the two repeated zeros of $(U_{m,\gamma}(z)-U_{m,\gamma}(y_c)-ic_i)^2$ at $(z,c_i)=(y_c,0)$ can also be cancelled by the quadratic poles of $\phi'_1$ at the same point. More precisely, notice that 
\begin{equation}
\begin{aligned}\label{expansion of phi_1^E'}
 (\phi_1^E)'&=\phi'_1\phi_2+\phi_1\phi_2' \\&
            =\phi_2  \dfrac{-\lambda}{(U_{m,\gamma}(y)-c_r)^2}\int_{y_c}^{y}\phi_1(z,\lambda)(U_{m,\gamma}(z)-c_r)^2\;dz\\&
            \quad -i2c_i\phi_1  \dfrac{1}{(U_{m,\gamma}(y)-c)^2\phi^2_1(y,m,\gamma,c_r,\lambda)}\int_{y_c}^{y}\dfrac{U'_{m,\gamma}(z)(U_{m,\gamma}(z)-c)}{(U_{m,\gamma}(z)-c_r)}\phi'_1\phi_1\phi_2\;dz.
\end{aligned}
\end{equation}
Plugging it back into the integrand in \ref{good derivative phi1E} allows us to say that the double poles at $(y,c_i)=(y_c,0)$ of $(\phi_1^E)'$ brought about by the terms $(U_{m,\gamma}(y)-c_r)^2$ and $(U_{m,\gamma}(y)-c)^2$ can both be cancelled by the quadratic vanishing of the term  $(U_{m,\gamma}(z)-U_{m,\gamma}(y_c)-ic_i)^2$ in the integrand also at the same point. Furthermore, the simple pole due the term in the integrand of \ref{good derivative phi1E} involving $\partial_G$ is also annihilated by the vanishing of $(\phi_1^E)'$ at $(y,c_i)=(y_c,0)$ which can be observed from \eqref{expansion of phi_1^E'}.
Lastly, all terms in the integrand are all bounded as a consequence of estimates in Lemma~\ref{phi1 phi2}. Therefore, the integral is well-defined. Altogether, we obtain the estimate \eqref{estimate on the good derivative of phi1E}. Hence, the proof is complete.\qedhere
\end{proof}

\begin{lemma} [Estimates on $\partial_m \phi_1^E$ ]\label{partial m phi1}
Let $\phi_1^E(y,m,\gamma,c,\lambda)=\phi_1(y,m,\gamma,c_r,\lambda) \phi_2(y,m,c,\lambda)$. For $c=c_r+ic_i$ and $\lambda=-1$, we have
\[
\left| \frac{\partial_m \phi_1^{E}(y,m,\gamma,c_r+ic_i,-1) }{\phi_1^{E}(y,m,\gamma,c_r+ic_i,-1)}\right|\leq C\gamma|y-y_c|^2,
\]
with constant $C$ independent of $\gamma,m,c,y$.
\end{lemma}
\begin{proof}
Here, the proof is a straight adaptation of the one in Lemma~\ref{partial G phi1}.
Again, from Lemma~\ref{phi1 phi2}, we have $\phi_1^E(y,m,\gamma,c_r+ic_i,-1)=\phi(y,m,\gamma,c_r+ic_i,-1) /(U_{m,\gamma}(y)-U_{m,\gamma}(y_c)-ic_i)$. As a consequence,  $\phi_1^{E}$ solves \eqref{eq satisfied by phi 1E}.
Upon applying  $\partial_m$ to both sides of $\eqref{eq satisfied by phi 1E}$, we arrive at 
\begin{equation}
    \begin{aligned}
   &\bigg((U_{m,\gamma}(y)-U_{m,\gamma}(y_c)-ic_i)^2\partial_m(\phi_1^E)'(y,m,\gamma,c_r+ic_i,-1)\bigg)'\\&
   =(\partial_m(\phi_1^E(y,m,\gamma,c_r+ic_i,-1)))(U_{m,\gamma}(y)-U_{m,\gamma}(y_c)-ic_i)^2\\&- \bigg(\dfrac{\partial_m\Big((U_{m,\gamma}(y)-U_{m,\gamma}(y_c)-ic_i)^2\Big)}{(U_{m,\gamma}(y)-U_{m,\gamma}(y_c)-ic_i)^2}\bigg)'(U_{m,\gamma}(y)-U_{m,\gamma}(y_c)-ic_i)^2(\phi_1^E)'(y,m,\gamma,c_r+ic_i,-1).
\end{aligned}
\end{equation}

Moreover, it is equivalent to saying that
\begin{equation}
    \begin{aligned}
   &\bigg((U_{m,\gamma}(y)-U_{m,\gamma}(y_c)-ic_i)^2\left(\frac{\partial_m\phi_1^E(y,m,\gamma,c_r+ic_i,-1)}{\phi_1^E(y,m,\gamma,c_r+ic_i,-1)}\right)'(\phi_1^E)^2(y,m,\gamma,c_r+ic_i,-1)\bigg)'\\&\quad=- \bigg(\dfrac{\partial_m\Big((U_{m,\gamma}(y)-U_{m,\gamma}(y_c)-ic_i)^2\Big)}{(U_{m,\gamma}(y)-U_{m,\gamma}(y_c)-ic_i)^2}\bigg)'(U_{m,\gamma}(y)-U_{m,\gamma}(y_c)-ic_i)^2\\&
   \quad \quad  \times (\phi_1^E)'(y,m,\gamma,c_r+ic_i,-1)\phi_1^E(y,m,\gamma,c_r+ic_i,-1).
\end{aligned}
\end{equation}
Solving for $\frac{\partial_m\phi_1^E}{\phi_1^E}$ gives us
\[
\begin{aligned}
    &\frac{\partial_m\phi_1^E(y,m,\gamma,c_r+ic_i,-1))}{\phi_1^E(y,m,\gamma,c_r+ic_i,-1))}
    \\&\quad=-\int_{y_c}^y\dfrac{1}{(U_{m,\gamma}(w)-U_{m,\gamma}(y_c)-ic_i)^2)(\phi_1^E)^2(w,m,\gamma,c_r+ic_i,-1)}\\
    &\quad\quad\times\int_{y-y_c}^w\bigg(\dfrac{\partial_m\Big((U_{m,\gamma}(z)-U_{m,\gamma}(y_c)-ic_i)^2\Big)}{(U_{m,\gamma}(z)-U_{m,\gamma}(y_c)-ic_i)^2}\bigg)'(U_{m,\gamma}(z)-U_{m,\gamma}(y_c)-ic_i)^2\\&
    \quad \quad \times (\phi_1^E)'(z,m,\gamma,c_r+ic_i,-1)\phi_1^E(z,m,\gamma,c_r+ic_i,-1)\;dz\;dw.
\end{aligned}
\] 
Direct computation tells us that
\[
\begin{aligned}
    \bigg|\bigg(\dfrac{\partial_m\Big((U_{m,\gamma}(z)-ic_i)^2\Big)}{(U_{m,\gamma}(z)-U_{m,\gamma}(y_c)-ic_i)^2}\bigg)'\bigg|&\leq \lVert\Gamma \rVert_{L^{\infty}} \dfrac{\bigg|2\gamma \bigg((U_{m,\gamma}(z)-U_{m,\gamma}(y_c)-ic_i)+\gamma z U'_{m,\gamma}(z)\bigg)\bigg|}{(U_{m,\gamma}(z)-U_{m,\gamma}(y_c)-ic_i)^2}\\&
    \leq 2\gamma \lVert\Gamma \rVert_{L^{\infty}} \dfrac{\Big|\lVert U'_{m,\gamma}(z)\rVert_{L^{\infty}}\Big|}{|z-y_c|+|c_i|}.
\end{aligned}
\]

Observe that the term above has one singularity. In particular, upon plugging in this estimate into the inner integral in the expression of $\partial_m\phi_1^E/\phi_1^E$, the singularity in the inner integrand will be cancelled by the zero of $(\phi_1^E)'$. Further, the two singularities in the outter integrand can be overcome by the two zeros of $(U_{m,\gamma}(z)-U_{m,\gamma}(y_c)-ic_i)^2$. This results in the integral to be well-defined. Hence, we arrive at
\[
\bigg|\frac{\partial_m\phi_1^E(y,m,\gamma,c_r+ic_i,-1))}{\phi_1^E(y,m,\gamma,c_r+ic_i,-1))}\bigg|\leq C \gamma (y-y_c)^2.
\]
We would like to remark that the constant $C$ is independent of $\gamma$ and obtained by using the estimates in Lemma~\ref{phi1 phi2}. Hence, the proof is complete. \qedhere
\end{proof}

\begin{lemma}\label{F}
    Consider the function $\phi_1$ in \eqref{phi1}. It follows that 
    \[
    -C_1\dfrac{y^2}{\sqrt{-\lambda}}\geq \dfrac{\partial_\lambda \phi_1(y,m,\gamma,c_r,\lambda)}{\phi_1(y,m,\gamma,c_r,\lambda)}\geq -C_0\dfrac{y^2}{\sqrt{-\lambda}},
    \]
for some positive constants $C_1$ and $C_0$. In addition, we have the following estimate for $|\partial_\lambda m|$,
\begin{equation}
|\partial_\lambda m| \approx \Big( 1-e^{-2 \sqrt{-\lambda}}\Big).
\end{equation}
where $\partial_\lambda m$ is given by \eqref{monotonicity of m}.
\end{lemma}
    \begin{proof}
        Recall that $\phi_1$ solves the equation \eqref{differential equation for phi1}. Differentiating \eqref{differential equation for phi1}  with respect to $\lambda$ yields
\begin{equation} \label{DE for partial lambda phi1}
(U^2(y) \partial_{\lambda}\phi'_1(y,m,\gamma,c_r,\lambda))'=-\lambda\partial_{\lambda} \phi_1(y,m,\gamma,c_r,\lambda)U^2-\phi_1(y,m,\gamma,c_r,\lambda) U^2(y).
\end{equation}
We think of $\partial_\lambda \phi_1$ as an unknown and consider the ansatz  \begin{equation}\label{ansatz for DE partial lambda phi1}
    \partial_{\lambda} \phi_1(y,m,\gamma,c_r,\lambda)= \mathfrak{F}(\phi_1,y)\phi_1(y,m,\gamma,c_r,\lambda)
\end{equation} for \eqref{DE for partial lambda phi1} . Plugging it into  \eqref{DE for partial lambda phi1} gives
\begin{equation*}
\begin{aligned}
&-\lambda \mathfrak{F}(\phi_1,y) \phi_1(y,m,\gamma,c_r,\lambda) U^2(y) -\phi_1(y,m,\gamma,c_r,\lambda) U^2(y)\\&\qquad=(U^2(y)\mathfrak{F}'(\phi_1,y)\phi_1(y,m,\gamma,c_r,\lambda) )'+(U^2(y) \mathfrak{F}(\phi_1,y) \phi'_1(y,m,\gamma,c_r,\lambda) )'\\
                                                      &\qquad=(U^2(y)\mathfrak{F}'(\phi_1,y)\phi_1(y,m,\gamma,c_r,\lambda) )'+(U^2(y) \phi'_1)'\mathfrak{F}(\phi_1,y) +U^2(y) \mathfrak{F}'(\phi_1,y)\phi'_1(y,m,\gamma,c_r,\lambda) \\
                                                      &\qquad=(U^2(y)\mathfrak{F}'(\phi_1,y)\phi_1)'-\lambda \mathfrak{F}(\phi_1,y)\phi_1(y,m,\gamma,c_r,\lambda)  U^2(y)+U^2(y) \mathfrak{F}'(\phi_1,y)\phi'_1(y,m,\gamma,c_r,\lambda) ,
\end{aligned}
\end{equation*}
from which we can say
\begin{equation}
\begin{aligned}
0&=(U^2(y))'\mathfrak{F}'(\phi_1,y) \phi^2_1(y,m,\gamma,c_r,\lambda)+U^2(y) \mathfrak{F}''(\phi_1,y) \phi^2_1(y,m,\gamma,c_r,\lambda)\\&\qquad +2U^2(y)\mathfrak{F}'(\phi_1,y) \phi'_1(y,m,\gamma,c_r,\lambda)\phi_1(y,m,\gamma,c_r,\lambda)+\phi^2_1(y,m,\gamma,c_r,\lambda)U^2(y)\\
&=(U^2(y) \mathfrak{F}'(\phi_1,y) \phi^2_1(y,m,\gamma,c_r,\lambda))'+\phi^2_1(y,m,\gamma,c_r,\lambda) U^2(y).
\end{aligned}
\end{equation} Straightforward computation yields
\begin{equation}
\mathfrak{F}(\phi_1,y)=-\int_0^y\dfrac{\int_{0}^{w}\phi^2_1(y,m,\gamma,c_r,\lambda) U^2(z) \; dz}{U^2(w) \phi^2_1(y,m,\gamma,c_r,\lambda)}\;dw\leq 0.
\end{equation}

Furthermore, we compute more precise upper and lower bounds of $\mathfrak{F}$. Without lost of generality, we assume that $w>0$ (i.e. $0\leq z<w$). Hence, we deduce the lower bound for $\mathfrak{F}$:
\begin{equation}
\begin{aligned}
0&\geq- \int_0^y\int_{0}^{w}\dfrac{\phi^2_1(z,m,\gamma,c_r,\lambda)U^2(z) \; dz}{U^2(w) \phi^2_1(w,m,\gamma,c_r,\lambda)}\;dw\\
  &\geq- \int_0^y\int_{0}^{w}e^{-\sqrt{-\lambda}|w-z|}\dfrac{U^2(z) \; dz}{U^2(w)}\;dw
  \geq \dfrac{y^2}{-\sqrt{-\lambda}},
\end{aligned}
\end{equation}
where we have used the inequality for $\phi_1$ in Lemma~\ref{phi1 phi2}. In the case when $w<0$, the same argument holds. Again, via a straightforward computation using the estimate for $\phi_1$ in Lemma~\ref{phi1 phi2}, we obtain an upper bound of $\mathfrak{F}$, namely $\mathfrak{F} \leq -C y^2/ \sqrt{-\lambda}$, where $C$ is some constant. Combining both bounds, we get
\begin{equation}\label{estimate for A}
-C_1\dfrac{y^2}{\sqrt{-\lambda}}\geq \mathfrak{F} \geq -C_0\dfrac{y^2}{\sqrt{-\lambda}}.
\end{equation}

Having obtained the estimate for $\mathfrak{F}$, we now proceed to derive the bound for $\partial_{\lambda} m$. We list out a number of inequalities that are useful for the derivation. First, we present a lower bound for $\partial_{\lambda} m$,

\begin{equation}
\begin{aligned}
|\partial_{\lambda} m| \gtrsim \int_{-1}^1 \dfrac{1}{\phi_1^2} \; dy &\gtrsim \Bigg( \int_{-1}^0 e^{2 \sqrt{-\lambda}y}\;dy + \int_{0}^1e^{-2 \sqrt{-\lambda}y}\;dy\Bigg)\\
&\gtrsim \dfrac{1}{\sqrt{-\lambda}} \bigg( 1-e^{-2 \sqrt{-\lambda}}\bigg).
\end{aligned}
\end{equation}
Via the same computations, we show that $\partial_{\lambda}m$ satisfies the following upper bound
\begin{equation}
\begin{aligned}
|\partial_{\lambda} m| \lesssim \int_{-1}^1 \dfrac{1}{\phi_1^2} \; dy &\lesssim \Bigg( \int_{-1}^0 e^{2 \sqrt{-\lambda}y}\;dy + \int_{0}^1e^{-2 \sqrt{-\lambda}y}\;dy\Bigg)\\
&\lesssim \dfrac{1}{\sqrt{-\lambda}} \bigg( 1-e^{-2 \sqrt{-\lambda}}\bigg).
\end{aligned}
\end{equation}

Finally, using all inequalities displayed above and inserting them into the expression in \eqref{monotonicity of m} gives us
\begin{equation}
|\partial_\lambda m| \approx \bigg( 1-e^{-2 \sqrt{-\lambda}}\bigg).
\end{equation}
\end{proof}

Next, we present you a lemma that provides a useful estimate in computing the bounds for $\partial_{c_i}W$ and $\partial_m W$. 
\begin{lemma} \label{Supremum estimate}
    Let $a<0<b$, for all $f\in H^{1}(a,b)$, we have the following estimate
    \begin{equation}
        \sup _{v\in [\frac{a}{2},\frac{b}{2}]}\sup_{c_i\in (0,1]}\left|\int_{a}^{b} \dfrac{1}{v-w+ic_i} f(v)\;dv\right|\lesssim \lVert f \rVert^{1/2}_{L^2}\bigg(\lVert f' \rVert^{1/2}_{L^2}+\lVert f \rVert^{1/2}_{L^\infty}\bigg).
    \end{equation}
\end{lemma}
\begin{proof}
    Define the following integral operator
    \[
    T_{c_i,a,b}f(v):=\int_{a}^b \dfrac{1}{v-w+ic_i}f(w)\;dw. 
    \]
Hence,
\begin{equation} \label{bound on A}
   \begin{aligned}
       &\sup_{c_i\in(0,1]}\left|p.v.\int_{a}^{b} \dfrac{1}{v-ic_i} f(v)\;dv\right|\leq \sup_{\substack{c_i \in (0,1] \\ v \in [\frac{a}{2},\frac{b}{2}]}}\Big|T_{c_i,a,b}f(v)\Big|\\ 
       &\lesssim \sup_{c_i \in (0,1]}\bigg(\lVert T_{c_i,a,b}f (v)\rVert^{1/2}_{L^2_v}\lVert \partial_v T_{c_i,a,b}f(v) \rVert^{1/2}_{L^2_v} + \lVert T_{c_i,a,b}f(v)\rVert_{L^2_v}\bigg). 
   \end{aligned}
\end{equation}

Observe that via integration by parts, we obtain
\[
\partial_v T_{c_i,a,b}f(v)=\dfrac{1}{v-b+ic_i}f(b)-\dfrac{1}{v-a+ic_i}f(a)+T_{c_i,a,b}(f'(v)).
\] Moreover, we have the following estimates:
\[
\begin{aligned}
    \sup_{c_i\in (0,1]}\lVert T_{c_i,a,b}f(v) \rVert_{L^2_v} &\lesssim \lVert f(v) \rVert_{L^2_v},\\
    \sup_{c_i\in (0,1]}\lVert \partial_v T_{c_i,a,b}f(v) \rVert_{L^2_v} &\lesssim \lVert f(v) \rVert_{L^\infty_v}+\sup_{\substack{c_i \in (0,1] }}\Big\| T_{c_i,a,b} f'(v) \Big\|_{L^2}\\
    &\lesssim \lVert f(v) \rVert_{L^\infty_v} +\sup_{c_i\in(0,1]} \lVert f'(v) \rVert_{L^2}.
\end{aligned}
\]
Combining all these estimates and plug them back into \eqref{bound on A} gives the desired bound in the statement of the lemma. \qedhere
\end{proof}
\begin{remark} \label{remark A7}
    For any $f\in H^{1}(-1,1)$, we have
    \[
    \begin{aligned}
    \int_{-1}^1 \dfrac{f(y)-f(y_c)}{(U_{m,\gamma}(y)-U_{m,\gamma}(y_c)-ic_i)^2}\;dy&=-\int_{-1}^1\dfrac{f(y)-f(y_c)}{U'_{m,\gamma}(y)} \partial_y \bigg(\dfrac{1}{(U_{m,\gamma}(y)-U_{m,\gamma}(y_c)-ic_i)}\bigg)\;dy\\&=\dfrac{f(y_c)-f(y)}{(U_{m,\gamma}(y)-U_{m,\gamma}(y_c)-ic_i)^2}\bigg(\dfrac{1}{(U_{m,\gamma}(y)-U_{m,\gamma}(y_c)-ic_i)}\bigg)\bigg|_{-1}^1\\& \qquad+ \int_{-1}^1 \dfrac{1}{(U_{m,\gamma}(y)-U_{m,\gamma}(y_c)-ic_i)}\partial_y\bigg(\dfrac{f(y)-f(y_c)}{U'_{m,\gamma}(y)}\bigg)\;dy.
    \end{aligned}
    \] Hence, via Lemma~\ref{Supremum estimate} one can estimate the above integral and obtain its upper bound. More precisely,
    \begin{equation}
    \begin{aligned}
      \sup_{y\in [\frac{1}{2},\frac{1}{2}]}  \sup_{c_i\in (0,1]}&\left|\int_{-1}^1 \dfrac{f(y)-f(y_c)}{(U_{m,\gamma}(y)-U_{m,\gamma}(y_c)-ic_i)^2}\;dy\right|\\&\lesssim \lVert \partial_y\bigg(\dfrac{f(y)-f(y_c)}{U'_{m,\gamma}(y)}\bigg)\rVert^{1/2}_{L^2}\\&
        \quad \times \bigg(\lVert \partial^2_y\bigg(\dfrac{f(y)-f(y_c)}{U'_{m,\gamma}(y)}\bigg) \rVert^{1/2}_{L^2}+\lVert \partial_y\bigg(\dfrac{f(y)-f(y_c)}{U'_{m,\gamma}(y)}\bigg) \rVert^{1/2}_{L^\infty}\bigg).
    \end{aligned}
    \end{equation}
\end{remark}

The following corollary is the consequence the previous lemma. 
\begin{corollary}\label{bounds on integral}
 Let $f \in H^{1}(-1,1) $, the following estimate holds
\[
\bigg|\sup_{0<|c_r|<1} \int_{-1}^{1}\dfrac{(U_{m,\gamma}(y)-c_r)}{(U_{m,\gamma}(y)-c_r)^2+c_i^2} f(y) \;dy\bigg| \leq C \lVert f \rVert^{1/2}_{L^2}\bigg(\lVert f' \rVert^{1/2}_{L^2}+\lVert f \rVert^{1/2}_{L^\infty}\bigg).
\]
\end{corollary}
\begin{proof}
The proof of this corollary follows directly from the proof of Lemma \eqref{Supremum estimate}. However, one has to first introduce a change of variable, namely $v=U_{m,\gamma}(y)$. 
\end{proof}

\section{Regularity}\label{Regularity}
Next, this appendix concerns the regularity of the background flow and the map $G$ or its extension $\widetilde{G}$. Observe that the assumptions on $\mathcal{U}$ in the lemma is generic. We also do not assume any oddness or eveness on $\mathcal{U}$. We demand that $\mathcal{U}$ and its even derivatives (if exist) to vanish at $y=0$. The regularity used for $\mathcal{U}$ is also weaker compared to the one in \cite{LinZeng2011}.

In practice, when employing the bifurcation argument, we shall specify $(\mathcal{U}(y),0)$; we deal with the background shear $(U_{m,\gamma}(y),0)$. It is important to note that the conclusion on the higher regularity on $\widetilde{G}$ in this lemma is used to obtain the higher regularity result in Lemma~\ref{Bifurcation}.

\allowdisplaybreaks
\begin{lemma}\label{regularity lemma}
 Fix $N \in \mathbb{Z}^+ \cup \{0\}$. Suppose that $\mathcal{U} \in C^{2N+3}(-1,1)$, $\mathcal{U}'>c_0>0$, $\mathcal{U}(0)=0$, $\mathcal{U}^{(2j)}(0)=0$ for $j=1,2,...N,N+1$, and $-C_0<\mathcal{U}''/\mathcal{U}<0$ for some $c_0, C_0>0$. Consider a function $\widetilde{\psi}_0$ obtained by modifying the background stream function
\begin{equation}\label{definition of tilde psi}
\widetilde{\psi}_0(y):=\left\{ \begin{aligned}
\psi_0(y)\; &\text{for } y\geq0,\\
-\psi_0(y)\; &\text{for } y<0,
\end{aligned} \right.
\end{equation}
where $\int_{0}^y \mathcal{U}(z)\;dz=\psi_0(y)$.
If  $G(\widetilde{\psi_0}(y))=\mathcal{U}'(y)$, then $G(\widetilde{\psi_0}(y))\in C^{N+1}(-1,1).$
\end{lemma}
\begin{proof}
Let $\psi_0$ be the associated stream function of the background shear flow $(\mathcal{U},0)$. Since, $\mathcal{U}\in C^{2N+3}(-1,1)$, then $\psi_0 \in C^{2N+4}(-1,1).$ Further, from the monotonicity of $\mathcal{U}$ and the fact that $\mathcal{U}(0)=0,$ we can easily infer that $\mathcal{U}(y)<0$ for all $y\in[-1,0)$ and $\mathcal{U}(y)>0$ for all $y\in(0,1].$ As a result, the background stream function $\psi_0$ is decreasing on [-1,0] and increasing on [0,1]. Thus, $\widetilde{\psi}_0$ is increasing on $[-1,1]$. 

 Our method relies on the mathematical induction argument subject to the regularity exponent $N$. As the base case, we assume that 
 $N=0$. This is equivalent to saying that $\mathcal{U}\in C^3(-1,1)$. Differentiating $G$ in $y$ yields $G'(\widetilde{\psi}_0(y))= \mathcal{U}''(y)/\widetilde{\psi}_0'(y)$. Away from $y=0$, $G'(\widetilde{\psi}_0(y))$ is clearly well-defined. The subtlety occurs at $y=0$ 
as expected due to the definition of $\widetilde{\psi}_0$ in \eqref{definition of tilde psi}. In other words, we need to show that $\lim_{y\to 0^+}G'(\psi_0(y)$ and $\lim_{y\to 0^-}G'(-\psi_0(y)))$  exist and are  both equal. To prove this, we use the Taylor expansion of $\psi_0$ around $y=0$
\[
\psi_0(y)=\frac{\psi_0''(0)}{2}y^2+\frac{\psi_0^{(4)}(0)}{24}y^4+R(y):=F(y^2)+R(y),
\] where we have used the fact that $\mathcal{U}(0)=\mathcal{U}''(0)=0.$ Hence, the first derivative of $G$ can be expressed as follows
\begin{equation*}
\begin{aligned}
G'(\psi_0)&=\dfrac{\bigg(2F'(y^2)+4y^2F''(y^2)+R''(y)\bigg)'}{2yF'(y^2)+R'(y)}\\
                &=\dfrac{\bigg(w(y^2)+R''(y)\bigg)'}{2yF'(y^2)+R'(y)}\\
                &=\dfrac{w'(y^2)+\frac{R'''(y)}{2y}}{F'(y^2)+\frac{R'(y)}{2y}}:=\dfrac{a_1(y^2)+b_1(y)}{c_1(y^2)+d_1(y)}
\end{aligned}
\end{equation*} 
where we have done another grouping of terms and call it $w(y^2)$. Note that $\frac{R'''(y)}{2y}$ and $\frac{R'(y)}{2y}$ vanish as $y\to0^+$. Therefore, it is straightforward to see that $\lim_{y\to 0^+}G'(\psi_0(y))$ exists. Replacing $\psi_0$ by $-\psi_0$ and going through the same computation as above, we conclude that $\lim_{y\to 0^-}G'(-\psi_0(y))$ exists and is equal to $\lim_{y\to 0^+}G'(\psi_0(y))$. Hence $G\in C^1(\min \widetilde{\psi_0} ,\max \widetilde{\psi_0}).$ 

Now, we set $N=1$. Again, using the Taylor expansion of $\psi_0$, we obtain
 \begin{equation}
 \begin{aligned}
 G''(\psi_0(y))&=\dfrac{1}{\psi_0'}\bigg(\dfrac{a_1(y^2)+b_1(y)}{c_1(y^2)+d_1(y)}\bigg)'=\dfrac{1}{\psi_0'}\bigg(\dfrac{w'(y^2)+\frac{R'''(y)}{2y}}{F'(y^2)+\frac{R'(y)}{2y}}\bigg)'\\
                          &=\dfrac{1}{\psi_0'}\dfrac{2yw''(y^2)F'(y^2)+F'(y^2)(\frac{R'''}{2y})'+\frac{R'}{2y}2yw''(y^2)+\frac{R'}{2y}(\frac{R'''}{2y})'}{(F'(y^2)+\frac{R'}{2y})^2}\\&\qquad-\dfrac{2yw'(y^2)F''(y^2)+w'(y^2)(\frac{R'}{2y})'+\frac{R'''}{2y}2yF''(y^2)+\frac{R'''}{2y}(\frac{R'}{2y})'}{(F'(y^2)+\frac{R'}{2y})^2}\\
                           &=\dfrac{1}{\psi_0'}\dfrac{2yp(y^2)+q(y)}{r'(y^2)+s(y)}\\
                           &=\dfrac{1}{\big(F'(y^2)+\frac{R'(y)}{2y}\big)}\dfrac{p(y^2)+\frac{q(y)}{2y}}{r'(y^2)+s(y)}=:=\dfrac{a_2(y^2)+b_2(y)}{c_2(y^2)+d_2(y)}.
 \end{aligned}
 \end{equation}

For the sake of induction, we suppose that for $N=2,3,...n-1$ and $\mathcal{U} \in C^{2N+3}$, we have $G \in C^{N+1}$ with
\begin{equation*}
G^{(n)}(\psi_0(y))=\dfrac{1}{2yF'(y^2)+R'(y)}\Bigg(\dfrac{a_{n-1}(y^2)+b_{n-1}(y)}{c_{n-1}(y^2)+d_{n-1}(y)}\Bigg)',
\end{equation*}
where we have arranged terms such that $b_{n-1}$ and $d_{n-1}$ contains all terms consisting of $R$ and its higher derivatives coming from the Taylor expansion of $\psi_0$. With 
this in mind, we can explicitly represent $G^{(n)}$ as follows:
\begin{equation}\label{n-th derivative}
G^{(n)}(\psi_0(y))=\dfrac{a_n(y^2)+b_n(y)}{c_n(y^2)+d_n(y)}
\end{equation}
for some $a_n, b_n, c_n,$ and $d_n$ with $b_n(y), d_n(y) \to 0, a_n(y^2)\to a_n(0), c_n(y)\to c_n(0)$  as $y\to 0.$
 
Now, suppose that $N=n$, we show that for $\mathcal{U} \in C^{2n+3}$ satisfying the hypothesis in the lemma, we have $G\in C^{n+1}$. We begin by writing the $n+1$ derivative of $G(\psi_0)$ for $y>0$. Again, this derivative is well-defined due to the regularity of $\mathcal{U}$. Differentiating \eqref{n-th derivative}, we obtain
 \begin{equation}
 \begin{aligned}
 &G^{(n+1)}(\psi_0(y))\\&=\dfrac{1}{\psi_0'}\dfrac{\bigg(c_n(y^2)+d_n(y)\bigg)\bigg(2ya_n'(y^2)+b_n'(y)\bigg)-\bigg(a_n(y^2)+b_n(y)\bigg)\bigg(2yc_n'(y^2)+d_n'(y)\bigg)}{c_n^2(y^2)+2c_n(y^2)d_n(y)+d_n^2(y)}
 \end{aligned}
 \end{equation}
 By the same regrouping idea, we can write 
 \begin{equation*}
 G^{(n+1)}(\psi_0(y))=\dfrac{a_{n+1}(y^2)+b_{n+1}(y)}{c_{n+1}(y^2)+d_{n+1}(y)},
 \end{equation*}
 for some $a_{n+1}, b_{n+1}, c_{n+1}, d_{n+1}$. One can easily check that both $b_{n+1}(y)$ and $d_{n+1}(y)$ decay as $y\to 0$. Hence, $\lim_{y\to 0}G^{(n+1)}(\psi_0(y))$ exists. Therefore, we have shown that $G \in C^{n+1}$. The proof is complete. \qedhere
\end{proof}

\section*{Acknowledgments}

The authors would like to thank Hui Li for close readings and useful inputs of earlier versions of the manuscript.

\begin{center}
\bibliographystyle{alpha}
\bibliography{Ref.bib}
\end{center}
\end{document}